\documentclass[11pt]{article}

\usepackage{enumerate,xspace}
\usepackage{amsmath,amssymb,wasysym}
\usepackage[all]{xy}
\usepackage{proof}
\usepackage[svgnames]{xcolor}
\usepackage{tikz}
\usepackage{tikz-cd}

\usepackage{stmaryrd} 
\usepackage{mathtools}
\usepackage{latexsym}
\usepackage{dsfont}
\usepackage{multicol}
\usepackage{cmll}
\usepackage{multirow}
\usepackage{longtable}

\usepackage{lscape}
\usepackage{array}

\usepackage{hyperref} 
\hypersetup{
    colorlinks,
    citecolor=red,
    filecolor=red,
    linkcolor=blue,
    urlcolor=red
}

\newtheorem{observation}{Remark}[section]
\newtheorem{lemma}[observation]{Lemma}  
\newtheorem{theorem}[observation]{Theorem}
\newtheorem{definition}[observation]{Definition}
\newtheorem{example}[observation]{Example}

\newtheorem{proposition}[observation]{Proposition} 
\newtheorem{corollary}[observation]{Corollary}

\makeatletter

\newdimen\w@dth

\def\setw@dth#1#2{\setbox\z@\hbox{\scriptsize $#1$}\w@dth=\wd\z@
\setbox\@ne\hbox{\scriptsize $#2$}\ifnum\w@dth<\wd\@ne \w@dth=\wd\@ne \fi
\advance\w@dth by 1.2em}

\def\t@^#1_#2{\allowbreak\def\n@one{#1}\def\n@two{#2}\mathrel
{\setw@dth{#1}{#2}
\mathop{\hbox to \w@dth{\rightarrowfill}}\limits
\ifx\n@one\empty\else ^{\box\z@}\fi
\ifx\n@two\empty\else _{\box\@ne}\fi}}
\def\t@@^#1{\@ifnextchar_ {\t@^{#1}}{\t@^{#1}_{}}}

\def\t@left^#1_#2{\def\n@one{#1}\def\n@two{#2}\mathrel{\setw@dth{#1}{#2}
\mathop{\hbox to \w@dth{\leftarrowfill}}\limits
\ifx\n@one\empty\else ^{\box\z@}\fi
\ifx\n@two\empty\else _{\box\@ne}\fi}}
\def\t@@left^#1{\@ifnextchar_ {\t@left^{#1}}{\t@left^{#1}_{}}}

\def\two@^#1_#2{\def\n@one{#1}\def\n@two{#2}\mathrel{\setw@dth{#1}{#2}
\mathop{\vcenter{\hbox to \w@dth{\rightarrowfill}\kern-1.7ex
                 \hbox to \w@dth{\rightarrowfill}}%
       }\limits
\ifx\n@one\empty\else ^{\box\z@}\fi
\ifx\n@two\empty\else _{\box\@ne}\fi}}
\def\tw@@^#1{\@ifnextchar_ {\two@^{#1}}{\two@^{#1}_{}}}

\def\tofr@^#1_#2{\def\n@one{#1}\def\n@two{#2}\mathrel{\setw@dth{#1}{#2}
\mathop{\vcenter{\hbox to \w@dth{\rightarrowfill}\kern-1.7ex
                 \hbox to \w@dth{\leftarrowfill}}%
       }\limits
\ifx\n@one\empty\else ^{\box\z@}\fi
\ifx\n@two\empty\else _{\box\@ne}\fi}}
\def\t@fr@^#1{\@ifnextchar_ {\tofr@^{#1}}{\tofr@^{#1}_{}}}

\newdimen\W@dth
\def\setW@dth#1#2{\setbox\z@\hbox{$#1$}\W@dth=\wd\z@
\setbox\@ne\hbox{$#2$}\ifnum\W@dth<\wd\@ne \W@dth=\wd\@ne \fi
\advance\W@dth by 1.2em}

\def\T@^#1_#2{\allowbreak\def\N@one{#1}\def\N@two{#2}\mathrel
{\setW@dth{#1}{#2}
\mathop{\hbox to \W@dth{\rightarrowfill}}\limits
\ifx\N@one\empty\else ^{\box\z@}\fi
\ifx\N@two\empty\else _{\box\@ne}\fi}}
\def\T@@^#1{\@ifnextchar_ {\T@^{#1}}{\T@^{#1}_{}}}

\def\T@left^#1_#2{\def\N@one{#1}\def\N@two{#2}\mathrel{\setW@dth{#1}{#2}
\mathop{\hbox to \W@dth{\leftarrowfill}}\limits
\ifx\N@one\empty\else ^{\box\z@}\fi
\ifx\N@two\empty\else _{\box\@ne}\fi}}
\def\T@@left^#1{\@ifnextchar_ {\T@left^{#1}}{\T@left^{#1}_{}}}

\def\Tofr@^#1_#2{\def\N@one{#1}\def\N@two{#2}\mathrel{\setW@dth{#1}{#2}
\mathop{\vcenter{\hbox to \W@dth{\rightarrowfill}\kern-1.7ex
                 \hbox to \W@dth{\leftarrowfill}}%
       }\limits
\ifx\N@one\empty\else ^{\box\z@}\fi
\ifx\N@two\empty\else _{\box\@ne}\fi}}
\def\T@fr@^#1{\@ifnextchar_ {\Tofr@^{#1}}{\Tofr@^{#1}_{}}}

\def\Two@^#1_#2{\def\N@one{#1}\def\N@two{#2}\mathrel{\setW@dth{#1}{#2}
\mathop{\vcenter{\hbox to \W@dth{\rightarrowfill}\kern-1.7ex
                 \hbox to \W@dth{\rightarrowfill}}%
       }\limits
\ifx\N@one\empty\else ^{\box\z@}\fi
\ifx\N@two\empty\else _{\box\@ne}\fi}}
\def\Tw@@^#1{\@ifnextchar_ {\Two@^{#1}}{\Two@^{#1}_{}}}

\def\to{\@ifnextchar^ {\t@@}{\t@@^{}}}
\def\from{\@ifnextchar^ {\t@@left}{\t@@left^{}}}
\def\tofro{\@ifnextchar^ {\t@fr@}{\t@fr@^{}}}
\def\To{\@ifnextchar^ {\T@@}{\T@@^{}}}
\def\From{\@ifnextchar^ {\T@@left}{\T@@left^{}}}
\def\Two{\@ifnextchar^ {\Tw@@}{\Tw@@^{}}}
\def\Tofro{\@ifnextchar^ {\T@fr@}{\T@fr@^{}}}

\makeatother

\title{Exponential Functions in Cartesian Differential Categories}
\author{Jean-Simon Pacaud Lemay}

\begin{document}
\allowdisplaybreaks

\maketitle

\begin{abstract} In this paper, we introduce differential exponential maps in Cartesian differential categories, which generalizes the exponential function $e^x$ from classical differential calculus. A differential exponential map is an endomorphism which is compatible with the differential combinator in such a way that generalizations of $e^0 = 1$, $e^{x+y} = e^x e^y$, and $\frac{\partial e^x}{\partial x} = e^x$ all hold. Every differential exponential map induces a commutative rig, which we call a differential exponential rig, and conversely, every differential exponential rig induces a differential exponential map. In particular, differential exponential maps can be defined without the need of limits, converging power series, or unique solutions of certain differential equations -- which most Cartesian differential categories do not necessarily have. That said, we do explain how every differential exponential map does provide solutions to certain differential equations, and conversely how in the presence of unique solutions, one can derivative a differential exponential map. Examples of differential exponential maps in the Cartesian differential category of real smooth functions include the exponential function, the complex exponential function, the split complex exponential function, and the dual numbers exponential function. As another source of interesting examples, we also study differential exponential maps in the coKleisli category of a differential category. 
\end{abstract}

\subparagraph*{Acknowledgements.} The author would first like to thank Geoff Cruttwell and Robin Cockett for their support in this project, their editorial comments, and very useful discussions which greatly helped the development of this paper. The author would also like to thank the anonymous referee for their review and comments which helped improve this paper. The author would also like to thank the following for financial support regarding this paper: Kellogg College, the Department of Computer Science of the University of Oxford, the Clarendon Fund, the Oxford-Google DeepMind Graduate Scholarship, and the Oxford Travel Abroad Bursary. 
\newpage
\tableofcontents

\section{Introduction}

Cartesian differential categories \cite{blute2009cartesian}, introduced by Blute, Cockett, and Seely, come equipped with a differential combinator $\mathsf{D}$ which provides a categorical axiomatization of the differential from multivariable differential calculus. Important examples of Cartesian differential categories include the category of real smooth functions (Example \ref{smoothex}), the coKleisli category of a differential category \cite{blute2006differential}, the differential objects of a tangent category \cite{cockett2014differential}, and categorical models of Ehrhard and Regnier's differential $\lambda$-calculus \cite{ehrhard2003differential} (which are in fact called Cartesian \emph{closed} differential categories \cite{manzonetto2012categorical}). Other interesting (and surprising) examples include abelian functor calculus \cite{bauer2018directional} and cofree Cartesian differential categories \cite{cockett2011faa,lemay2018tangent}. Since their introduction, Cartesian differential categories have a rich literature and have been successful in generalizing many concepts from classical differential calculus. More recently, Cartesian differential categories have also started to find their way in applications. 

In particular, Cockett and Cruttwell have introduced the notion of dynamical systems and their solutions in tangent categories \cite{cockett2017connections}, which generalize ordinary differential equations in this context, specifically initial value problems. Since every Cartesian differential category is a tangent category, this implies that dynamical systems allow one to study differential equations in a Cartesian differential category. In classical differential calculus, one of the most important tools used for solving differential equations is the exponential function $e^x$. Therefore, it is desirable to generalize the exponential function for Cartesian differential categories. 

The exponential function $e^x$ admits numerous equivalent characterization. It can either be defined as the inverse of the natural logarithm function $ln(x)$, or as the limit:
\[ e^x = \lim \limits_{n \to \infty} (1 + \frac{x}{n})^n\]
or as the convergent power series:
\[e^x = \sum \limits^{\infty}_{n =0} \frac{x^n}{n!}\]
or even as the solution to $f^\prime(x) = f(x)$ with initial condition $f(0) = 1$. However in arbitrary Cartesian differential categories, functions need to be defined at zero (which excludes $ln(x)$) and one does not necessarily have a notion of convergence, infinite sums, or even (unique) solutions to initial value problems. Therefore one must look for a more algebraic characterization of the exponential function. In classical algebra, an exponential ring \cite{van1984exponential} is a ring equipped with an endomorphism $e$ which is a monoid morphism from the additive structure to the multiplicative structure, that is: 
\begin{align*}
e(x+y) = e(x)e(y) && e(0) = 1
\end{align*}
The canonical example of an exponential ring is the field of real numbers $\mathbb{R}$ with the exponential function $e^x$. While this seems promising, arbitrary objects in a Cartesian differential category do not necessarily come equipped with a multiplication. Rather than requiring this extra ring structure on objects, it turns out that the differential combinator $\mathsf{D}$ will allow us to bypass the need for a multiplication. 

In the category of real smooth functions, which is the canonical example of a Cartesian differential category, the differential combinator $\mathsf{D}$ applied to the exponential function $e^x$ is the smooth function $\mathbb{R} \times \mathbb{R} \xrightarrow{\mathsf{D}[e^x]} \mathbb{R}$ defined as: 
\[\mathsf{D}[e^x](x,y) = e^x y\]
Thus the multiplication of $\mathbb{R}$ appears in $\mathsf{D}[e^x]$. Inspired by this observation, the generalization of the exponential function in a Cartesian differential category can be defined simply in terms of an endomorphism $A \xrightarrow{e} A$ which is compatible with the differential combinator $\mathsf{D}$ in the sense that 
\begin{align*}
\mathsf{D}[e](0,x) = x && e(x+y)=\mathsf{D}[e](x, e(y))
\end{align*}
We call such endomorphisms \textbf{differential exponential maps}, which is the main novel notion of study in this paper. Differential exponential maps generalize the exponential function for Cartesian differential categories. Indeed for $e^x$, the differential exponential maps axioms correspond precisely to $e^0 x = x$ and $e^x e^y =e^{x+y}$ respectively. 

Previously, we mentioned that not every object in a Cartesian differential category has a multiplication. However, it turns out that every differential exponential map $A \xrightarrow{e} A$ does induce a commutative rig structure on $A$, and thus $A$ does come equipped with a multiplication. The construction is once again inspired by the classical exponential function $e^x$. Applying the differential combinator on $e^x$ twice we obtain:
\[\mathsf{D}^2[e^x]((x,y), (z,w))=e^{x}yz + e^xw\]
Setting $x=0$ and $w=0$, one re-obtains precisely the multiplication of $\mathbb{R}$:
\[\mathsf{D}^2[e^x]((0,y), (z,0))=yz\]
The unit for the multiplication is, of course, obtain by evaluating $e^x$ at $0$, $e^0 = 1$. This construction is easily generalized to an arbitrary Cartesian differential category and one can show that every differential exponential map induces a commutative rig. Commutativity of the multiplication follows from the symmetry of the partial derivatives axiom \textbf{[CD.7]} of the differential combinator. The unit identities for the multiplication will follow from both differential exponential map axioms. Proving associativity is a bit trickier but essentially follows from the observation that: 
\[ D^3[e^x]((0,x),(y,0), (z,0),(0,0)) = xyz \]
and then using both the differential combinator axioms and differential exponential map axioms, one shows that both sides of the associativity law are equal to the third-order partial derivative. 

Conversely, it is possible to obtain differential exponential maps from special kinds of commutative rigs. Indeed, one can alternatively axiomatize an object equipped with a differential exponential map instead as a commutative rig equipped with an endomorphism which satisfies the three fundamental properties of the exponential function that $e^0 = 1$, $e^{x+y} = e^x e^y$, and $\frac{\partial e^x}{\partial x} = e^x$. We call such rigs: \textbf{differential exponential rigs}. One of the main results of this paper it that there is a bijective correspondence between differential exponential maps and differential exponential rigs. In the category of real smooth functions, interesting examples of differential exponential rigs include $\mathbb{R}$ with the exponential function $e^x$, $\mathbb{R}^2$ equipped with the complex numbers multiplication and the complex exponential function $e^{x+iy} = e^x\cos(y) + i e^x\sin(y)$, $\mathbb{R}^2$ equipped with the split complex numbers multiplication and the split complex exponential function $e^{x+jy} = e^x\cosh(y) + j e^x\sinh(y)$ \cite{olariu2002complex}, and also $\mathbb{R}^2$ equipped with the dual numbers multiplication and the dual numbers exponential function $e^{x+y\varepsilon} = e^x + e^x y \varepsilon$ \cite{rosenfeld2013geometry}. 

As one of the main motivations for their development, differential exponentials maps does allow one to solve a certain class of linear dynamical systems in any Cartesian differential category. Specifically, one can solve the dynamical systems which generalize the initial value problems of the form $f^\prime(x) = f(x)a$ with initial condition $f(0) = b$ (for some constants $a$ and $b$), whose classical solution is $f(x) = e^{ax} b$. The types of differential equations are of particular interest in control systems theory \cite[Chapter 5]{astrom2010feedback}. Furthermore, it turns out that a differential exponential map is indeed the a solution to the dynamical system which generalizes the initial value problem $f^\prime(x) = f(x)$ with initial condition $f(0) = 1$. In future work on solving differential equations in a Cartesian differential category, differential exponential maps will hopefully play a key role. Such an application can be found in \cite[Section 5]{cockett2019differential}, where it is shown that in a tangent category, a differential curve object admits a canonical differential exponential map, which induces solutions to many interesting dynamical systems including one which in turn induces an action on differential bundles. 

An important example of a Cartesian differential category is the coKleisli category of the comonad $\oc$ of a differential category (with finite products). In this setting, a differential exponential map is an endomorphism in the coKleisli category and therefore a map of type $\oc A \xrightarrow{e} A$ in the base category. In a differential storage category, that is, when $\oc$ has the Seely isomorphisms, the differential structure is captured by the codereliction map $A \xrightarrow{\eta} \oc A$ and it turns out that a commutative rig in the coKleisli category is precisely a commutative monoid (over the tensor product) in the base category. As such, a differential exponential map in the coKleisli category can alternatively be given by a commutative monoid $A$ in the base category equipped with a monoid morphism $\oc A \xrightarrow{e} A$ which is a retract of the codereliction. We call such commutative monoids \textbf{$\oc$-differential exponential algebras}, and there is a bijective correspondence between $\oc$-differential exponential algebras and differential exponential maps in the coKleisli category. Interesting examples of differential storage categories include both the category of sets and relations and the category of vector spaces (over a field of characteristic $0$), where the comonad $\oc$ is given by the free exponential modality. In both of these examples, $\oc$-differential exponential algebras correspond precisely to commutative monoids. 

\textbf{Outline and Main Results:} Section \ref{CDCsec} is a background section on Cartesian differential categories where we briefly review the basic definitions, as well as to introduce the notation and terminology used in this paper. In particular, we review the canonical commutative monoid structure $\oplus$ on every object in a Cartesian differential category (Lemma \ref{opluslem}), which plays a key role throughout this paper. In Section \ref{DEMsec} we introduce differential exponential maps (Definition \ref{DEMdef}) and in particular show that the category of differential exponential maps is a Cartesian tangent category (Proposition \ref{DEMtangent}). In Example \ref{expex} we provide examples of differential exponential maps in the category of smooth real functions, which include the classical exponential function, the complex exponential function, the split complex exponential function, and the dual number exponential function. In Section \ref{dessec} we introduce differential exponential rigs. We show that every differential exponential rig induces a differential exponential map (Proposition \ref{prop1}) and conversely that every differential exponential map induces a differential exponential rig (Proposition \ref{prop2}), and show that these constructions are inverses of each other (Theorem \ref{isothm}). As an immediate consequence, the category of differential exponential rigs is also a Cartesian tangent category (Proposition \ref{DEStan}). In Section \ref{DYNsec} we explain the relationship between differential exponential maps and solutions to differential equations in arbitrary Cartesian differential categories. In particular, we show that every differential exponential map is a solution to the expected differential equation (Proposition \ref{esolprop}) and that a certain class of dynamical systems admit a solution (Proposition \ref{comlinprop}). We also show that in the presence of unique solutions to differential equations, one can build a differential exponential map (Proposition \ref{uniqueprop}). In Section \ref{diffexpsec} we study differential exponential maps in the coKleisli category of a differential (storage) category and give equivalent characterizations in these cases (Proposition \ref{coKexp1} and Proposition \ref{coKexp2}). We also introduce $\oc$-differential exponential algebras (Definition \ref{!deadef}) for differential storage categories. We show that every $\oc$-differential exponential algebra induces a differential exponential map in the coKleisli category (Proposition \ref{propcok1}), and conversely that every differential exponential map in the coKleisli category induces a $\oc$-differential exponential algebra (Proposition \ref{propcok1}), and that these constructions are inverses of each other (Theorem \ref{isothm2}). We conclude this paper with Section \ref{consec} which discusses some interesting potential future work to do with differential exponential maps. 

\textbf{Conventions:} We use diagrammatic order notation for composition: this means that the composite map $fg$ is the map which first does $f$ then $g$. 

\section{Background: Cartesian Differential Categories}\label{CDCsec}

In this section, if only to introduce notation, we briefly review Cartesian differential categories, their underlying Cartesian left additive structure, and their induced Cartesian tangent category structure. That said, we assume that the reader is familiar with the theory of Cartesian differential categories. For a more in-depth discussion we refer the reader to \cite{blute2009cartesian,cockett2014differential}. 

For a category with finite products, we denote the binary product of objects $A$ and $B$ by $A \times B$ with projection maps $A \times B \xrightarrow{\pi_0} A$ and $A \times B \xrightarrow{\pi_1} B$, pairing operation $\langle -, - \rangle$ and thus $f \times g = \langle \pi_0 f, \pi_1 g \rangle$, and wed denote the chosen terminal object as $\top$. Also, an important map which will use throughout this paper is the canonical interchange map $(A \times B) \times (C \times D) \xrightarrow{c} (A \times C) \times (B \times D)$ defined as follows: 
  \begin{equation}\label{cdef}\begin{gathered} c := \langle \pi_0 \times \pi_0, \pi_1 \times \pi_1 \rangle
  \end{gathered}\end{equation}
which is a natural isomorphism such that $cc=1$. 

\begin{definition}\label{CLACdef} A \textbf{left additive category} \cite[Definition 1.1.1]{blute2009cartesian} is a category $\mathbb{X}$ such that each hom-set $\mathbb{X}(A,B)$ is a commutative monoid with addition $\mathbb{X}(A,B) \times \mathbb{X}(A,B) \xrightarrow{+} \mathbb{X}(A,B)$, $(f,g) \mapsto f +g$, and zero $0 \in \mathbb{X}(A,B)$, such that composition on the {\em left} preserves the additive structure, that is:
  \begin{equation}\label{add}\begin{gathered} f(g+h)=fg+fh \quad \quad \quad f0=0
  \end{gathered}\end{equation}
A map $k$ in a left additive category is \textbf{additive} \cite[Definition 1.1.1]{blute2009cartesian} if composition on the right by $h$ preserves the additive structure, that is:
  \begin{equation}\label{addmap}\begin{gathered} (g+h)k =gk+hk \quad \quad \quad 0k=0
  \end{gathered}\end{equation}
A \textbf{Cartesian left additive category} \cite[Definition 1.2.1]{blute2009cartesian} is a left additive category $\mathbb{X}$ with finite products $\times$ and terminal object $\top$ such that all projection maps $A \times B \xrightarrow{\pi_0} A$ and $A \times B \xrightarrow{\pi_1} B$ are additive. 
\end{definition}

We note that the definition of a Cartesian left additive category presented here is not precisely that given in \cite[Definition 1.2.1]{blute2009cartesian}, but was shown to be equivalent in \cite[Lemma 2.4]{lemay2018tangent}. Also note that in a Cartesian left additive category, the unique map from an object $A$ to the terminal object $\top$ is in fact the zero map $A \xrightarrow{0} \top$. 

\begin{definition}\label{cartdiffdef} A \textbf{Cartesian differential category} \cite[Definition 2.1.1]{blute2009cartesian} is a Cartesian left additive category $\mathbb{X}$ equipped with a \textbf{differential combinator} $\mathsf{D}$, which is a family of operators ${\mathbb{X}(A,B) \xrightarrow{\mathsf{D}} \mathbb{X}(A \times A,B)}$, $f \mapsto \mathsf{D}[f]$, such that the following axioms hold:  
\begin{enumerate}[{\bf [CD.1]}]
\item $\mathsf{D}[f+g] = \mathsf{D}[f] + \mathsf{D}[g]$ and $\mathsf{D}[0]=0$;
\item $\left (1 \times (\pi_0 + \pi_1) \right) \mathsf{D}[f] = (1 \times \pi_0)\mathsf{D}[f] + (1 \times \pi_1)\mathsf{D}[f]$ and $\langle 1, 0 \rangle \mathsf{D}[f]=0$;
\item $\mathsf{D}[1]=\pi_1$, $\mathsf{D}[\pi_0] = \pi_1\pi_0$ and $\mathsf{D}[\pi_0] = \pi_1\pi_1$;
\item $\mathsf{D}[\langle f, g \rangle] = \langle \mathsf{D}[f] , \mathsf{D}[g] \rangle$; 
\item $\mathsf{D}[fg] = \langle \pi_0 f, \mathsf{D}[f] \rangle \mathsf{D}[g]$; 
\item $\left( \langle 1,0 \rangle \times \langle 0,1 \rangle \right) \mathsf{D}^2[f] = \mathsf{D}[f]$; 
\item $c \mathsf{D}^2[f] = \mathsf{D}^2[f]$. 
\end{enumerate}
\end{definition}

It is important to note that unlike in \cite{blute2009cartesian,cockett2014differential,cockett2016differential}, we use the convention used in the more recent work on Cartesian differential categories where the linear argument of $\mathsf{D}[f]$ is its second argument rather than its first argument. We also note that the definition of a Cartesian differential category given here is not precisely that given in \cite[Definition 2.1.1]{blute2009cartesian} but rather an equivalent version found in \cite[Section 3.4]{cockett2016differential}. A discussion on the intuition for the differential combinator axioms can be found in \cite[Remark 2.1.3]{blute2009cartesian}. The canonical example of a Cartesian differential category is the category of Euclidean spaces and smooth maps between them -- which will also be our main example throughout this paper. 

\begin{example}\label{smoothex} \normalfont Let $\mathbb{R}$ be the set of real numbers and let $\mathsf{SMOOTH}$ be the category whose objects are the Euclidean vector spaces $\mathbb{R}^n$ (including the singleton $\mathbb{R}^0 = \lbrace \ast \rbrace$) and whose maps are smooth function ${\mathbb{R}^n \xrightarrow{F} \mathbb{R}^m}$, which of course are in fact tuples $F = \langle f_1, \hdots, f_m \rangle$ for some smooth functions ${\mathbb{R}^n \xrightarrow{f_i} \mathbb{R}}$. $\mathsf{SMOOTH}$ is a Cartesian differential category where the finite product structure and additive structure are defined in the obvious way, and whose differential combinator is given by the standard derivative of smooth functions. Explicitly, recall that for a smooth map $\mathbb{R}^n \xrightarrow{f} \mathbb{R}$, its gradient $\mathbb{R}^n \xrightarrow{\nabla(f)} \mathbb{R}^n$ is defined as:
\begin{align*}
\nabla(f)(\vec a) = \left( \frac{\partial f(\vec x)}{\partial x_1} (\vec a), \hdots, \frac{\partial f(\vec x)}{\partial x_n} (\vec a) \right)
\end{align*}
Then its derivative $\mathbb{R}^n \times \mathbb{R}^n \xrightarrow{\mathsf{D}[f]} \mathbb{R}$ is defined as: 
\begin{align*}
\mathsf{D}[f](\vec a, \vec b) = \nabla(f)(\vec a) \cdot \vec b = \sum \limits^n_{i=1} \frac{\partial f(\vec x)}{\partial x_i} (\vec a)b_i 
\end{align*}
where $\cdot$ is the dot product of vectors. In the case of $\mathbb{R}^n \xrightarrow{F} \mathbb{R}^m$, which is in fact a tuple of smooth functions $F = \langle f_1, \hdots, f_m \rangle$, we define $\mathbb{R}^n \times \mathbb{R}^n \xrightarrow{\mathsf{D}[F]} \mathbb{R}^m$ as $\mathsf{D}[F] = \langle \mathsf{D}[f_1] \hdots, \mathsf{D}[f_n] \rangle$. It is also possible to define $\mathsf{D}[F]$ in terms of the Jacobian matrix of $F$. 
\end{example}

Many other interesting examples of Cartesian differential categories can be found throughout the literature such as categorical models of the differential $\lambda$-calculus \cite{ehrhard2003differential}, which are called Cartesian \emph{closed} differential categories \cite{manzonetto2012categorical}, the subcategory of differential objects of a Cartesian tangent category \cite{cockett2014differential}, and the coKleisli category of a differential category \cite{blute2009cartesian,blute2006differential}. We will take a closer look at the coKleisli category of a differential category in Section \ref{diffexpsec}. 

An important class of maps in a Cartesian differential category is the class of linear maps. Later in Section \ref{dessec}, we will also discuss bilinear maps. 

\begin{definition}\label{lindef} In a Cartesian differential category, a map $f$ is said to be \textbf{linear} \cite[Definition 2.2.1]{blute2009cartesian} if $\mathsf{D}[f]= \pi_1 f$. 
\end{definition}

\begin{example}\label{smoothlinex} \normalfont In $\mathsf{SMOOTH}$, a map $\mathbb{R}^n \xrightarrow{F} \mathbb{R}^m$ is linear in the Cartesian differential sense precisely when it is $\mathbb{R}$-linear in the classical sense, that is, $F(s \vec a + t \vec b) = sF(\vec a) + t F(\vec b)$ for all $s,t \in \mathbb{R}$ and $\vec a, \vec b \in \mathbb{R}^n$. \end{example}

Here are now some useful properties about linear maps for this paper: 

\begin{lemma}\label{linlem} \cite[Lemma 2.2.2]{blute2009cartesian} In a Cartesian differential category:
\begin{enumerate}[{\em (i)}]
\item If $f$ is linear then $f$ is additive;
\item Identity maps $A \xrightarrow{1} A$ and projection maps $A \times B \xrightarrow{\pi_0} A$ and $A \times B \xrightarrow{\pi_1} B$ are linear;
\item If $A \xrightarrow{f} B$ and $B \xrightarrow{g} C$ are linear then their composition $A \xrightarrow{fg} C$ is linear;
\item If $A \xrightarrow{f} B$ is an isomorphism and is linear then its inverse $B \xrightarrow{f^{-1}} A$ is linear; 
\item If $C \xrightarrow{f} A$ and $C \xrightarrow{g} B$ are linear then their pairing $C \xrightarrow{\langle f, g \rangle} A \times B$ is linear;
\item If $A \xrightarrow{f} B$ and $C \xrightarrow{g} D$ are linear then their product $A \times C \xrightarrow{f \times g} B \times D$ is linear;
\item Zero maps $A \xrightarrow{0} B$ are linear;
\item If $A \xrightarrow{f} B$ and $A \xrightarrow{g} B$ are linear then their sum $A \xrightarrow{f+g} B$ is linear;
\item If $A \xrightarrow{f} B$ is linear then $\mathsf{D}[fg] = (f \times f) \mathsf{D}[g]$ and $\mathsf{D}[kf] = \mathsf{D}[k]f$;
\item The interchange map $(A \times B) \times (C \times D) \xrightarrow{c} (A \times C) \times (B \times D)$ is linear. 
\end{enumerate}
\end{lemma}

For a Cartesian differential category $\mathbb{X}$, define its subcategory of linear maps $\mathsf{LIN}[\mathbb{X}]$ to be the category whose objects are the same as $\mathbb{X}$ and whose maps are linear in $\mathbb{X}$. Lemma \ref{linlem} tells us that $\mathsf{LIN}[\mathbb{X}]$ is a well-defined category and also that it has finite biproducts, and thus is a Cartesian left additive category where every map is additive. $\mathsf{LIN}[\mathbb{X}]$ also inherits the differential combinator from $\mathbb{X}$ and so $\mathsf{LIN}[\mathbb{X}]$ is a Cartesian differential category where every map is linear. Therefore the obvious forgetful functor $\mathsf{LIN}[\mathbb{X}] \xrightarrow{\mathsf{U}} \mathbb{X}$ preserves the Cartesian differential structure strictly. 
 
The differential combinator of a Cartesian differential category induces an endofunctor and this endofunctor makes a Cartesian differential category a Cartesian tangent category. We will not review the full definition of a tangent category here but we will highlight certain properties that will be important for this paper. We invite the reader to read the full definition of a tangent category in \cite{cockett2014differential,cockett2016differential}. 

\begin{proposition}\label{tangentfunctorprop} {\cite[Proposition 4.7]{cockett2014differential}} Every Cartesian differential category $\mathbb{X}$ is a Cartesian tangent category where the \textbf{tangent functor} $\mathsf{T}: \mathbb{X} \to \mathbb{X}$ is defined on objects as $\mathsf{T}(A) := A \times A$ and on morphisms as $\mathsf{T}(f) := \langle \pi_0 f, \mathsf{D}[f] \rangle$. 
\end{proposition} 

\begin{corollary}\label{lintan} For a Cartesian differential category $\mathbb{X}$, $\mathsf{LIN}[\mathbb{X}]$ is a Cartesian tangent category where the tangent functor $\mathsf{T}: \mathsf{LIN}[\mathbb{X}] \to \mathsf{LIN}[\mathbb{X}]$ is defined on objects as $\mathsf{T}(A) := A \times A$ and on morphisms as $\mathsf{T}(f) := f \times f$. Furthermore, the forgetful functor $\mathsf{LIN}[\mathbb{X}] \xrightarrow{\mathsf{U}} \mathbb{X}$ preserves the Cartesian tangent structure strictly.
\end{corollary}

Here are now some useful properties involving the tangent functor (which we leave to the reader to check for themselves): 

\begin{lemma}\label{tanlem} In a Cartesian differential category:
\begin{enumerate}[{\em (i)}]
\item $\mathsf{D}[fg] = \mathsf{T}(f) \mathsf{D}[g]$;
\item $\langle 1, 0 \rangle \mathsf{T}(f) = f \langle 1, 0 \rangle$; 
\item If $f$ is linear then $\mathsf{T}(f) = f \times f$; 
\item $\mathsf{T}(\langle f,g \rangle) = \langle \mathsf{T}(f), \mathsf{T}(g) \rangle c$; 
\item $\mathsf{D}[f \times g] = c\left(\mathsf{D}[f] \times \mathsf{D}[g] \right)$ and $\mathsf{T}(f \times g) c = c (\mathsf{T}(f) \times \mathsf{T}(g))$;
\item $\mathsf{T}(f+g) = \mathsf{T}(f) + \mathsf{T}(g)$ and $\mathsf{T}(0) = 0$;
\item $\mathsf{D}\left[\mathsf{T}(f) \right] = c \mathsf{T}(\mathsf{D}[f])$. 
\end{enumerate}
\end{lemma}

We conclude this section with the observation that in a Cartesian differential category, the additive structure induces a canonical commutative monoid structure on every object. Cartesian left additive categories can be axiomatized in terms of equipping each object with a commutative monoid structure such that the projection maps are monoid morphisms \cite[Proposition 1.2.2]{blute2009cartesian}. 

In a Cartesian differential category, a commutative monoid is a triple $(A, \oast, i)$ consisting of an object $A$, map $A \times A \xrightarrow{\oast} A$, and a point $\top \xrightarrow{i} A$ such that the following diagrams commute
  \begin{equation}\label{ostarmonoideq}\begin{gathered} \xymatrixcolsep{2.5pc}\xymatrixrowsep{1.5pc}\xymatrix{
  A \ar[d]_-{\langle 0i, 1 \rangle} \ar@{=}[dr]^-{} \ar[r]^-{\langle 1, 0i  \rangle} & A \times A \ar[d]^-{\oast} & (A \times A) \times A \ar[rr]^-{\oast \times 1} \ar[d]_-{\cong} & & A \times A \ar[dd]^-{\oast}  \\ 
  A \times A \ar[r]_-{\oast} & A & A \times (A \times A) \ar[d]_-{1 \times \oast} \\
   & & A \times A \ar[rr]_-{\oast} && A } \\    
   \xymatrixcolsep{3pc}\xymatrixrowsep{1.5pc}\xymatrix{ A \times A \ar[dr]_-{\oast} \ar[r]^-{\langle \pi_1, \pi_0 \rangle} & A \times A \ar[d]^-{\oast} \\ 
& A}
   \end{gathered}\end{equation} 
   
\begin{lemma}\label{opluslem} In a Cartesian differential category, for an object $A$ define the map $A \times A \xrightarrow{\oplus} A$ as $\oplus = \pi_0 + \pi_1$. Then the triple $(A, \oplus, 0)$ is a commutative monoid and furthermore: 
\begin{enumerate}[{\em (i)}]
\item $\oplus$ is linear, that is, $\mathsf{D}[\oplus] = \pi_1 \oplus$;
\item $\mathsf{T}(\oplus) = \oplus \times \oplus$;
\item $c \mathsf{T}(\oplus) = \oplus$; 
\item A map $f$ is additive if and only if $\oplus f = (f \times f) \oplus$ and $0f = f$. 
\end{enumerate}
\end{lemma}

\section{Differential Exponential Maps}\label{DEMsec}

In this section, we introduce \emph{differential exponential maps}, which generalizes the notion of the classical exponential function $e^x$ for arbitrary Cartesian differential categories. 

\begin{definition}\label{DEMdef} A \textbf{differential exponential map} in a Cartesian differential category is a map $A \xrightarrow{e} A$, such that the following diagrams commute:
 \begin{equation}\label{dem}\begin{gathered} \xymatrixcolsep{5pc}\xymatrix{ A \ar@{=}[dr]_-{} \ar[r]^-{\langle 0, 1 \rangle} & A \times A \ar[d]^-{\mathsf{D}[e]} & A \times A \ar[r]^-{1 \times e} \ar[d]_-{\oplus} & A \times A \ar[d]^-{\mathsf{D}[e]} \\
 & A & A\ar[r]_-{e} & A } \end{gathered}\end{equation}
 where $\oplus$ is defined as in Lemma \ref{opluslem}. 
 \end{definition}

The intuition for a differential exponential map is best explained in Example \ref{expex}.i, which shows that the classical exponential function $e^x$ (which is, of course, the main motivating example) is a differential exponential map. Briefly, since $e^x$ is its own derivative, applying the differential combinator on $e^x$ results in the smooth function $\mathbb{R} \times \mathbb{R} \xrightarrow{\mathsf{D}[e^x]} \mathbb{R}$ defined as: 
\[\mathsf{D}[e^x](x,y)=e^xy\]
The left diagram of (\ref{dem}) generalizes that $e^0 y = y$ (since $e^0 =1$), while the right diagram generalizes that $e^{x}e^y=e^{x+y}$. The differential combinator is the key piece that allows one to bypass the need for a multiplication operation and a multiplicative unit in the definition of a differential exponential map. That said, in Section \ref{dessec} we will see that every differential exponential map does induce a multiplication and that analogues of the three essential properties of the classical exponential function are satisfied (Proposition \ref{prop2}). And conversely, we will also see how one can also axiomatize differential exponential maps in terms of rig structure and analogues of the three essential properties of the classical exponential function (Proposition \ref{prop1}). And, as mentioned in the introduction, we also highlight that the definition of a differential exponential map does not require any added structure or property on the Cartesian differential category such as a notion of converging limits or infinite sums. Before giving examples of differential exponential maps, which can be found in Example \ref{expex}, let us first consider the category of differential exponential maps and constructions of differential exponential maps. 

For a Cartesian differential category $\mathbb{X}$, define its category of differential exponential maps as the category $\mathsf{DEM}[\mathbb{X}]$ whose objects are pairs $(A, e)$ consisting of an object $A \in \mathbb{X}$ and a differential exponential map $A \xrightarrow{e} A$, and where a map $(A,e) \xrightarrow{f} (B,e^\prime)$ is a \emph{linear} map $A \xrightarrow{f} B$ in $\mathbb{X}$ such that the following diagram commutes: 
 \begin{equation}\label{demmap}\begin{gathered} \xymatrixcolsep{5pc}\xymatrixrowsep{1.5pc}\xymatrix{ A \ar[d]_-{e} \ar[r]^-{f} & B \ar[d]^-{e^\prime} \\
A \ar[r]_-{f} & B } \end{gathered}\end{equation}
and where composition and identity maps are as in $\mathbb{X}$. The reason for why maps in $\mathsf{DEM}[\mathbb{X}]$ are restricted to being linear will become apparent in the proof of Theorem \ref{isothm}. There is the obvious forgetful functor $\mathsf{DEM}[\mathbb{X}] \xrightarrow{\mathsf{U}} \mathsf{LIN}[\mathbb{X}]$ which maps $\mathsf{U}(A,e) = A$ and $\mathsf{U}(f) = f$.  

\begin{lemma}\label{reflectlem}The forgetful functor $\mathsf{DEM}[\mathbb{X}] \xrightarrow{\mathsf{U}} \mathsf{LIN}[\mathbb{X}]$ creates all limits. 
\end{lemma}
\begin{proof} Let $\mathbb{D} \xrightarrow{\mathsf{F}} \mathsf{DEM}[\mathbb{X}]$ be a functor such that the limit of the composite $\mathbb{D} \xrightarrow{\mathsf{F}} \mathsf{DEM}[\mathbb{X}] \xrightarrow{\mathsf{U}} \mathsf{LIN}[\mathbb{X}]$ exists in $ \mathsf{LIN}[\mathbb{X}]$ which we denote 
$\lim \limits_{X \in \mathbb{D}} \mathsf{U}\left(\mathsf{F}(X) \right)$ with projections $\lim \limits_{X \in \mathbb{D}} \mathsf{U}\left(\mathsf{F}(X) \right) \xrightarrow{\pi_x} \mathsf{U}\left(\mathsf{F}(X) \right)$. Note that $\mathsf{U}\left(\mathsf{F}(X) \right)$ is the underlying object of $\mathsf{F}(X)$, and so it comes equipped with a differential exponential map $\mathsf{U}\left(\mathsf{F}(X) \right) \xrightarrow{e_X} \mathsf{U}\left(\mathsf{F}(X) \right)$. Therefore $\mathsf{F}(X) = \left(\mathsf{U}\left(\mathsf{F}(X) \right), e_X \right)$. Now for every map $X \xrightarrow{f} Y$ in $\mathbb{D}$, since $\mathsf{F}(f)$ is a map in $\mathsf{DEM}[\mathbb{X}]$, the following diagram commutes:
\[\xymatrixcolsep{5pc}\xymatrixrowsep{1.5pc}\xymatrix{ & \lim \limits_{X \in \mathbb{D}} \mathsf{U}\left(\mathsf{F}(X) \right) \ar[dr]^-{\pi_Y} \ar[dl]_-{\pi_X} \\  
\mathsf{U}\left(\mathsf{F}(X) \right) \ar[d]_-{e_X} \ar[rr]^-{\mathsf{U}\left(\mathsf{F}(f) \right)} & & \mathsf{U}\left(\mathsf{F}(Y) \right) \ar[d]^-{e_Y} \\ 
\mathsf{U}\left(\mathsf{F}(X) \right) \ar[rr]_-{\mathsf{U}\left(\mathsf{F}(f) \right)} & & \mathsf{U}\left(\mathsf{F}(Y) \right)}\]
By the universal property of $\lim \limits_{X \in \mathbb{D}} \mathsf{U}\left(\mathsf{F}(X) \right)$, there is a unique map:
 \begin{equation}\label{pixeq0}\begin{gathered} {\lim \limits_{X \in \mathbb{D}} \mathsf{U}\left(\mathsf{F}(X) \right) \xrightarrow{\lim \limits_{X \in \mathbb{D}} e_X} \lim \limits_{X \in \mathbb{D}} \mathsf{U}\left(\mathsf{F}(X) \right)} \end{gathered}\end{equation}
 which makes the following diagram commute: 
\[ \xymatrixcolsep{5pc}\xymatrixrowsep{1.5pc}\xymatrix{ \lim \limits_{X \in \mathbb{D}} \mathsf{U}\left(\mathsf{F}(X) \right) \ar@{-->}[dd]_-{\exists ! ~ \lim \limits_{X \in \mathbb{D}} e_X} \ar[r]^-{\pi_X} & \mathsf{U}\left(\mathsf{F}(X) \right) \ar[dd]^-{e_X} \\ \\
\lim \limits_{X \in \mathbb{D}} \mathsf{U}\left(\mathsf{F}(X) \right) \ar[r]_-{\pi_X} & \mathsf{U}\left(\mathsf{F}(X) \right) } \]
We wish to show that $\lim \limits_{X \in \mathbb{D}} e_X$ is a differential exponential map. To do so, first note that for each $X \in \mathbb{D}$, $\pi_X$ is linear and so by Lemma \ref{linlem} it follows that: 
 \begin{equation}\label{pixeq1}\begin{gathered} \mathsf{D}\left[\lim \limits_{X \in \mathbb{D}} e_X \right] \pi_X = (\pi_X \times \pi_X) \mathsf{D}\left[e_X \right]
 \end{gathered}\end{equation}
Now since $\pi_X$ is linear, it is also additive (Lemma \ref{linlem}) and therefore we obtain the following:
 \begin{align*}
\langle 0,1 \rangle \mathsf{D}\left[\lim \limits_{X \in \mathbb{D}} e_X \right] \pi_X &=~\langle 0,1 \rangle (\pi_X \times \pi_X) \mathsf{D}\left[e_X \right] \tag{\ref{pixeq1}}\\
&=~ \left \langle 0 \pi_X, \pi_X \right \rangle \mathsf{D}\left[e_X \right] \\
&=~ \langle 0, \pi_X \rangle \mathsf{D}\left[e_X \right] \tag{$\pi_X$ is additive} \\
&=~ \pi_X \langle 0, 1 \rangle \mathsf{D}\left[e_X \right] \\
&=~\pi_X \tag{\ref{dem}} \\\\
\oplus (\lim \limits_{X \in \mathbb{D}} e_X) \pi_X &=~ \oplus \pi_X e_X \tag{\ref{pixeq0}} \\
&=~ (\pi_X \times \pi_X) \oplus e_X \tag{$\pi_X$ is additive + Lemma \ref{opluslem}} \\
&=~ (\pi_X \times \pi_X) (1 \times e_X) \mathsf{D}[e_X] \tag{\ref{dem}} \\
&=~ \left(1 \times \lim \limits_{X \in \mathbb{D}} e_X \right) (\pi_X \times \pi_X) \mathsf{D}\left[e_X \right] \tag{\ref{pixeq0}} \\
&=~\left(1 \times \lim \limits_{X \in \mathbb{D}} e_X \right) \mathsf{D}\left[\lim \limits_{X \in \mathbb{D}} e_X \right] \pi_X \tag{\ref{pixeq1}}
\end{align*}
Therefore for each $X \in \mathbb{D}$, we have that: 
 \begin{align*}\left \langle 0,1 \right\rangle \mathsf{D}\left[\lim \limits_{X \in \mathbb{D}} e_X \right] \pi_X = \pi_X && \oplus (\lim \limits_{X \in \mathbb{D}} e_X) \pi_X = \left(1 \times \lim \limits_{X \in \mathbb{D}} e_X \right) \mathsf{D}\left[\lim \limits_{X \in \mathbb{D}} e_X \right] \pi_X\end{align*} 
 Then by the universal property of the limit, it follows that:
 \begin{align*}
\left \langle 0,1 \right\rangle \mathsf{D}\left[\lim \limits_{X \in \mathbb{D}} e_X \right] = 1 && \oplus (\lim \limits_{X \in \mathbb{D}} e_X) = \left(1 \times \lim \limits_{X \in \mathbb{D}} e_X \right) \mathsf{D}\left[\lim \limits_{X \in \mathbb{D}} e_X \right] 
\end{align*}
and so we conclude that $\lim \limits_{X \in \mathbb{D}} e_X$ is a differential exponential map. From here, it is straightforward to conclude that the limit of $\mathbb{D} \xrightarrow{\mathsf{F}} \mathsf{DEM}[\mathbb{X}]$ is the pair:
 \[(\lim \limits_{X \in \mathbb{D}} \mathsf{U}\left(\mathsf{F}(X) \right), \lim \limits_{X \in \mathbb{D}} e_X)\]
 with projections $(\lim \limits_{X \in \mathbb{D}} \mathsf{U}\left(\mathsf{F}(X) \right), \lim \limits_{X \in \mathbb{D}} e_X) \xrightarrow{\pi_X} \left(\mathsf{U}\left(\mathsf{F}(X) \right), e_X \right)$. Furthermore, $\mathsf{U}(\lim \limits_{X \in \mathbb{D}} \mathsf{U}\left(\mathsf{F}(X) \right), \lim \limits_{X \in \mathbb{D}} e_X) = \lim \limits_{X \in \mathbb{D}} \mathsf{U}\left(\mathsf{F}(X) \right)$ and $\mathsf{U}(\pi_X) = \pi_X$, and it follows from the definition of $\lim \limits_{X \in \mathbb{D}} e_X$ that this is the unique cone over $\mathsf{F}$ with this property. Therefore, we conclude that $\mathsf{DEM}[\mathbb{X}] \xrightarrow{\mathsf{U}} \mathsf{LIN}[\mathbb{X}]$ creates all limits. 
\end{proof} 

By Lemma \ref{linlem}, the projection maps of the product are linear. As such, an immediate consequence of Lemma \ref{reflectlem} is that the product of differential exponential maps is again a differential exponential map. Therefore, the category of differential exponentials maps has finite products. 

\begin{corollary}\label{corprod} In a Cartesian differential category:
\begin{enumerate}[{\em (i)}]
\item For the terminal object $\top$, $\top \xrightarrow{1} \top$ is a differential exponential map; 
\item If $A \xrightarrow{e} A$ and $B \xrightarrow{e^\prime} B$ are differential exponential maps, then their product $A \times B \xrightarrow{e \times e^\prime} A \times B$ is a differential exponential map.
\end{enumerate}
Furthermore, for a Cartesian differential category $\mathbb{X}$, $\mathsf{DEM}[\mathbb{X}]$ has finite products where the terminal object is $(\top, 1)$, and where the product of $(A, e)$ and $(B, e^\prime)$ is $(A \times B, e \times e^\prime)$ with projection maps $(A \times B, e \times e^\prime) \xrightarrow{\pi_0} (A,e)$ and $(A \times B, e \times e^\prime) \xrightarrow{\pi_1} (B,e^\prime)$.
\end{corollary}

It is important to note that a differential exponential map for $A \times B$ is not necessarily the product of differential exponential maps, that is, of the form $e \times e^\prime$. See Example \ref{expex} for three examples of differential exponential maps which are not the products of differential exponential maps. 

Our next observation is that the category of differential exponential maps is also a Cartesian tangent category. We first show that the tangent functor maps differential exponential maps to differential exponential maps. 
  
\begin{lemma}\label{exptan} If $A \xrightarrow{e} A$ is a differential exponential map, then $A \times A \xrightarrow{\mathsf{T}(e)} A \times A$ is a differential exponential map.
\end{lemma}
\begin{proof} We first show that $\langle 0, 1 \rangle \mathsf{D}\left[\mathsf{T}(e) \right] = 1$: 
 \begin{align*}
\langle 0, 1 \rangle \mathsf{D}\left[\mathsf{T}(e) \right] &=~ \langle 0, 1 \rangle c \mathsf{T}(\mathsf{D}[e]) \tag{Lemma \ref{tanlem}} \\
&=~\mathsf{T}(\langle 0,1 \rangle) \mathsf{T}(\mathsf{D}[e]) \tag{Lemma \ref{tanlem}} \\
&=~\mathsf{T} \left(\langle 0,1 \rangle \mathsf{D}[e] \right) \tag{$\mathsf{T}$ is a functor} \\
&=~\mathsf{T}(1) \tag{\ref{dem}} \\
&=~ 1 \tag{$\mathsf{T}$ is a functor}
\end{align*}
Next we show that $(1 \times \mathsf{T}(e)) \mathsf{D}\left[\mathsf{T}(e) \right] = \oplus \mathsf{T}(e)$: 
 \begin{align*}
(1 \times \mathsf{T}(e)) \mathsf{D}\left[\mathsf{T}(e) \right] &=~ (1 \times \mathsf{T}(e)) c \mathsf{T}(\mathsf{D}[e]) \tag{Lemma \ref{tanlem}} \\
&=~ c \mathsf{T}(1 \times e) \mathsf{T}(\mathsf{D}[e]) \tag{Lemma \ref{tanlem}} \\
&=~c \mathsf{T} \left( (1 \times e) \mathsf{D}[e] \right) \tag{$\mathsf{T}$ is a functor} \\
&=~c \mathsf{T} \left( \oplus e \right) \tag{\ref{dem}} \\
&=~c \mathsf{T}(\oplus) \mathsf{T}(e) \tag{$\mathsf{T}$ is a functor} \\
&=~ \oplus \mathsf{T}(e) \tag{Lemma \ref{tanlem}}
\end{align*}
So we conclude that $\mathsf{T}(e)$ is a differential exponential map. 
\end{proof} 

 \begin{proposition}\label{DEMtangent} For a Cartesian differential category $\mathbb{X}$, $\mathsf{DEM}[\mathbb{X}]$ is a Cartesian tangent category where the tangent functor $\mathsf{T}: \mathsf{DEM}[\mathbb{X}] \to \mathsf{DEM}[\mathbb{X}]$ is defined on objects as $\mathsf{T}(A, e) := (A \times A, \mathsf{T}(e))$ and on maps as $\mathsf{T}(f) = f \times f$, and where the remaining tangent structure is the same as for $\mathsf{LIN}[\mathbb{X}]$ (which is the same as for $\mathbb{X}$ and can be found in {\cite[Proposition 4.7]{cockett2014differential}}). 
 \end{proposition}
 \begin{proof} The tangent functor $\mathsf{T}$ is well defined by Lemma \ref{exptan}. Since all the maps of the tangent structure of $\mathbb{X}$ are linear, it follows that by their respective naturality with the tangent functor of $\mathbb{X}$, they are also maps in $\mathsf{DEM}[\mathbb{X}]$ which are natural for its tangent functor. The existence of the necessary limits for tangent structure in $\mathsf{DEM}[\mathbb{X}]$ will follow from Corollary \ref{lintan} and Lemma \ref{reflectlem}. And lastly, the required equalities for tangent structure will hold since they hold in $\mathsf{LIN}[\mathbb{X}]$. So we conclude that $\mathsf{DEM}[\mathbb{X}]$ is a tangent category. Furthermore, since $\mathbb{X}$ is a Cartesian tangent category, it follows that $\mathsf{DEM}[\mathbb{X}]$ is also a Cartesian tangent category. \end{proof} 

It may be tempting to think that $\mathsf{DEM}[\mathbb{X}]$ is also a Cartesian differential category, but this is not the case. Indeed note that even if $(A, e) \xrightarrow{f} (B, e^\prime)$ and $(A, e) \xrightarrow{g} (A, e^\prime)$ are maps in $\mathsf{DEM}[\mathbb{X}]$, their sum $f +g$ is not (in general) a map in $\mathsf{DEM}[\mathbb{X}]$ since it is not necessarily the case that $(f+g)e$ equals $e^\prime(f+g)$. This implies that $\mathsf{DEM}[\mathbb{X}]$ is not a Cartesian left additive category, and so in particular not a Cartesian differential category. 

\begin{example}\label{expex} \normalfont Here are now examples of differential exponential maps in the Cartesian differential category $\mathsf{SMOOTH}$ (as defined in Example \ref{smoothex}).  
\begin{enumerate}[{\em (i)}]
\item Consider the exponential function $\mathbb{R} \xrightarrow{e^x} \mathbb{R}$, $x \mapsto e^x$. Since $\frac{\partial e^x}{\partial x}(a) = e^a$, we have that:
\[\mathsf{D}[e](a,b) = \frac{\partial e^x}{\partial x}(a)b = e^a b\]
The exponential function satisfies the left diagram of (\ref{dem}) since $e^0=1$:
\[ \mathsf{D}[e](0,a) = e^0 a = a \]
while the right diagram of (\ref{dem}) is also satisfied since $e^a e^b = e^{a+b}$: 
\[\mathsf{D}[e^x](a,e^b) = e^a e^b = e^{a+b} \]
Then the exponential function $\mathbb{R} \xrightarrow{e^x} \mathbb{R}$ is a differential exponential map. 
\item Applying Corollary \ref{corprod}.ii to the exponential function $\mathbb{R} \xrightarrow{e^x} \mathbb{R}$, the point-wise exponential functions $\mathbb{R}^{n} \xrightarrow{e^x \times \hdots \times e^x} \mathbb{R}^n$:
\[(x_1, \hdots, x_n) \mapsto (e^{x_1}, \hdots, e^{x_n})\]
are differential exponential maps in $\mathsf{SMOOTH}$. 
\item Applying Lemma \ref{exptan} to the exponential function $\mathbb{R} \xrightarrow{e^x} \mathbb{R}$, the smooth function ${\mathbb{R}^{2} \xrightarrow{\mathsf{T}(e^x)} \mathbb{R}^2}$, which is worked out to be: 
\[(x,y) \mapsto (e^x, e^xy)\]
is a differential exponential map. To better understand $\mathsf{T}(e^x)$, consider the ring of dual numbers $\mathbb{R}[\varepsilon]$ \cite[Section 1.1.3]{rosenfeld2013geometry}:
 \[ \mathbb{R}[\varepsilon] = \lbrace x + y\varepsilon \vert ~ x,y \in \mathbb{R}, \varepsilon^2=0 \rbrace\] 
As explained in \cite[Section 1.1.5]{rosenfeld2013geometry}, the dual number exponential function is $e^{x+y\varepsilon} = e^x + e^x y \varepsilon$. It may be useful to the reader to work out this example using the power series definition of the exponential function: 
\[ e^{x+y\varepsilon} = \sum \limits^{\infty}_{n =0} \frac{(x + y \varepsilon)^n}{n!} \]
Note that by the binomial theorem and power series multiplication, one can still derive that $e^{x+y\varepsilon} = e^xe^{y\varepsilon}$. Therefore, it remains to compute $e^{y\varepsilon}$. Since $\varepsilon^{n} =0$ for all $n \geq 2$, we obtain:
\[ e^{y\varepsilon} = \sum \limits^{\infty}_{n =0} \frac{(y \varepsilon)^n}{n!} = \sum \limits^{\infty}_{n =0} \frac{y^n \varepsilon^n}{n!} = 1 + y \varepsilon \]
So $e^{y\varepsilon} = 1 + y \varepsilon$. Therefore,
\[ e^{x+y\varepsilon} = e^xe^{y\varepsilon} = e^x(1+y\varepsilon) = e^x + e^x y \varepsilon \]
Writing dual numbers $x+y\varepsilon$ instead as $(x,y)$, it becomes clear that $\mathsf{T}(e^x)$ is the real smooth function associated to the dual numbers exponential function. This relation was to be expected since tangent categories are closely related to dual numbers and Weil algebras \cite{leung2017classifying}. Furthermore, note that in dual number notation, $\mathsf{D}[\mathsf{T}(e^x)]\left( a+b\varepsilon , c+d\varepsilon \right) = e^{a+b\varepsilon}(c+d\varepsilon)$. 
\item Let $\mathbb{C}$ be the field of complex numbers: 
 \[ \mathbb{C} = \lbrace x + iy \vert ~ x,y \in \mathbb{R}, i^2=-1 \rbrace\] 
The complex exponential function is $e^{x+iy} = e^x\cos(y) + i e^x\sin(y)$, where $\cos$ and $\sin$ are the trigonometric cosine and sine functions respectively. The complex exponential function is of course derived from the power series definition:
\[ e^{x+iy} = \sum \limits^{\infty}_{n =0} \frac{(x + iy)^n}{n!} \]
and then simplifying by using that $e^{x+iy}= e^x e^{iy}$, $i^2=-1$, and the Taylor series expansions of $\sin(x)$ and $\cos(x)$. The complex exponential function can be expressed as the smooth real function $\mathbb{R}^2 \xrightarrow{\epsilon} \mathbb{R}^2$, 
\[(x,y) \mapsto (e^x \cos(y), e^x \sin(y))\]
By the Leibniz rule and the derivative identities for both the exponential and trigonometric functions, we have that:
\begin{align*}
\mathsf{D}[\epsilon]\left( (a,b), (c,d) \right) = \left( e^a\cos(b)c - e^a \sin(b)d, e^a\sin(b)c + e^a \cos(b)d\right)
\end{align*}
Or using complex number notation: $\mathsf{D}[\epsilon]\left( a+ib , c+id \right) = e^{a+ib}(c+id)$. It is well known that the complex exponential function satisfies the same basic properties as the real exponential function, that is, $e^0 = 1$ and $e^{z+w} = e^z e^w$ for $z,w \in \mathbb{C}$. As such, we can easily compute that (using complex number notation): 
\begin{align*}
\mathsf{D}[\epsilon]\left( 0, a+ib \right) &= e^{0}(a+ib) = a+ib \\
\mathsf{D}[\epsilon]\left( a+ib , e^{c+id} \right) &= e^{a+ib}e^{c+id} = e^{(a+ib) + (c + id)} 
\end{align*}
Therefore, it follows that $\mathbb{R}^2 \xrightarrow{\epsilon} \mathbb{R}^2$ is a differential exponential map. 
\item Let $\mathbb{C}^\prime$ be the ring of split-complex numbers \cite[Section 1.1.2]{rosenfeld2013geometry} (also known as hyperbolic twocomplex numbers \cite{olariu2002complex}):
 \[ \mathbb{C}^\prime = \lbrace x + jy \vert ~ x,y \in \mathbb{R}, j^2= 1 \rbrace\] 
The split complex exponential function \cite[Section 1.3]{olariu2002complex} is instead defined using the hyperbolic cosine and sine functions  $\cosh$ and $\sinh$, that is, $e^{x+jy} = e^x\cosh(y) + j e^x\sinh(y)$. As in the previous example, the split complex exponential function is derived from its power series definition: 
\[ e^{x+jy} = \sum \limits^{\infty}_{n =0} \frac{(x + jy)^n}{n!} \]
and then simplifying by using that $e^{x+jy}= e^x e^{jy}$, $j^2=1$, and the Taylor series expansions of $\sinh(x)$ and $\cosh(x)$, as done in \cite[Section 1.3]{olariu2002complex}. 
Similar to the complex exponential function, the split complex exponential function can be expressed as the smooth real function $\mathbb{R}^2 \xrightarrow{\epsilon^\prime} \mathbb{R}^2$:
\[(x,y) \mapsto (e^x \cosh(y), e^x \sinh(y))\]
By the Leibniz rule and the derivative identities for both the exponential and hyperbolic functions, we have that:
\begin{align*}
\mathsf{D}[\epsilon^\prime]\left( (a,b), (c,d) \right) = \left( e^a\cosh(b)c + e^a \sinh(b)d, e^a\sinh(b)c + e^a \cosh(b)d\right)
\end{align*}
Or using split complex numbers: $\mathsf{D}[\epsilon]\left( a+jb , c+jd \right) = e^{a+jb}(c+jd)$. As explained in \cite{olariu2002complex}, the split complex exponential function satisfies that $e^0=1$ and $e^{u+v} = e^u e^v$ for $u,v \in \mathbb{C}^\prime$. Then we can compute that (using split complex number notation): 
\begin{align*}
\mathsf{D}[\epsilon^\prime]\left( 0, a+jb \right) = \mathsf{D}[\epsilon]\left( 0, a+jb \right) = e^{0}(a+jb) = a+jb \\
\mathsf{D}[\epsilon^\prime]\left( a+jb , e^{c+jd} \right) = e^{a+jb}e^{c+jd} = e^{(a+jb) + (c + jd)} 
\end{align*}
Therefore, $\mathbb{R}^2 \xrightarrow{\epsilon^\prime} \mathbb{R}^2$ is a differential exponential map. \end{enumerate}
Note that Example \ref{expex}.(iii), (iv), and (v) are not the product of differential exponential maps in the sense of Corollary \ref{corprod}.ii. That said, one could take the product of any of these differential exponential maps to obtain a multitude of other examples. 
\end{example}

\begin{example}\label{curveexp1} \normalfont Another example of a differential exponential map can be found in \cite[Definition 5.20]{cockett2019differential}. Briefly, in a Cartesian tangent category, a differential curve object \cite[Definition 5.14]{cockett2019differential} (viewed as an object in the Cartesian differential category of differential objects \cite{cockett2014differential}) admits a canonical differential exponential map which arises as the solution to the well known associated differential equations. As a particular example, in the tangent category of real smooth manifolds, the differential curve object is $\mathbb{R}$ and its induced differential exponential map is the canonical exponential function $e^x$. In Section \ref{DYNsec} we will discuss the link between differential exponential maps and differential equations, and in particular, we will see in Proposition \ref{uniqueprop} how a differential exponential map arises as the solution of a certain initial value problem. 
\end{example}

Corollary \ref{corprod}.i tells us that the identity map of the terminal object is a differential exponential map. So we conclude this section with the observation that a differential exponential map is linear if and only if it is the identity map of a terminal object. 

\begin{lemma}\label{noexp} A differential exponential map $A \xrightarrow{e} A$ is reduced (i.e. $0e=0$) if and only if $A$ is a terminal object. Therefore, a differential exponential map $A \xrightarrow{e} A$ is additive or linear if and only if $A$ is a terminal object. 
\end{lemma}
\begin{proof} Suppose that $A \xrightarrow{e} A$ is a differential exponential map which satisfies $0e=0$. Then we have that:  
 \begin{align*}
e &=~\langle 1, 0 \rangle \oplus e \tag{Lemma \ref{opluslem}} \\
&=~ \langle 1, 0 \rangle (1 \times e) \mathsf{D}[e] \tag{\ref{dem}} \\
&=~\langle 1, 0 e \rangle \mathsf{D}[e] \\
&=~\langle 1, 0 \rangle \mathsf{D}[e] \tag{$e$ reduced}\\
&=~ 0 \tag*{\bf[CD.2]}
\end{align*}
So $e=0$. Now note that in a Cartesian left additive category, $A$ is a terminal object if and only if $1_A = 0$ (we leave this to the reader to check for themselves). Then we have that: 
 \begin{align*}
1_A &=~\langle 0, 1 \rangle \mathsf{D}[e] \tag{\ref{dem}} \\
 &=~ \langle 0, 1 \rangle \mathsf{D}[0] \tag{$e=0$} \\
 &=~ 0 \tag*{\bf[CD.1]}
\end{align*}
So $1_A = 0$, and so $A$ is a terminal object. Conversely, if $A$ is a terminal object, then we must have that $e = 1_A = 0$, and so clearly $e$ is reduced. For the second statement, note that by definition every additive map is reduced, and since linear maps are additive (Lemma \ref{linlem}), they are also reduced. So if $e$ is additive or linear, it is reduced and therefore $A$ is a terminal object. Conversely, if $A$ is a terminal object, then $e$ must be the identity, and identity maps are always linear and additive (Lemma \ref{linlem}). 
\end{proof} 

\begin{example}\label{exnoexp} \normalfont Every category with finite biproducts is a Cartesian differential category with differential combinator defined as $\mathsf{D}[f] = \pi_1 f$, which means that every map is linear. Therefore, in this case, the only differential exponential maps are the identity maps on terminal objects. \end{example}

Note that for a Cartesian differential category $\mathbb{X}$, Lemma \ref{noexp} also implies that the only differential object \cite{cockett2014differential} in the Cartesian tangent category $\mathsf{DEM}[\mathbb{X}]$ is the terminal object. 

\section{Differential Exponential Rigs}\label{dessec}

In this section, we introduce \emph{differential exponential rigs}, which provide an equivalent alternative characterization of differential exponential maps. We will show that every differential exponential map induces a differential exponential rig (Proposition \ref{prop2}) and that conversely, every differential exponential rig induces a differential exponential map (Proposition \ref{prop1}). We will also show that for a Cartesian differential category, its category of differential exponential maps is isomorphic to its category of differential exponential rigs (Theorem \ref{isothm}). We begin by reviewing differential rigs, which are rigs in a Cartesian differential category whose multiplication is bilinear in the differential sense.

\begin{definition} A \textbf{differential rig} in a Cartesian differential category is a triple $(A, \odot, u)$ consisting of an object $A$ and two maps $A \times A \xrightarrow{\odot} A$ and $\top \xrightarrow{u} A$ such that $(A, \odot, u)$ is a commutative monoid and $A \times A \xrightarrow{\odot} A$ is bilinear, that is, the following equality holds: 
 \begin{equation}\label{odotbilin}\begin{gathered} \mathsf{D}[\odot] = (\pi_0 \times \pi_1) \odot + (\pi_1 \times \pi_0) \odot \end{gathered}\end{equation}
\end{definition} 

We should justify the use of the term rig in differential rig. Indeed, the term (commutative) rig should imply that there are two (commutative) monoid structures that satisfy the expected distributivity axioms. But we know that in a Cartesian differential category, as discussed in Lemma \ref{opluslem}, every object $A$ already comes equipped with a commutative monoid structure $(A, \oplus, 0)$. So every differential rig $(A, \odot, u)$ does come with two commutative monoid structures. The required distributivity axioms are captured by the fact that $\odot$ is bilinear, and therefore additive in each argument -- which is an equivalent way of saying that $\oplus$ and $\odot$ distribute over one another in the rig sense. 

\begin{lemma} If $(A, \odot, u)$ is a differential rig, then $(A, \odot, u, \oplus, 0)$ is a commutative rig, that is, the following diagrams commute:
 \begin{equation}\label{distributeeq}\begin{gathered} \xymatrixcolsep{2.5pc}\xymatrixrowsep{1.5pc}\xymatrix{
 A \ar[d]_-{\langle 0, 1 \rangle} \ar[dr]^-{0} \ar[r]^-{\langle 1, 0 \rangle} & A \times A \ar[d]^-{\odot} & A \times (A \times A) \ar[d]_-{\left\langle 1 \times \pi_0, 1 \times \pi_1 \right\rangle} \ar[r]^-{1 \times \oplus} & (A \times A) \ar[dd]^-{\odot} \\ 
 A \times A \ar[r]_-{\odot} & A & (A \times A) \times (A \times A) \ar[d]_-{\odot \times \odot} \\
 & & A \times A \ar[r]_-{\oplus} & A } \\
 \xymatrixcolsep{2.5pc}\xymatrixrowsep{1.5pc}\xymatrix{
 (A \times A) \times A \ar[r]^-{\oplus \times 1} \ar[d]_-{\left\langle \pi_0 \times 1, \pi_1 \times 1 \right\rangle} & A \times A \ar[dd]^-{\odot} \\ 
 (A \times A) \times (A \times A) \ar[d]_-{\odot \times \odot} \\
 A \times A \ar[r]_-{\oplus} & A }\end{gathered}\end{equation} 
where $\oplus$ is defined as in Lemma \ref{opluslem}. 
\end{lemma}

A differential exponential rig is a differential rig equipped with an endomorphism which satisfies analogues of the three essential properties of the classical exponential function. This endomorphism will, of course, turn out to be a differential exponential map. 

\begin{definition} A \textbf{differential exponential rig} in a Cartesian differential category is a quadruple $(A, \odot, u, e)$ consisting of a differential rig $(A, \odot, u)$ and a map $A \xrightarrow{e} A$, such that the following diagrams commute:
 \begin{equation}\label{des}\begin{gathered} \xymatrixcolsep{2.5pc}\xymatrix{A \times A \ar[dr]_-{\mathsf{D}[e]} \ar[r]^-{e \times 1} & A \times A \ar[d]^-{\odot} & \top \ar[dr]_-{u} \ar[r]^-{0} & A \ar[d]^-{e} & A \times A \ar[d]_-{e \times e} \ar[r]^-{\oplus} & A \ar[d]^-{e} \\
 & A && A & A \times A \ar[r]_-{\odot} & A } \end{gathered}\end{equation}
 where $\oplus$ is defined as in Lemma \ref{opluslem}. 
 \end{definition}
  
Examples of differential exponential rigs can be found below in Example \ref{ex2} after we have proven Proposition \ref{prop2} that every differential exponential map induces a differential exponential rig. If one keeps in mind the classical exponential function $e^x$, then the axioms of a differential exponential rig are straightforward to understand. The leftmost diagram of (\ref{des}) generalizes that $\mathsf{D}[e^x](x,y)=e^xy$, the middle diagram generalizes that $e^0=1$, and lastly the rightmost diagram generalizes that $e^{x+y}=e^xe^y$. Also note that the two right most diagrams of (\ref{des}) says that $e$ is a monoid morphism from $(A, \oplus, 0)$ to $(A, \odot, u)$. It is also worth discussing why the generalization of $e^{x+y}=e^xe^y$ is included in the axioms of a differential exponential rig, apart from being a desirable useful identity and one which is often included in algebraic generalizations of exponential functions \cite{van1984exponential}. Indeed in the classical case, the two leftmost diagrams of (\ref{des}) are sufficient for characterizing the exponential function since $e^x$ is the unique solution to the initial value problem $f^\prime(x) = f(x)$ with $f(0)=1$, and from that definition it is possible to prove that $e^{x+y}=e^xe^y$. In an arbitrary Cartesian differential category, however, it is not necessarily the case that $\mathsf{D}[e] = (e \times 1) \odot$ and $0e=u$ implies $\oplus e = (e \times e) \odot$. Furthermore, all three identities are necessary in proving that there is a bijective correspondence between differential exponential maps and differential exponential rigs. In turn, this allows us to drop the extra requirement for differential rig structure in characterizing these abstract exponential functions. That said, as we will see in Proposition \ref{uniqueprop}, if one assumes uniqueness of solutions to certain initial value problems as in the classical case, then it is possible to derive $\oplus e = (e \times e) \odot$ from $\mathsf{D}[e] = (e \times 1) \odot$ and $0e=u$. 

\begin{proposition}\label{prop1} Let $(A, \odot, u, e)$ be a differential exponential rig. Then $A \xrightarrow{e} A$ is a differential exponential map. \end{proposition}
\begin{proof} We first show that $\langle 0,1 \rangle \mathsf{D}[e] = 1$: 
\begin{align*}
\langle 0,1 \rangle \mathsf{D}[e] &=~ \langle 0,1 \rangle (e \times 1) \odot \tag{\ref{des}} \\
&=~ \langle 0e, 1 \rangle \odot \\
&=~ \langle 0u, 1 \rangle \odot \tag{\ref{des}} \\
&=~ 1 \tag{\ref{ostarmonoideq}}
\end{align*}
Next we show that $(1 \times e) \mathsf{D}[e] = \oplus e$: 
\begin{align*}
(1 \times e) \mathsf{D}[e] &=~ (1 \times e)(e \times 1) \odot \tag{\ref{des}} \\
&=~ (e \times e) \odot \\
&=~ \oplus e \tag{\ref{des}}
\end{align*}
So we conclude that $e$ is a differential exponential map.  
\end{proof} 

In order to show the converse of Proposition \ref{prop1}, consider the classical exponential function $e^x$ and note that its second order derivative is:
\[\mathsf{D}^2[e^x]((x,y), (z,w))=e^{x}yz + e^xw\]
Setting $x=0$ and $w=0$, one obtains precisely the multiplication of real numbers: 
\[\mathsf{D}^2[e^x]((0,y), (z,0))=yz\]
The unit for this multiplication is obtain from $e^0 = 1$. Generalizing this construction allows one to show how a differential exponential map induces a differential exponential rig. 

Before doing the proof in an arbitrary Cartesian differential category, it may be worth mapping out how the proof will go using $e^x$. Commutativity of the multiplication, which is:
\[\mathsf{D}^2[e^x]((0,y), (z,0))= \mathsf{D}^2[e^x]((0,z), (y,0)) \]
follows from \textbf{[CD.7]}, which allows one to swap the middle two arguments of the second derivative. To show that $e^0$ is the unit, that is:
\[\mathsf{D}^2[e^x]((0,e^0), (x,0))= x \]
we first observe that: 
\[\mathsf{D}^2[e^x]((0,e^x), (y,0))= e^x y = \mathsf{D}[e^x](x,y) \]
and then use the differential exponential map axiom to conclude that $e^0$ is indeed the unit. Associativity of the multiplication, which is that: 
\begin{align*}
&\mathsf{D}^2[e^x]\left((0,x), \left(\mathsf{D}^2[e^x]((0,y), (z,0)),0 \right) \right) =\\
 &\mathsf{D}^2[e^x]\left( \left(0,\mathsf{D}^2[e^x]((0,x), (y,0)) \right), \left(z,0 \right) \right)
\end{align*}
is the trickiest part of the proof. Essentially, using the differential combinator axioms and the differential exponential map axioms, one can simplify both sides of the associativity law to 
the third order derivative of $e^x$ (with $0$ evaluated in the appropriate variables): 
\[ D^3[e^x]((0,x),(y,0), (z,0),(0,0)) = xyz \]
which itself turns out to be real numbers multiplication of three variables. 

\begin{proposition}\label{prop2} Let $A \xrightarrow{e} A$ be a differential exponential map, and define the maps ${A \times A \xrightarrow{\odot_e} A}$ and $\top \xrightarrow{u_e} A$ respectively as follows: 
\[\odot_e := \xymatrixcolsep{5pc}\xymatrix{ A \times A \ar[r]^-{\langle 0,1 \rangle \times \langle 1, 0 \rangle} & (A \times A) \times (A \times A) \ar[r]^-{\mathsf{D}^2[e]} & A }\]
\[u_e := \xymatrixcolsep{5pc}\xymatrix{ \top \ar[r]^-{0} & A \ar[r]^-{e} & A } \]
Then $(A, \odot_e, u_e, e)$ is a differential exponential rig. \end{proposition}
\begin{proof} We will first prove that $e$ satisfies the three identities of (\ref{des}), as these will help simplify the proof that $(A, \odot_e, u_e)$ is a differential rig. We first prove that $\mathsf{D}[e] = (e \times 1) \odot_e$: 
\begin{align*}
\mathsf{D}[e] &=~ \left( \langle 0,1 \rangle \times \langle 1, 0 \rangle \right) (\oplus \times \oplus) \mathsf{D}[e] \tag{Lemma \ref{opluslem}} \\
&=~\left( \langle 0,1 \rangle \times \langle 1, 0 \rangle \right) \mathsf{D}[\oplus e] \tag{$\oplus$ linear + Lemma \ref{linlem}} \\
&=~\left( \langle 0,1 \rangle \times \langle 1, 0 \rangle \right) \mathsf{D}\!\left[(1 \times e)\mathsf{D}[e] \right] \tag{\ref{dem}} \\
&=~\left( \langle 0,1 \rangle \times \langle 1, 0 \rangle \right) \mathsf{T}(1 \times e) \mathsf{D}^2[e] \tag{Lemma \ref{tanlem}} \\
&=~\left( \langle 0,1 \rangle \times \langle 1, 0 \rangle \right) c \left(\mathsf{T}(1) \times \mathsf{T}(e) \right) \mathsf{D}^2[e] \tag{Lemma \ref{tanlem}}\\
&=~\left( \langle 0,1 \rangle \times \langle 1, 0 \rangle \right) c \left((1 \times 1) \times \mathsf{T}(e) \right) \mathsf{D}^2[e] \tag{$\mathsf{T}$ is a functor}\\
&=~\langle \pi_1, \pi_0 \rangle \left( \langle 0,1 \rangle \times \langle 1, 0 \rangle \right) \left((1 \times 1) \times \mathsf{T}(e) \right) \mathsf{D}^2[e]\\
&=~\langle \pi_1, \pi_0 \rangle (1 \times e) \left( \langle 0,1 \rangle \times \langle 1, 0 \rangle \right) \mathsf{D}^2[e] \tag{Lemma \ref{tanlem}} \\
&=~(e \times 1) \left( \langle 0,1 \rangle \times \langle 1, 0 \rangle \right) \mathsf{D}^2[e]\\
&=~(e \times 1) \odot_e
\end{align*}
Using the above equality, we can easily show that $\oplus e = (e \times e) \odot_e$:
  \begin{align*}
 \oplus e &=~ (1 \times e)\mathsf{D}[e] \\
 &=~ (1 \times e)(e \times 1)\odot_e \\
 &=~ (e \times e) \odot_e
 \end{align*}
The remaining identity, $0 e = u_e$ is automatic by construction. So $e$ satisfies three identities of (\ref{des}). Next we show that $(A, \odot_e, u_e)$ is a differential rig. 

We first explain why $\odot_e$ is bilinear. In \cite[Section 3]{cockett2011faa} it was shown that for every map ${A \xrightarrow{f} B}$, its second order partial derivative:
\[(A \times A) \times A \xrightarrow{(1 \times 1) \times \langle 1, 0 \rangle} (A \times A) \times (A \times A) \xrightarrow{\mathsf{D}^2[f]} B\]
was bilinear in \emph{context} $A$, so bilinear in its last two arguments $A$ or equivalently bilinear with respect to the differential combinator of the simple slice category over $A$ \cite[Section 4.5]{blute2009cartesian}. By \cite[Proposition 4.1.3]{blute2015cartesian}, bilinear maps in context are preserved by pre-composition with maps which leave the bilinear arguments unaffected, that is, by pre-composition by maps of the form $(g \times 1) \times 1$. Therefore the composite:
\[(\top \times A) \times A \xrightarrow{(0 \times 1) \times \langle 1, 0 \rangle} (A \times A) \times (A \times A) \xrightarrow{\mathsf{D}^2[f]} B\]
is bilinear in context $\top$. However maps which are bilinear in context $\top$ correspond precisely to bilinear maps without context. In this case, we obtain that the composite: 
\[A \times A \xrightarrow{\langle 0, 1 \rangle \times \langle 1, 0 \rangle} (A \times A) \times (A \times A) \xrightarrow{\mathsf{D}^2[f]} B\]
is bilinear. Setting $f=e$, we conclude that $\odot_e$ is bilinear. 

Now we show that $(A, \odot_e, u_e)$ is a commutative monoid by following the intuition provided above. First that $\odot_e$ is commutative, $\langle \pi_1, \pi_0 \rangle \odot_e = \odot_e$, follows from {\bf [CD.7]}: 
 \begin{align*}
\langle \pi_1, \pi_0 \rangle \odot_e &=~ \langle \pi_1, \pi_0 \rangle \left( \langle 0,1 \rangle \times \langle 1, 0 \rangle \right) \mathsf{D}^2[e] \\
&=~ \left( \langle 0,1 \rangle \times \langle 1, 0 \rangle \right) c \mathsf{D}^2[e] \\
&=~ \left( \langle 0,1 \rangle \times \langle 1, 0 \rangle \right) \mathsf{D}^2[e] \tag*{{\bf [CD.7]}} \\
&=~\odot_e 
 \end{align*}
By commutativity, we need only show one of the unit identities, $\langle 0u_e, 1 \rangle \odot_e = 1$: 
\begin{align*}
\langle 0u_e, 1 \rangle \odot_e &=~ \langle 0 e, 1 \rangle \odot_e \tag{\ref{des}} \\
&=~ \langle 0, 1 \rangle (e \times 1) \odot_e \\
&=~ \langle 0, 1 \rangle \mathsf{D}[e] \tag{\ref{des}} \\
&=~ 1 \tag{\ref{dem}} 
\end{align*}
Finally we prove associativity, which in the author's opinion is the most complex proof in this paper. Let $(A \times A) \times A \xrightarrow{\alpha} A \times (A \times A)$ be the canonical associativity isomorphism $\alpha = \langle \pi_0, \pi_1 \times 1 \rangle$. By Lemma \ref{linlem}, $\alpha$ is linear and so is its inverse $\alpha^{-1}$. As suggested above, our goal will be to show that the third-order derivative gives the three-fold multiplication, that is, we will show that we have the following equality:
\[ \alpha (1 \times \odot_e) \odot_e = \left( \left( \langle 0, 1 \rangle \times \langle 1, 0 \rangle \right) \times \left \langle \langle 1,0 \rangle, 0 \right \rangle \right) \mathsf{D}^3 [e] \]
To do so, first note that by {\bf[CD.5]} and {\bf[CD.6]}, one can show that we have the following equality (which we leave as an exercise to the reader): 
 \begin{equation}\label{eq77}\begin{gathered} \alpha (1 \times \mathsf{D}[f]) \mathsf{D}[g] = \left( (1 \times 1) \times \langle 0,1 \rangle \right) \mathsf{D}\left[ (1 \times f) \mathsf{D}[g] \right]
 \end{gathered}\end{equation}
 Using the above identity, we compute that: 
 \begin{align*}
& \alpha (1 \times \odot_e) \mathsf{D}[e] =~ \alpha \left( 1 \times \left( \langle 0, 1 \rangle \times \langle 1, 0 \rangle \right) \right) (1 \times \mathsf{D}^2[e]) \mathsf{D}[e] \\
 &=~ \left( \left( 1 \times \langle 0, 1 \rangle \right) \times \langle 1, 0 \rangle \right) \alpha (1 \times \mathsf{D}^2[e]) \mathsf{D}[e] \\
 &=~ \left( 1 \times \left( \langle 0, 1 \rangle \times \langle 1, 0 \rangle \right) \right) \left( (1 \times 1) \times \langle 0,1 \rangle \right) \mathsf{D}\left[ (1 \times \mathsf{D}[e]) \mathsf{D}[e] \right] \tag{\ref{eq77}} \\
 &=~ \left( (1 \times \langle 0,1 \rangle) \times \langle 0, \langle 0, 1 \rangle \right) \mathsf{D}\left[ (1 \times \mathsf{D}[e]) \mathsf{D}[e] \right] \\
 &=~ \left( (1 \times \langle 0,1 \rangle) \times \langle 0, \langle 0, 1 \rangle \right) \mathsf{D}\left[ \alpha^{-1} \left( (1 \times 1) \times \langle 0,1 \rangle \right) \mathsf{D}\left[ (1 \times e) \mathsf{D}[e] \right] \right] \tag{\ref{eq77}} \\
 &=~ \left( (1 \times \langle 0,1 \rangle) \times \langle 0, \langle 0, 1 \rangle \right) \mathsf{D}\left[ \alpha^{-1} \left( (1 \times 1) \times \langle 0,1 \rangle \right) \mathsf{D}\left[ \oplus e \right] \right] \tag{\ref{dem}} \\
  &=~ \left( (1 \times \langle 0,1 \rangle) \times \langle 0, \langle 0, 1 \rangle \right) \mathsf{D}\left[ \alpha^{-1} \left( (1 \times 1) \times \langle 0,1 \rangle \right)(\oplus \times \oplus) \mathsf{D}\left[ e \right] \right] \tag{Lemma \ref{opluslem} + Lemma \ref{linlem}} \\
    &=~ \left( (1 \times \langle 0,1 \rangle) \times \langle 0, \langle 0, 1 \rangle \right) \mathsf{D}\left[ \alpha^{-1} (\oplus \times 1) \mathsf{D}\left[ e \right] \right] \tag{Lemma \ref{opluslem}} \\
       &=~ \left( (1 \times \langle 0,1 \rangle) \times \langle 0, \langle 0, 1 \rangle \right)(\alpha^{-1} \times \alpha^{-1}) \mathsf{D}\left[ (\oplus \times 1) \mathsf{D}\left[ e \right] \right] \tag{Lemma \ref{linlem}} \\
    &=~ \left( (1 \times \langle 0,1 \rangle) \times \langle 0, \langle 0, 1 \rangle \right)(\alpha^{-1} \times \alpha^{-1}) \left((\oplus \times 1) \times (\oplus \times 1) \right) \mathsf{D}^2\left[ e \right]\tag{Lemma \ref{opluslem} + Lemma \ref{linlem}} \\
 &=~ \left( (\langle 1, 0 \rangle \times 1) \times \langle \langle 0,1 \rangle, 0 \right) \left((\oplus \times 1) \times (\oplus \times 1) \right) \mathsf{D}^2\left[ e \right] \\
 &=~ \left( (1 \times 1) \times \langle 1, 0 \rangle \right) \mathsf{D}^2\left[ e \right] \tag{Lemma \ref{opluslem}} 
\end{align*}
So we have that: 
 \begin{equation}\label{eq88}\begin{gathered} \alpha (1 \times \odot_e) \mathsf{D}[e] = \left( (1 \times 1) \times \langle 1, 0 \rangle \right) \mathsf{D}^2\left[ e \right] 
 \end{gathered}\end{equation}
 Using this identity, we can simplify $\alpha (1 \times \odot_e) \odot_e$:
  \begin{align*}
& \alpha (1 \times \odot_e) \odot_e =~ \alpha (1 \times \odot_e) \left( \langle 0, 1 \rangle \times \langle 1, 0 \rangle \right) \mathsf{D}^2[e] \\
 &=~\alpha \left( \langle 0, 1 \rangle \times \langle 1, 0 \rangle \right) (1 \times \mathsf{T}(\odot_e)) \mathsf{D}^2[e] \tag{Lemma \ref{tanlem}} \\
 &=~\alpha \left( \langle 0, 1 \rangle \times \langle 1, 0 \rangle \right) (1 \times \mathsf{T}(\odot_e)) c \mathsf{D}^2[e] \tag{{\bf[CD.7]}} \\
 &=~\alpha \left( \langle 0, 1 \rangle \times \langle 1, 0 \rangle \right) (\mathsf{T}(1) \times \mathsf{T}(\odot_e)) c \mathsf{D}^2[e] \tag{$\mathsf{T}$ is a functor} \\
 &=~\alpha \left( \langle 0, 1 \rangle \times \langle 1, 0 \rangle \right) c \mathsf{T}(1 \times \odot_e) \mathsf{D}^2[e] \tag{Lemma \ref{tanlem}} \\
  &=~\alpha \left( \langle 0, 1 \rangle \times \langle 1, 0 \rangle \right) c \mathsf{D}\left[ (1 \times \odot_e) \mathsf{D}[e] \right] \tag{Lemma \ref{tanlem}} \\ 
  &=~\alpha \left( \langle 0, 1 \rangle \times \langle 1, 0 \rangle \right) c \mathsf{D}\left[ \alpha^{-1} \left( (1 \times 1) \times \langle 1, 0 \rangle \right) \mathsf{D}^2\left[ e \right] \right] \tag{\ref{eq88}} \\ 
   &=~ \alpha \left( \langle 0, 1 \rangle \times \langle 1, 0 \rangle \right) c (\alpha^{-1} \times \alpha^{-1}) \mathsf{D}\left[ \left( (1 \times 1) \times \langle 1, 0 \rangle \right) \mathsf{D}^2\left[ e \right] \right] \tag{Lemma \ref{linlem}} \\
  &=~ \alpha \left( \langle 0, 1 \rangle \times \langle 1, 0 \rangle \right) c (\alpha^{-1} \times \alpha^{-1})\\
  &~~~ \left( \left( (1 \times 1) \times \langle 1, 0 \rangle \right) \times \left( (1 \times 1) \times \langle 1, 0 \rangle \right) \right) \mathsf{D}^3 [e] \tag{Lemma \ref{linlem}} \\
   &=~ \alpha \left( \langle 0, 1 \rangle \times \langle 1, 0 \rangle \right) c (\alpha^{-1} \times \alpha^{-1}) \\
    &~~~ \left( \left( (1 \times 1) \times \langle 1, 0 \rangle \right) \times \left( (1 \times 1) \times \langle 1, 0 \rangle \right) \right) c \mathsf{D}^3 [e] \tag{{\bf[CD.7]}} \\
     &=~ \alpha \left( \langle 0, 1 \rangle \times \langle 1, 0 \rangle \right) c (\alpha^{-1} \times \alpha^{-1})\\
      &~~~ \left( \left( (1 \times 1) \times \langle 1, 0 \rangle \right) \times \left( (1 \times 1) \times \langle 1, 0 \rangle \right) \right) c \mathsf{D}\left[ c \mathsf{D}^2 [e] \right] \tag{{\bf[CD.7]}} \\
      &=~ \alpha \left( \langle 0, 1 \rangle \times \langle 1, 0 \rangle \right) c (\alpha^{-1} \times \alpha^{-1}) \\
       &~~~\left( \left( (1 \times 1) \times \langle 1, 0 \rangle \right) \times \left( (1 \times 1) \times \langle 1, 0 \rangle \right) \right) c (c \times c) \mathsf{D}^3 [e] \tag{Lemma \ref{linlem}} \\
      &=~ \left( \left( \langle 0, 1 \rangle \times \langle 1, 0 \rangle \right) \times \left \langle \langle 1,0 \rangle, 0 \right \rangle \right) \mathsf{D}^3 [e]
\end{align*}
So we have that: 
 \begin{equation}\label{eq99}\begin{gathered} \alpha (1 \times \odot_e) \odot_e = \left( \left( \langle 0, 1 \rangle \times \langle 1, 0 \rangle \right) \times \left \langle \langle 1,0 \rangle, 0 \right \rangle \right) \mathsf{D}^3 [e]
 \end{gathered}\end{equation}
 Now using the above identity and that we've already shown that $\odot_e$ is commutative, we finally can show that $\odot_e$ is associative, $(\odot_e \times 1) \odot_e= \alpha (1 \times \odot_e) \odot_e$: 
  \begin{align*}
(\odot_e \times 1) \odot_e &=~ (\odot_e \times 1) \langle \pi_1, \pi_0 \rangle \odot_e \tag{\ref{ostarmonoideq}} \\
&=~ \langle \pi_1, \pi_0 \rangle (1 \times \odot_e) \odot_e \\
&=~ \langle \pi_1, \pi_0 \rangle \alpha^{-1} \left( \left( \langle 0, 1 \rangle \times \langle 1, 0 \rangle \right) \times \left \langle \langle 1,0 \rangle, 0 \right \rangle \right) \mathsf{D}^3 [e] \tag{\ref{eq99}}\\ 
&=~ \langle \pi_1, \pi_0 \rangle \alpha^{-1} \left( \left( \langle 0, 1 \rangle \times \langle 1, 0 \rangle \right) \times \left \langle \langle 1,0 \rangle, 0 \right \rangle \right) \mathsf{D}\left[ c \mathsf{D}^2 [e] \right] \tag{{\bf[CD.7]}} \\
&=~ \langle \pi_1, \pi_0 \rangle \alpha^{-1} \left( \left( \langle 0, 1 \rangle \times \langle 1, 0 \rangle \right) \times \left \langle \langle 1,0 \rangle, 0 \right \rangle \right) (c \times c) \mathsf{D}^3 [e] \tag{Lemma \ref{linlem}} \\ 
&=~ \langle \pi_1, \pi_0 \rangle \alpha^{-1} \left( \left( \langle 0, 1 \rangle \times \langle 1, 0 \rangle \right) \times \left \langle \langle 1,0 \rangle, 0 \right \rangle \right) (c \times c) c \mathsf{D}^3 [e] \tag{{\bf[CD.7]}} \\
 &=~ \left( \left( \langle 0, 1 \rangle \times \langle 1, 0 \rangle \right) \times \left \langle \langle 1,0 \rangle, 0 \right \rangle \right) \mathsf{D}^3 [e] \\
 &=~ \alpha (1 \times \odot_e) \odot_e \tag{\ref{eq99}} 
\end{align*}
So $(A, \odot_e, u_e)$ is a differential rig, and therefore we conclude that $(A, \odot_e, u_e)$ is a differential exponential rig.
\end{proof} 

In the proof of Theorem \ref{isothm}, we will show that the constructions of Proposition \ref{prop1} and Proposition \ref{prop2} are in fact inverse of each other. The construction from Proposition \ref{prop2} is also compatible with some of the constructions of differential exponential maps in the following sense: 

\begin{lemma}\label{lemmaconstructions2} In a Cartesian differential category:
\begin{enumerate}[{\em (i)}]
\item For the terminal object $\top$, the following equality holds:
\[\odot_{1_\top} = 0 \quad \quad \quad u_{1_\top}= 1_\top\]
\item If $A \xrightarrow{e} A$ and $B \xrightarrow{e^\prime} B$ are differential exponential maps, then the following equality holds for the differential exponential map $A \times B \xrightarrow{e \times e^\prime} A \times B$:
\[\odot_{e \times e^\prime} = c (\odot_e \times \odot_{e^\prime}) \quad \quad \quad u_{e \times e^\prime} = \langle u_e, u_{e^\prime} \rangle\]
\item If $A \xrightarrow{e} A$ is a differential exponential map, then the following equality holds for the differential exponential map $A \times A \xrightarrow{\mathsf{T}(e)} A \times A$:
\[\odot_{\mathsf{T}(e)} = c \mathsf{T}(\odot_e) \quad \quad \quad u_{\mathsf{T}(e)} = \langle u_e, 0 \rangle\]
\end{enumerate}
\end{lemma}
\begin{proof} 
\begin{enumerate}[{\em (i)}]
\item This is automatic by uniqueness of maps into the terminal object. 
\item For the unit, this is mostly straightforward: 
\begin{align*}
u_{e \times e^\prime} &=~0 (e \times e^\prime) \\
&=~\langle 0 e, 0 e^\prime \rangle \\
&=~ \langle u_e, u_{e^\prime} \rangle
\end{align*}
For the multiplication, we have that: 
\begin{align*}
\odot_{e \times e^\prime} &=~\left( \langle 0, 1 \rangle \times \langle 1, 0 \rangle \right) \mathsf{D}^2[e \times e^\prime] \\
&=~\left( \langle 0, 1 \rangle \times \langle 1, 0 \rangle \right) \mathsf{D}\!\left[ c \left( \mathsf{D}[e] \times \mathsf{D}[e^\prime] \right) \right] \tag{Lemma \ref{tanlem}} \\
&=~ \left( \langle 0, 1 \rangle \times \langle 1, 0 \rangle \right) (c \times c) \mathsf{D}\!\left[ \mathsf{D}[e] \times \mathsf{D}[e^\prime] \right] \tag{Lemma \ref{linlem}} \\
&=~ \left( \langle 0, 1 \rangle \times \langle 1, 0 \rangle \right) (c \times c) c \left(\mathsf{D}^2[e] \times \mathsf{D}^2[e^\prime] \right) \tag{Lemma \ref{tanlem}} \\ 
&=~ c \left(\left( \langle 0, 1 \rangle \times \langle 1, 0 \rangle \right) \times \left( \langle 0, 1 \rangle \times \langle 1, 0 \rangle \right) \right) \left(\mathsf{D}^2[e] \times \mathsf{D}^2[e^\prime] \right) \\
&=~ c (\odot_e \times \odot_{e^\prime})
\end{align*}
\item We first show the equality for the unit: 
\begin{align*}
\langle u_e, 0 \rangle &=~\langle 0e, 0 \rangle \\
&=~ 0 e \langle 1 , 0 \rangle \\
 &=~ 0 \langle 1, 0 \rangle \mathsf{T}(e) \tag{Lemma \ref{tanlem}} \\
&=~0 \mathsf{T}(e) \tag{$\langle 1, 0 \rangle$ is additive} \\ 
&=~ u_{\mathsf{T}(e)}
\end{align*}
For the multiplication, we have that: 
\begin{align*}
\odot_{\mathsf{T}(e)} &=~\left( \langle 0, 1 \rangle \times \langle 1, 0 \rangle \right) \mathsf{D}^2[\mathsf{T}(e)] \\
&=~\left( \langle 0, 1 \rangle \times \langle 1, 0 \rangle \right) \mathsf{D}\!\left[ c \mathsf{T}(\mathsf{D}[e]) \right] \tag{Lemma \ref{tanlem}} \\
&=~ \left( \langle 0, 1 \rangle \times \langle 1, 0 \rangle \right) (c \times c) \mathsf{D}\!\left[ \mathsf{T}(\mathsf{D}[e]) \right] \tag{Lemma \ref{linlem}} \\
&=~ \left( \langle 0, 1 \rangle \times \langle 1, 0 \rangle \right) (c \times c) c \mathsf{T} \left( \mathsf{D}^2[e] \right) \tag{Lemma \ref{tanlem}} \\ 
&=~ c \left(\left( \langle 0, 1 \rangle \times \langle 1, 0 \rangle \right) \times \left( \langle 0, 1 \rangle \times \langle 1, 0 \rangle \right) \right) \mathsf{T} \left( \mathsf{D}^2[e] \right) \\
&=~ c \mathsf{T} \left( \left( \langle 0, 1 \rangle \times \langle 1, 0 \rangle \right) \mathsf{D}^2[e] \right) \tag{Lemma \ref{linlem}} \\
&=~ c \mathsf{T}(\odot_e)
\end{align*}
\end{enumerate}
\end{proof} 

\begin{example}\label{ex2} \normalfont Here we apply Proposition \ref{prop2} to the examples of differential exponential maps from the previous section to construct examples of differential exponential rigs in $\mathsf{SMOOTH}$ (which are in fact rings, since $\mathsf{SMOOTH}$ has additive inverses). 
\begin{enumerate}[{\em (i)}]
\item For the exponential function $\mathbb{R} \xrightarrow{e^x} \mathbb{R}$, the induced multiplication is precisely given by the standard multiplication of real numbers, that is:
\[\odot_{e^x}(x,y) = xy\]
and $u_{e^x}(\ast) = 1$. 
\item Applying Lemma \ref{lemmaconstructions2}.ii to the point-wise exponential $\mathbb{R}^{n} \xrightarrow{e^x \times \hdots \times e^x} \mathbb{R}^n$, we obtain the point-wise multiplication of vectors, that is, 
\[\odot_{e^x \times \hdots \times e^x}\left( (x_1, \hdots, x_n), (y_1, \hdots, y_n) \right) = (x_1 y_1, \hdots, x_n y_n)\]
and $u_{e^x \times \hdots \times e^x}(\ast) = (1, \hdots, 1)$. 
\item Applying Lemma \ref{lemmaconstructions2}.iii to the tangent exponential function ${\mathbb{R}^{2} \xrightarrow{\mathsf{T}(e^x)} \mathbb{R}^2}$, we obtain the multiplication:
\[\odot_{\mathsf{T}(e)}\left( (x_1,y_1), (x_2,y_2) \right) = \left( x_1x_2, x_1y_2 + y_1x_2 \right)\]
with unit $u_{\mathsf{T}(e)}(\ast) = (1, 0)$. Observe that this ring structure on $\mathbb{R}^2$ is precisely that of the ring of dual numbers $\mathbb{R}[\varepsilon]$. Indeed, writing $(x,y)$ as $x + y \varepsilon$ with $\varepsilon^2=0$, we see that $\odot_{\mathsf{T}(e)}$ is precisely the multiplication of dual numbers:
 \[(x_1 + y_1\varepsilon)(x_2 +y_2 \varepsilon) = x_1x_2 + (x_1y_2 + y_1x_2)\varepsilon\]
\item For the complex exponential function $\mathbb{R}^2 \xrightarrow{\epsilon} \mathbb{R}^2$, we obtain the multiplication:
\[\odot_{\epsilon}\left( (x_1,y_1), (x_2,y_2) \right) = \left( x_1x_2 - y_1y_2, x_1y_2 + x_2y_1 \right)\]
with unit $u_{\epsilon}(\ast) = (1, 0)$. Unsurprisingly, this ring structure on $\mathbb{R}^2$ is that of complex numbers $\mathbb{C}$. Indeed, writing $(x,y)$ as $x+iy$ with $i^2=-1$, $\odot_{\epsilon}$ gives precisely complex number multiplication: 
\[(x_1+iy_1)(x_2+iy_2) = (x_1x_2 - y_1y_2) + i(x_1y_2 + x_2y_1)\]
\item For the split complex exponential function $\mathbb{R}^2 \xrightarrow{\epsilon^\prime} \mathbb{R}^2$, we obtain the multiplication:
\[\odot_{\epsilon^\prime}\left( (x_1,y_1), (x_2,y_2) \right) = \left( x_1x_2 + y_1y_2, x_1y_2 + x_2y_1 \right)\]
with unit $u_{\epsilon^\prime}(\ast) = (1, 0)$. This ring structure on $\mathbb{R}^2$ is that of split complex numbers $\mathbb{C}^\prime$. Indeed, writing $(x,y)$ as $x+jy$ with $j^2=1$, $\odot_{\epsilon^\prime}$ gives split complex number multiplication: 
\[(x_1+jy_1)(x_2+jy_2) = (x_1x_2 + y_1y_2) + j(x_1y_2 + x_2y_1)\]
\end{enumerate}
\end{example}

\begin{example} \normalfont \label{curveexp2} As briefly discussed in Example \ref{curveexp1}, a differential curve object with solutions to linear systems admits a differential exponential map, and so by applying Proposition \ref{prop2}, a differential curve object is also a differential exponential rig \cite[Corollary 5.26]{cockett2019differential}. As an interesting application of this induced rig structure, in turns out that every differential bundle \cite{cockett2016differential} is a module of the differential curve object \cite[Proposition 5.4]{cockett2019differential}. 
\end{example}

For a Cartesian differential category $\mathbb{X}$, define its category of differential exponential rigs as the category $\mathsf{DES}[\mathbb{X}]$ whose objects are differential exponential rigs $(A,\odot, u, e)$ and where a map $(A,\odot, u, e) \xrightarrow{f} (B,\odot^\prime, u^\prime, e^\prime)$ is a \emph{linear} map $A \xrightarrow{f} B$ in $\mathbb{X}$ such that the following diagrams commutes: 
 \begin{equation}\label{desmap}\begin{gathered} \xymatrixcolsep{2.5pc}\xymatrix{A \ar[d]_-{e} \ar[r]^-{f} & B \ar[d]^-{e^\prime} & \top \ar[r]^-{u} \ar[dr]_-{u^\prime} & A \ar[d]^-{u} & A \times A \ar[d]_-{\odot} \ar[r]^-{f \times f} & B \times B \ar[d]^-{\odot^\prime} \\
A \ar[r]_-{f} & B & & B & A \ar[r]_-{f} & B } \end{gathered}\end{equation}
and where composition and identity maps are as in $\mathbb{X}$. Note that the two right most diagrams above imply that $f$ is a monoid morphism. 

\begin{theorem}\label{isothm} For a Cartesian differential category $\mathbb{X}$, its category of differential exponential maps $\mathsf{DEM}[\mathbb{X}]$ is isomorphic to its category of differential exponential rigs $\mathsf{DES}[\mathbb{X}]$ via the inverse functors $\mathsf{DEM}[\mathbb{X}] \xrightarrow{\mathsf{E}} \mathsf{DES}[\mathbb{X}]$ and $\mathsf{DES}[\mathbb{X}] \xrightarrow{\mathsf{E}^{-1}} \mathsf{DEM}[\mathbb{X}]$ defined respectively as follows:
\begin{align*} \mathsf{E}(A,e) = (A, \odot_e, u_e, e) && \mathsf{E}(f) = f && \mathsf{E}^{-1}(A, \odot, u, e)= (A,e) && \mathsf{E}^{-1}(f) = f 
\end{align*}
\end{theorem}
\begin{proof} We first need to show that $\mathsf{E}$ and $\mathsf{E}^{-1}$ are well-defined functors. By Proposition \ref{prop1}, $\mathsf{E}^{-1}$ is well-defined on objects, while $\mathsf{E}$ is well-defined on objects by Proposition \ref{prop2}. If $f$ is a map in $\mathsf{DES}[\mathbb{X}]$ then by definition it is also a map in $\mathsf{DEM}[\mathbb{X}]$, so $\mathsf{E}^{-1}$ is well-defined on maps. Furthermore, $\mathsf{E}^{-1}$ clearly preserves composition and identity, and therefore $\mathsf{E}^{-1}$ is a well-defined functor. On the other hand, if $(A,e) \xrightarrow{f} (B,e^\prime)$ is a map in $\mathsf{DEM}[\mathbb{X}]$, we must show that $f$ also satisfies the three identities (\ref{desmap}). By definition, one already has that $ef= fe^\prime$, and so it remains to show that $f$ is also a monoid morphism. Recall that since $f$ is linear, $f$ is also additive (Lemma \ref{linlem}). Now we first show that $u_e f= u_{e^\prime}$: 
\begin{align*}
u_e f &=~0 e f \\
&=~ 0 f e^\prime \tag{\ref{demmap}} \\
&=~ 0 e^\prime \tag{$f$ additive} \\
&=~u_{e^\prime} 
\end{align*}
Next we show that $\odot_e f = (f \times f) \odot_{e^\prime}$: 
\begin{align*}
\odot_e f &=~ \left( \langle 0,1 \rangle \times \langle 1, 0 \rangle \right) \mathsf{D}^2[e] f \\
&=~ \left( \langle 0,1 \rangle \times \langle 1, 0 \rangle \right) \mathsf{D}^2[ef] \tag{$f$ linear + Lemma \ref{linlem}} \\
&=~ \left( \langle 0,1 \rangle \times \langle 1, 0 \rangle \right) \mathsf{D}^2[fe^\prime]\tag{\ref{demmap}} \\
&=~ \left( \langle 0,1 \rangle \times \langle 1, 0 \rangle \right) \left((f \times f) \times (f \times f) \right) \mathsf{D}^2[e^\prime]\tag{$f$ linear + Lemma \ref{linlem}} \\
&=~ \left( \langle 0f,f \rangle \times \langle f, 0f \rangle \right)\mathsf{D}^2[e^\prime]\\
&=~ \left( \langle 0,f \rangle \times \langle f, 0 \rangle \right)\mathsf{D}^2[e^\prime] \tag{$f$ additive} \\
&=~(f \times f) \left( \langle 0,1 \rangle \times \langle 1, 0 \rangle \right) \mathsf{D}^2[e^\prime] \\
&=~ (f \times f) \odot_{e^\prime}
\end{align*}
Therefore $f$ is a map in $\mathsf{DES}[\mathbb{X}]$, so $\mathsf{E}$ is well-defined on maps. Clearly $\mathsf{E}$ preserves composition and identity, and therefore $\mathsf{E}$ is also a well-defined functor.

Next we show that $\mathsf{E}$ and $\mathsf{E}^{-1}$ are inverses of each other. Clearly we have both that $\mathsf{E}^{-1}\mathsf{E}(A,e) = (A,e)$ and $\mathsf{E}^{-1}\mathsf{E}(f) = f$. For the other direction, clearly $\mathsf{E}\mathsf{E}^{-1}(f) = f$ and so it remains to show that we also have that $\mathsf{E}\mathsf{E}^{-1}(A, \odot, u, e) = (A, \odot, u, e)$, that is, we must show that $\odot = \odot_e$ and $u_e = u$. Starting with the unit: 
\begin{align*}
u_e &=~0 e \\
&=~u \tag{\ref{des}}
\end{align*}
Next for the multiplication, we first observe that: 
\begin{align*}
\mathsf{D}^2[e] &=~ \mathsf{D}\left[ (e \times 1) \odot \right] \tag{\ref{des}} \\
&=~ \mathsf{T}(e \times 1) \mathsf{D}\left[ \odot \right] \tag{Lemma \ref{tanlem}} \\
&=~ c \left(\mathsf{T}(e) \times \mathsf{T}(1) \right) c \mathsf{D}\left[ \odot \right] \tag{Lemma \ref{tanlem}} \\ 
&=~ c \left(\mathsf{T}(e) \times1 \right) c \mathsf{D}\left[ \odot \right] \tag{$\mathsf{T}$ is a functor} \\ 
&=~c \left(\mathsf{T}(e) \times 1 \right) c(\pi_0 \times \pi_1) \odot + c \left(\mathsf{T}(e) \times 1 \right) c (\pi_1 \times \pi_0) \odot \tag{\ref{odotbilin}} \\
&=~ c \left(\mathsf{T}(e) \times 1 \right) (\pi_0 \times \pi_1) \odot + c \left(\mathsf{T}(e) \times 1 \right) (\pi_1 \times \pi_0) \langle \pi_1, \pi_0 \rangle \odot \\
&=~ c \left(\mathsf{T}(e) \times 1 \right) (\pi_0 \times \pi_1) \odot + c \left(\mathsf{T}(e) \times 1 \right) (\pi_1 \times \pi_0) \odot \tag{\ref{ostarmonoideq}} \\ 
&=~ c (\pi_0 \times \pi_1)(e \times 1) \odot + c (\mathsf{D}[e] \times \pi_0) \odot \tag{Definition of $\mathsf{T}$} \\
&=~ (\pi_0 \times \pi_1)(e \times 1) \odot + c \left( (e \times 1) \times (1 \times 1) \right) (\odot \times \pi_0) \odot \tag{\ref{des}} \\ 
&=~ (\pi_0 \times \pi_1)(e \times 1) \odot + \left( (e \times 1) \times (1 \times 1) \right) c (\odot \times \pi_0) \odot \\ 
&=~ (\pi_0 \times \pi_1)(e \times 1) \odot + \left( (e \times 1) \times (1 \times 1) \right) c \left( (1 \times 1) \times \pi_0 \right) (\odot \times 1) \odot \\ 
&=~ (\pi_0 \times \pi_1)(e \times 1) \odot \\
&~~~+ \left( (e \times 1) \times (1 \times 1) \right) c \left( (1 \times 1) \times \pi_0 \right) \alpha (1\times \odot) \odot \tag{\ref{ostarmonoideq}} \\ 
&=~ (\pi_0 \times \pi_1)(e \times 1) \odot \\
&~~~+ \left( (e \times 1) \times (1 \times 1) \right) \langle \pi_0, (\pi_1 \times \pi_0) \rangle (1 \times \langle \pi_1, \pi_0 \rangle) (1\times \odot) \odot \\
&=~ (\pi_0 \times \pi_1)(e \times 1) \odot \\
&~~~+ \left( (e \times 1) \times (1 \times 1) \right) \langle \pi_0\pi_0, (\pi_1 \times \pi_0) \rangle (1 \times \odot) \odot \tag{\ref{ostarmonoideq}} \\ 
&=~ (\pi_0 \times \pi_1)(e \times 1) \odot \\
&~~~+ \left( (e \times 1) \times (1 \times 1) \right) \langle \pi_0\pi_0, (\pi_1 \times \pi_0) \rangle \alpha^{-1} (\odot \times 1) \odot \tag{\ref{ostarmonoideq}} \\ 
&=~ (\pi_0 \times \pi_1)(e \times 1) \odot + \left( (e \times 1) \times (1 \times 1) \right) \left( (1 \times 1) \times \pi_0 \right) (\odot \times 1) \odot \\ 
&=~ (\pi_0 \times \pi_1)(e \times 1) \odot + \left( (e \times 1) \times \pi_0 \right) (\odot \times 1) \odot \\ 
&=~ (\pi_0 \times \pi_1) \mathsf{D}[e] + (\mathsf{D}[e] \times \pi_0) \odot 
\end{align*}
So we have that: 
\begin{equation}\label{D2e}\begin{gathered} \mathsf{D}^2[e] = (\pi_0 \times \pi_1) \mathsf{D}[e] + (\mathsf{D}[e] \times \pi_0) \odot 
 \end{gathered}\end{equation}
 Using the above identity, we can easily show that $\odot_e = \odot$: 
\begin{align*}
\odot_e &=~ \left( \langle 0,1 \rangle \times \langle 1, 0 \rangle \right) \mathsf{D}^2[e] \\
&=~\left( \langle 0,1 \rangle \times \langle 1, 0 \rangle \right) (\pi_0 \times \pi_1) \mathsf{D}[e] + \left( \langle 0,1 \rangle \times \langle 1, 0 \rangle \right) (\mathsf{D}[e] \times \pi_0) \odot \tag{\ref{D2e}} \\
&=~(0 \times 0) \mathsf{D}[e] + (1 \times 1) \odot \tag{\ref{dem}} \\
&=~0 + \odot \tag{\bf[CD.2]} \\
&=~\odot 
\end{align*}
So $(A, \odot, u, e)= (A, \odot_e, u_e, e)$. Therefore we conclude that $\mathsf{E}$ and $\mathsf{E}^{-1}$ are inverse functors and that $\mathsf{DEM}[\mathbb{X}]$ is isomorphic to $\mathsf{DES}[\mathbb{X}]$. 
\end{proof} 

We conclude this section with the observation that as an immediate consequence of both Theorem \ref{isothm} and Lemma \ref{lemmaconstructions2}: the category of differential exponential rigs is a Cartesian tangent category and that it is isomorphic as a Cartesian tangent category to the category of differential exponential maps. 

\begin{proposition}\label{DEStan} For a Cartesian differential category $\mathbb{X}$, $\mathsf{DES}[\mathbb{X}]$ has finite products where the terminal object is $(\top, 0, 1_\top, 1_\top)$, and where the product of $(A, \odot, u, e)$ and $(B, \odot^\prime, u^\prime, e^\prime)$ is $(A \times B, c (\odot \times \odot^\prime), \langle u, u^\prime \rangle, e \times e^\prime)$ 
with the obvious projection maps. $\mathsf{DES}[\mathbb{X}]$ is also a Cartesian tangent category where the tangent functor $\mathsf{T}: \mathsf{DES}[\mathbb{X}] \to \mathsf{DES}[\mathbb{X}]$ is defined as follows:
\[\mathsf{T}(A, \odot, u, e) := (A \times A, c \mathsf{T}(\odot), \langle u, 0 \rangle, \mathsf{T}(e)) \quad \quad \quad \mathsf{T}(f) = f \times f\]
and where the remaining tangent structure is the same as for $\mathbb{X}$ (which can be found in {\cite[Proposition 4.7]{cockett2014differential}}). Furthermore, both $\mathsf{DEM}[\mathbb{X}] \xrightarrow{\mathsf{E}} \mathsf{DES}[\mathbb{X}]$ and $\mathsf{DEM}[\mathbb{X}] \xrightarrow{\mathsf{E}^{-1}} \mathsf{DES}[\mathbb{X}]$ preserve the Cartesian tangent structure strictly. 
\end{proposition}

\section{Solutions to Dynamical Systems}\label{DYNsec}

As introduced in \cite[Section 5]{cockett2017connections}, ordinary differential equations in a Cartesian differential category are described as \emph{dynamical systems}, while solutions for these differential equations are described as morphisms between these dynamical systems. In the classical case, the exponential function $e^x$ can be defined as the unique solution to the initial value problem $f^\prime(x) = f(x)$ with $f(0)=1$. In this section, we explain how differential exponential maps provide solutions to certain (parametrized) dynamical systems and conversely how one can obtain a differential exponential map if one assumes that solutions are unique. See \cite[Section 5]{cockett2019differential} for more applications of differential exponential maps in regards to solving differential equations. 

We note that in \cite{cockett2017connections,cockett2019differential}, dynamical systems were defined for tangent categories and thus involves the tangent functor. Here we present the resulting definition specific to Cartesian differential categories, where dynamical systems can be described in terms of the differential combinator. 

\begin{definition} In a Cartesian differential category, 
\begin{enumerate}[{\em (i)}]
\item A \textbf{dynamical system} \cite[Definition 5.15]{cockett2017connections} is a triple $(A, a_0, a_1)$ consisting of an object $A$, a point $\top \xrightarrow{a_0} A$, and an endomorphism $A \xrightarrow{a_1} A$;
\item A \textbf{morphism of dynamical systems} $(A, a_0, a_1) \xrightarrow{f} (A^\prime, a^\prime_0, a^\prime_1)$ is a map $A \xrightarrow{f} A^\prime$ such that the following diagram commutes: 
 \begin{equation}\label{solution2}\begin{gathered} \xymatrixcolsep{5pc}\xymatrix{ & \top \ar[dr]^-{a^\prime_0} \ar[dl]_-{a_0} \\
   A \ar[rr]^-{f} \ar[d]_-{\left \langle 1, a_1 \right \rangle} & & A^\prime \ar[d]^-{a^\prime_1} \\
   A \times A \ar[rr]_-{\mathsf{D}[f]} && A^\prime
    } \end{gathered}\end{equation}
    \item If $(A, a_0, a_1) \xrightarrow{f} (A^\prime, a^\prime_0, a^\prime_1)$ is a morphism of dynamical systems, we say that $f$ is an \textbf{$(A, a_0, a_1)$-solution} of $(A^\prime, a^\prime_0, a^\prime_1)$. 
\end{enumerate}
\end{definition}

\begin{example}\label{smoothdynex1} \normalfont A dynamical system in $\mathsf{SMOOTH}$ can be seen as a triple $(\mathbb{R}^n, \vec a, F)$ where $\mathbb{R}^n \xrightarrow{F} \mathbb{R}^n$ is a smooth function and a point ${\vec a \in \mathbb{R}^n}$. If $(\mathbb{R}^n, \vec a, F)$ and $(\mathbb{R}^m, \vec b, G)$ are dynamical systems, then a $(\mathbb{R}^n, \vec a, F)$-solution of $(\mathbb{R}^m, \vec b, G)$ is a smooth function $\mathbb{R}^n \xrightarrow{H} \mathbb{R}^m$ such that:
\begin{align*}
H(\vec a) = \vec b && \mathsf{D}[H](\vec x, F(\vec x)) = G(H(\vec x))
\end{align*}
which amounts to saying that $H$ is a solution to a certain (large) system of differential equations. For a more explicit example, let $\mathbb{R} \xrightarrow{\overline{c}} \mathbb{R}$ be a non-zero constant function $\overline{c}(x) = c$, $c\neq 0$, and define the smooth function $\mathbb{R} \xrightarrow{g} \mathbb{R}$ as $g(x) = -rx$ for some $r \in \mathbb{R}$. Then a $(\mathbb{R}, 0, \overline{c})$-solution of $(\mathbb{R}, a, g)$ is a smooth function $\mathbb{R} \xrightarrow{f} \mathbb{R}$ such that:
\begin{align*}
f(0) = a && \mathsf{D}[f](x, c) = - r f(x) 
\end{align*}
which is equivalent to saying that $f$ is a solution to the following linear differential equation: 
\begin{align*}
f(0) = b && f^\prime(x) + \lambda f(x) = 0  
\end{align*}
where $\lambda = \frac{r}{c}$. See \cite{cockett2017connections,cockett2019differential} for more details and intuition on dynamical systems. 
\end{example}

For any differential rig $(A, \odot, u)$, there is a canonical dynamical system $(A, 0, \overline{u})$ where $A \xrightarrow{\overline{u}} A$ is defined as follows: 
\begin{equation}\label{ubar}\begin{gathered} \overline{u} := \xymatrixcolsep{5pc}\xymatrix{ A \ar[r]^-{0} & \top \ar[r]^-{u} & A
 } \end{gathered}\end{equation}
and we can ask that $(A, 0, \overline{u})$-solutions be compatible with the multiplication. 

\begin{definition} Let $(A, \odot, u)$ be a differential rig and $(A, a_0, a_1)$ a dynamical system. An \textbf{$(A, \odot, u)$-solution} of $(A, a_0, a_1)$ is a map $A \xrightarrow{f} A$ such that the following diagrams commute: 
 \begin{equation}\label{solution4}\begin{gathered} \xymatrixcolsep{2.5pc}\xymatrixrowsep{1.5pc}\xymatrix{ \top \ar[r]^-{0} \ar[dr]_-{a_0} & A \ar[d]^-{f} & A \times A \ar[r]^-{f \times 1} \ar[ddrr]_-{\mathsf{D}[f]} & A \times A\ar[r]^-{a_1 \times 1} & A \times A \ar[dd]^-{\odot} \\   
 & A \\
 & &&& A } \end{gathered}\end{equation}
\end{definition}

\begin{lemma}\label{sollem1} Let $(A, \odot, u)$ be a differential rig, $(A, a_0, a_1)$ a dynamical system, ${A \xrightarrow{f} A}$ an endomorphism. Then $f$ is an $(A, \odot, u)$-solution of $(A, a_0, a_1)$ if and only if $f$ is an $(A, 0, \overline{u})$-solution of $(A, a_0, a_1)$ such that the following diagram commutes: 
 \begin{equation}\label{solution4.1}\begin{gathered} \xymatrixcolsep{5pc}\xymatrixrowsep{1.5pc}\xymatrix{ A \times A \ar[r]^-{\langle 1, \overline{u} \rangle \times 1} \ar[ddrr]_-{\mathsf{D}[f]} & (A \times A) \times A\ar[r]^-{\mathsf{D}[f] \times 1} & A \times A \ar[dd]^-{\odot} \\   
\\
 && A } \end{gathered}\end{equation}
\end{lemma}
\begin{proof} Suppose that $f$ is an $(A, 0, \overline{u})$-solution of $(A, a_0, a_1)$. We first show that $f$ is also a $(A, 0, \overline{u})$-solution of $(A, a_0, a_1)$. The top triangle of (\ref{solution2}) is precisely the left diagram of (\ref{solution4}). So it remains to show that $f$ also satisfies the bottom square of (\ref{solution2}): 
\begin{align*}
\langle 1, \overline{u} \rangle \mathsf{D}[f] &=~ \langle 1, \overline{u} \rangle (f \times 1)(a_1 \times 1)\odot \tag{\ref{solution4}} \\
&=~ f a_1 \langle 1, \overline{u} \rangle \odot \\
&=~ f a_1 \langle 1, 0u \rangle \odot \\
&=~ f a_1 \tag{\ref{ostarmonoideq}} 
\end{align*}
So $f$ is an $(A, 0, \overline{u})$-solution of $(A, a_0, a_1)$. As an immediate consequence, it follows that:
\begin{align*}
(\langle 1, \overline{u} \rangle \times 1)(\mathsf{D}[f] \times 1) \odot &=~(f \times 1)(a_1 \times 1) \odot \\
&=~ \mathsf{D}[f]  \tag{\ref{solution4}} 
\end{align*}
So $f$ also satisfies (\ref{solution4.1}). Conversely, suppose that $f$ is an $(A, 0, \overline{u})$-solution of $(A, a_0, a_1)$ which satisfies (\ref{solution4.1}). We must show that $f$ satisfies the two diagrams of (\ref{solution4}). As before, the left diagram of (\ref{solution4}) is the same as the top triangle of (\ref{solution2}). So it remains to show that $f$ also satisfies the right diagram of (\ref{solution4}):
\begin{align*}
\mathsf{D}[f] &=~(\langle 1, \overline{u} \rangle \times 1)(\mathsf{D}[f] \times 1) \odot   \tag{\ref{solution4.1}} \\
&=~(f \times 1)(a_1 \times 1) \odot \tag{\ref{solution2}} 
\end{align*}
So we conclude that $f$ is an $(A, \odot, u)$-solution of $(A, a_0, a_1)$. 
\end{proof} 

\begin{example}\label{smoothdynex2} \normalfont In $\mathsf{SMOOTH}$, consider the differential rig induced from the exponential function $e^x$ as defined in Example \ref{ex2}.i, that is, $\mathbb{R}$ with the standard multiplication of real numbers. Its canonical dynamical system as defined above is $(\mathbb{R}, 0, \overline{u_{e^x}})$ since $\overline{u_{e^x}}(x) = 1$. A $(\mathbb{R}, \odot_{e^x}, u_{e^x})$-solution of a dynamical system $(\mathbb{R}, a, g)$, where $a \in \mathbb{R}$ and $\mathbb{R} \xrightarrow{g} \mathbb{R}$, is a smooth function $\mathbb{R} \xrightarrow{f} \mathbb{R}$ such that: 
\begin{align*}
f(0) = a && \mathsf{D}[f](x,y) = f^\prime(x)y = g(f(x)) y 
\end{align*}
By Lemma \ref{sollem1}, setting $y = 1$, we see that $f$ is also a solution to the differential equation ${f^\prime(x) = f(g(x))}$ with initial value $f(0) = a$, or in other words, $f$ is a $(\mathbb{R}, 0, \overline{u_{e^x}})$-solution of $(\mathbb{R}, a, g)$ such that $f$ satisfies (\ref{solution4.1}):
\[\mathsf{D}[f](x,y) = \mathsf{D}[f](x,1)y\]
In fact, note that every arbitrary smooth function $\mathbb{R} \xrightarrow{f} \mathbb{R}$ satisfies (\ref{solution4.1}) since:
\[ \mathsf{D}[f](x,y) = f^\prime(x)y = \mathsf{D}[f](x,1)y \]
Then in this case, every $(\mathbb{R}, 0, \overline{u_{e^x}})$-solution of $(\mathbb{R}, a, g)$ is also a $(\mathbb{R}, \odot_{e^x}, u_{e^x})$-solution of $(\mathbb{R}, a, g)$.
\end{example}

For a differential exponential rig, its differential exponential map is the solution to the dynamical system which generalizes the initial value problem $f^\prime(x) = f(x)$ with $f(0) = 1$. 

\begin{proposition}\label{esolprop} Let $(A, \odot, u, e)$ be a differential exponential rig. Then the differential exponential map $e$ is an $(A, \odot, u)$-solution of the dynamical system $(A, u, 1)$, and therefore $e$ is also a $(A, 0, \overline{u})$-solution of $(A, u, 1)$. 
\end{proposition}
\begin{proof} The left diagram of (\ref{solution4}) is precisely the middle diagram of (\ref{des}) that $0e=u$. While the right diagram of (\ref{solution4}) is precisely the left diagram of (\ref{des}) that $\mathsf{D}[e] = (e \times 1) \odot$. Therefore, $e$ is an $(A, \odot, u)$-solution of $(A, u, 1)$. By Lemma \ref{sollem1}, it follows that $e$ is also a $(A, 0, \overline{u})$-solution of $(A, u, 1)$. 
\end{proof} 

\begin{example} \normalfont In $\mathsf{SMOOTH}$, $\mathbb{R} \xrightarrow{e^x} \mathbb{R}$ is an $(\mathbb{R}, \odot_{e^x}, u_{e^x})$-solution of $(\mathbb{R}, u, 1)$. In other words, $e^x$ is a solution to the following initial value problem: 
\[ \mathsf{D}[f](x,y) = f(x)y \quad \quad \quad f(0) = 1 \]
In fact, $e^x$ is the unique $(\mathbb{R}, \odot_{e^x}, u_{e^x})$-solution of $(\mathbb{R}, u, 1)$. By Proposition \ref{esolprop}, setting $y=1$, $e^x$ is also a $(\mathbb{R}, 0, \overline{u_{e^x}})$-solution of $(\mathbb{R}, u, 1)$, that is, $e^x$ is the unique solution to the initial value problem $f^\prime(x) = f(x)$ with $f(0) = 1$. 
\end{example}

Note, however, that Proposition \ref{esolprop} does not say that a differential exponential map is a unique solution. Indeed, in an arbitrary Cartesian differential category, solutions of dynamical systems need not be unique, and so it is possible that for a differential rig $(A, \odot, u)$, there are multiple $(A, 0, \overline{u})$-solutions of $(A, u, 1)$. Furthermore, as discussed in Section \ref{dessec}, a $(A, 0, \overline{u})$-solution of $(A, u, 1)$ is not necessarily a differential exponential map since $\oplus e = (e \times e) \odot$ does not necessarily hold. That said, as we will see in Proposition \ref{uniqueprop}, with the extra assumption that solutions be unique, then it follows that a $(A, 0, \overline{u})$-solution of $(A, u, 1)$ is, in this case, a differential exponential map. To do so, we first discuss the notion of solutions of \emph{parametrized} dynamical systems. 

\begin{definition} In a Cartesian differential category, 
\begin{enumerate}[{\em (i)}]
\item A \textbf{parametrized dynamical system} (over $X$ or in context $X$) is a triple $(B, b_0, b_1)$ consisting of an object $B$, a map $X \xrightarrow{b_0} B$, and an endomorphism $B \xrightarrow{b_1} B$;
\item If $(A, a_0, a_1)$ is a dynamical system and $(B, b_0, b_1)$ a parametrized dynamical system, then a \textbf{parametrized $(A, a_0, a_1)$-solution of $(B, b_0, b_1)$} is a map $A \times X \xrightarrow{f} B$ such that the following diagram commutes:
 \begin{equation}\label{solution1}\begin{gathered} \xymatrixcolsep{5pc}\xymatrix{ & X \ar[dr]^-{b_0} \ar[dl]_-{\langle 0a_0, 1 \rangle} \\
   A \times X \ar[rr]^-{f} \ar[d]_-{\left \langle \langle \pi_0, \pi_1 \rangle, \langle \pi_1 a_1, 0 \rangle \right \rangle} & & B \ar[d]^-{b_1} \\
   (A \times X) \times (A \times X) \ar[rr]_-{\mathsf{D}[f]} && B
    } \end{gathered}\end{equation}
 \item For a dynamical system $(A, a_0, a_1)$, an endomorphism $B \xrightarrow{b_1} B$ is said to be \textbf{$(A, a_0, a_1)$-complete} if for every map $X \xrightarrow{b_0} B$, there is an $(A, a_0, a_1)$-solution of the parametrized dynamical system $(B, b_0, b_1)$. 
\end{enumerate}
\end{definition}

Note that dynamical systems can be described as parametrized dynamical systems over the terminal object $\top$ and in this case (\ref{solution2}) is the same as (\ref{solution1}), modulo the isomorphism $A \cong A \times \top$. 

\begin{example}\label{smoothdynex3} \normalfont In $\mathsf{SMOOTH}$, a parametrized dynamical system is simply a triple $(\mathbb{R}^m, K, G)$ with smooth functions $\mathbb{R}^k \xrightarrow{K} \mathbb{R}^m$ and $\mathbb{R}^m \xrightarrow{G} \mathbb{R}^m$. If $(\mathbb{R}^n, \vec a, H)$ is a dynamical system, then a parametrized $(\mathbb{R}^n, \vec a, H)$-solution of $(\mathbb{R}^m, K, G)$ is a smooth function $\mathbb{R}^n \times \mathbb{R}^k \xrightarrow{F} \mathbb{R}^m$ such that: 
\begin{align*}
F(\vec a, \vec y) = K(\vec y) && \mathsf{D}[F] \left( (\vec x, \vec y, K(\vec y), \vec 0 \right) = G(F(\vec x, \vec y)) 
\end{align*}
As a particular example, let $m=1$ with smooth functions $\mathbb{R}^k \xrightarrow{h} \mathbb{R}$ and $\mathbb{R} \xrightarrow{g} \mathbb{R}$. A parametrized $(\mathbb{R}, 0, \overline{u_{e^x}})$-solution of $(\mathbb{R}, h, g)$ is a smooth function $\mathbb{R} \times \mathbb{R}^k \xrightarrow{f} \mathbb{R}$ such that: 
\begin{align*}
f(0, y_1, \hdots, y_k) &= h(y_1, \hdots, y_k)\\
 \frac{\partial f(t_0, t_1, \hdots, t_n)}{\partial t_0} (x, y_1, \hdots, y_k) &= g\left(f(x, y_1, \hdots, y_k)\right)
\end{align*}
\end{example}

As before, in the case of a differential rig, one can also ask for parametrized solutions to be compatible with the rig multiplication. 

\begin{definition} Let $(A, \odot, u)$ be a differential rig and $(A, a_0, a_1)$ be a para\-metrized dynamical system over $X$. 
\begin{enumerate}[{\em (i)}]
\item A \textbf{parametrized $(A, \odot, u)$-solution} of $(A, a_0, a_1)$ is a map $A \times X \xrightarrow{f} A$ such that the following diagrams commute: 
 \begin{equation}\label{solution5}\begin{gathered} \xymatrixcolsep{1.5pc}\xymatrixrowsep{1.5pc}\xymatrix{ X \ar[r]^-{\langle 0, 1 \rangle} \ar[dr]_-{a_0} & A \times X \ar[d]^-{f} & (A \times X) \times A \ar[r]^-{f \times 1} \ar[dd]_-{(1 \times 1) \times \langle 1, 0 \rangle} & A \times A\ar[r]^-{a_1 \times 1} & A \times A \ar[dd]^-{\odot} \\   
 & A \\
 & & (A \times X) \times (A \times X) \ar[rr]_-{\mathsf{D}[f]} && A }\end{gathered}\end{equation}
 \item An endomorphism $A \xrightarrow{a_1} A$ is \textbf{$(A, \odot, u)$-complete} if for every map $X \xrightarrow{a_0} A$, there is a parametrized $(A, \odot, u)$-solution of the parametrized dynamical system $(A, a_0, a_1)$. 
 \end{enumerate}
\end{definition}

\begin{lemma}\label{sollem2} Let $(A, \odot, u)$ be a differential rig.
\begin{enumerate}[{\em (i)}]
\item Let $(A, a_0, a_1)$ be a parametrized dynamical system over $X$. Then a map ${A \times X \xrightarrow{f} A}$ is a parametrized $(A, \odot, u)$-solution of $(A, a_0, a_1)$ if and only if $f$ is a parametrized $(A, 0, \overline{u})$-solution of $(A, a_0, a_1)$ such that the following diagram commutes: 
 \begin{equation}\label{solution5.1}\begin{gathered} \xymatrixcolsep{2.5pc}\xymatrixrowsep{1.5pc}\xymatrix{ (A \times X) \times A \ar[rr]^-{\left\langle 1 \times 1, \overline{u} \times 0 \right\rangle \times 1} \ar[dd]_-{(1 \times 1) \times \langle 1, 0 \rangle} && \left( (A \times X) \times (A \times X) \right) \times A\ar[d]^-{\mathsf{D}[f] \times 1} \\
 && A \times A \ar[d]^-{\odot} \\   
(A \times X) \times (A \times X) \ar[rr]_-{\mathsf{D}[f]} && A }\end{gathered}\end{equation}
\item If an endomorphism $A \xrightarrow{a_1} A$ is $(A, \odot, u)$-complete then $a_1$ is $(A, 0, \overline{u})$-complete.
\end{enumerate}
\end{lemma}
\begin{proof} Note that $(i)$ is a generalization of Lemma \ref{sollem1} and it is proved by similar calculations. Now suppose that an endomorphism $A \xrightarrow{a_1} A$ is $(A, \odot, u)$-complete. Then for every map $X \xrightarrow{a_0} A$, there is a parametrized $(A, \odot, u)$-solution of $(A, a_0, a_1)$, which is therefore also a parametrized $(A, 0, \overline{u})$-solution of $(A, a_0, a_1)$. So we conclude that $a_0$ is also $(A, 0, \overline{u})$-complete. 
\end{proof} 

\begin{example}\label{smoothdynex4} \normalfont In $\mathsf{SMOOTH}$, let $(\mathbb{R}, h, g)$ be a parametrized dynamical system over $\mathbb{R}^k$ with smooth functions $\mathbb{R}^k \xrightarrow{h} \mathbb{R}$ and $\mathbb{R} \xrightarrow{g} \mathbb{R}$. Then a parametrized $(\mathbb{R}, \odot_{e^x}, u_{e^x})$-solution of $(\mathbb{R}, h, g)$ is a smooth function $\mathbb{R} \times \mathbb{R}^k \xrightarrow{f} \mathbb{R}$ such that: 
\begin{align*}
f(0, y_1, \hdots, y_k) &= h(y_1, \hdots, y_k) \\
 \frac{\partial f(t_0, t_1, \hdots, t_n)}{\partial t_0} (x, y_1, \hdots, y_k)z &= g\left(f(x, y_1, \hdots, y_k)\right)z 
\end{align*}
By Lemma \ref{sollem2}, setting $z=1$, we have that $f$ is also a parametrized $(\mathbb{R}, 0, \overline{u_{e^x}})$-solution of $(\mathbb{R}, h, g)$. 
\end{example}

We wish to show that for a differential exponential rig, a certain class of linear endomorphisms are complete, that is, always have a parametrized solution. This corresponds to the fact that in the classical case, $e^x$ allows one to solve all initial value problems $f^\prime(x) = af(x)$ with initial condition $f(0) = b$ for any constants $a$ and $b$. In this case, the solution to this first order linear differential equation is $f(x) = e^{ax} b$. So the linear function which scalar multiplies by $a$ is complete. This can be generalized to the multivariable case, which is of interest in control systems theory \cite[Chapter 5]{astrom2010feedback}.  

\begin{definition} Let $(A, \odot, u)$ be a differential rig. For a point $\top \xrightarrow{a} A$, define the endomorphism $A \xrightarrow{\odot^a} A$ as multiplication by $a$, that is:
 \[ \odot^a := \xymatrixcolsep{5pc}\xymatrix{A \ar[r]^-{\langle 1, 0a\rangle} & A \times A \ar[r]^-{\odot} & A  
 } \]
\end{definition}

As seen in the example below, the map $\odot^a$ is simply multiplication on the right by the point $a$, which by commutativity of $\odot$ is the same as multiplying on the left. 
In the case of a differential exponential rig, we will show that $\odot^a$ is complete.  

\begin{example} \normalfont In $\mathsf{SMOOTH}$, for $a \in \mathbb{R}$, $\odot_{e^x}^a(x) = xa$. 
\end{example}

\begin{lemma}\label{flatlem} Let $(A, \odot, u)$ be a differential rig.
\begin{enumerate}[{\em (i)}]
\item $\odot^0 =0$ and $\odot^u = 1$; 
\item For every pair of points $\top \xrightarrow{a} A$ and $\top \xrightarrow{b} A$, $\odot^{a+b} = \odot^a + \odot^b$ and $\odot^a \odot^b = \odot^b \odot^a$;
\item For every point $\top \xrightarrow{a} A$, $\odot^a$ is linear and the following diagrams commute: 
   \begin{equation}\label{aflat}\begin{gathered} \xymatrixcolsep{5pc}\xymatrix{ A \times A \ar[r]^-{\odot} \ar[d]_-{1 \times \odot^a} & A \ar[d]^-{\odot^a} & \top \ar[dr]_-{a} \ar[r]^-{u} & A \ar[d]^-{\odot^a} \\
   A \times A \ar[r]_-{\odot} & A & & A 
    } \end{gathered}\end{equation}
\item For an endomorphism $A \xrightarrow{f} A$, $\odot^{uf} = f$ if and only if the following diagram, 
   \begin{equation}\label{modulemap}\begin{gathered} \xymatrixcolsep{5pc}\xymatrix{ A \times A \ar[r]^-{\odot} \ar[d]_-{1 \times f} & A \ar[d]^-{f} \\
   A \times A \ar[r]_-{\odot} & A 
    } \end{gathered}\end{equation}
\item For every point $\top \xrightarrow{a} A$, $\odot^{u\odot^a} = \odot^a$. 
\end{enumerate}
\end{lemma}
\begin{proof} We leave these as an exercise for the reader to check for themselves. 
\end{proof} 

\begin{proposition}\label{comlinprop} Let $(A, \odot, u, e)$ be a differential exponential rig. Then for every point ${\top \xrightarrow{a} A}$, the endomorphism $A \xrightarrow{\odot^a} A$ is $(A, \odot, u)$-complete. 
\end{proposition}
\begin{proof} Let $\top \xrightarrow{a} A$ be point and let $X \xrightarrow{a_0} A$ be an arbitrary map. Then $(A, a_0, \odot^a)$ is a parametrized dynamical system over $X$. Now consider the following composite: 
 \[ \xymatrixcolsep{5pc}\xymatrix{A \times X \ar[r]^-{\odot^a \times a_0} & A \times A \ar[r]^-{\mathsf{D}[e]} & A
 } \]
which by (\ref{des}) is equal to:
\[ (\odot^a \times a_0) \mathsf{D}[e] = (\odot^a \times a_0)(e \times 1)\odot \]
We need to show both equalities of (\ref{solution5}). Starting with the left identity of (\ref{solution5}), since by Lemma \ref{flatlem} $\odot^a$ is linear it is also additive, we have that:
\begin{align*}
 \langle 0, 1 \rangle (\odot^a \times a_0) \mathsf{D}[e] &=~ \langle 0 \odot^a, a_0 \rangle \mathsf{D}[e] \\
 &=~ \langle 0, a_0 \rangle \mathsf{D}[e] \tag{$\odot^a$ is additive} \\ 
&=~ a_0 \langle 0,1 \rangle \mathsf{D}[e] \\
&=~ a_0 \tag{\ref{dem}}
\end{align*}
And for the right identity of (\ref{solution5}): 
\begin{align*}
&\left( (1 \times 1) \times \langle 1, 0 \rangle \right) \mathsf{D}\left[(\odot^a \times a_0) \mathsf{D}[e] \right] \\
&=~\left( (1 \times 1) \times \langle 1, 0 \rangle \right) \mathsf{T}(\odot^a \times a_0) \mathsf{D}^2\left[ e \right] \tag{Lemma \ref{tanlem}} \\
&=~ \left( (1 \times 1) \times \langle 1, 0 \rangle \right) \mathsf{T}(\odot^a \times a_0) c \mathsf{D}^2\left[ e \right] \tag{\bf[CD.7]} \\
&=~ \left( (1 \times 1) \times \langle 1, 0 \rangle \right) c \left( \mathsf{T}(\odot^a) \times \mathsf{T}(a_0) \right) \mathsf{D}^2\left[ e \right] \tag{Lemma \ref{tanlem}} \\
&=~ \langle \pi_0 \times 1, \pi_0 \pi_1 \rangle \left( (1 \times 1) \times \langle 1, 0 \rangle \right) \left( \mathsf{T}(\odot^a) \times \mathsf{T}(a_0) \right) \mathsf{D}^2\left[ e \right] \\
&=~ \langle \pi_0 \times 1, \pi_0 \pi_1 \rangle \left( \mathsf{T}(\odot^a) \times a_0 \right) \left( (1 \times 1) \times \langle 1, 0 \rangle \right) \mathsf{D}^2\left[ e \right] \tag{Lemma \ref{tanlem}} \\
&=~ \langle \pi_0 \times 1, \pi_0 \pi_1 \rangle \left( \mathsf{T}(\odot^a) \times a_0 \right) \left( (1 \times 1) \times \langle 1, 0 \rangle \right) (\pi_0 \times \pi_1) \mathsf{D}[e] \\
&~~~+ \langle \pi_0 \times 1, \pi_0 \pi_1 \rangle \left( \mathsf{T}(\odot^a) \times a_0 \right)\left( (1 \times 1) \times \langle 1, 0 \rangle \right) (\mathsf{D}[e] \times \pi_0) \odot \tag{\ref{D2e}}\\
&=~ \langle \pi_0 \times 1, \pi_0 \pi_1 \rangle \left( \mathsf{T}(\odot^a) \times a_0 \right) (\pi_0 \times 0) \mathsf{D}[e] \\
&~~~+ \langle \pi_0 \times 1, \pi_0 \pi_1 \rangle \left( \mathsf{T}(\odot^a) \times a_0 \right) (\mathsf{D}[e] \times 1) \odot \\
&=~ 0 + \langle \pi_0 \times 1, \pi_0 \pi_1 \rangle \left( \mathsf{T}(\odot^a) \times a_0 \right) (\mathsf{D}[e] \times 1) \odot \tag{{\bf[CD.2]}} \\
&=~ \langle \pi_0 \times 1, \pi_0 \pi_1 \rangle \left((\odot^a \times \odot^a) \times a_0 \right)((e \times 1) \times 1) (\odot \times 1) \odot \tag{Lemma \ref{flatlem} + Lemma \ref{tanlem} + (\ref{des})} \\
&=~ ((\odot^a \times a_0) \times 1) ((e\times 1) \times 1) (\odot \times 1)(\odot^a \times 1) \odot \tag{Lemma \ref{flatlem}} \\
&=~((\odot^a \times a_0) \times 1) (\mathsf{D}[e] \times 1)(\odot^a \times 1) \odot \tag{\ref{des}} 
\end{align*}
So we conclude that $(\odot^a \times a_0) \mathsf{D}[e]$ is an $(A, \odot, u)$-solution of $(A, a_0, \odot^a)$ and therefore that $\odot^a$ is $(A, \odot, u)$-complete. 
\end{proof}

\begin{example} \normalfont In $\mathsf{SMOOTH}$, for every point $a \in \mathbb{R}$ and a smooth function $\mathbb{R}^m \xrightarrow{h} \mathbb{R}$, the smooth function $\mathbb{R} \times \mathbb{R}^m \xrightarrow{f} \mathbb{R}$ defined as $f(x, \vec y) = e^{x a}h(\vec y)$ is a parametrized $(\mathbb{R}, \odot_{e^x}, u_{e^x})$-solution of the parametrized dynamical system $(\mathbb{R}, \odot^a_{e^x}, h)$, that is:  
\begin{align*}
 \frac{\partial f(t, \vec u)}{\partial t}(x, \vec y) z = a f(x, \vec y) z && f(0, \vec y) = h(\vec y)
\end{align*}
Setting $z=1$, it follows that $f$ is also a solution to the differential equation:
\begin{align*}
\frac{\partial f(t, \vec u)}{\partial t}(x, \vec y) = a f(x, \vec y) &&  f(0, \vec y) = h(\vec y)
\end{align*}
\end{example}

We would now like to prove the ``converse'' of Proposition \ref{esolprop}, that is, we would like to obtain differential exponential maps as solutions to certain dynamical systems. To do so, we will require the extra assumption that solutions are unique, which is necessary to prove that $\oplus e= (e \times e) \odot$. 

\begin{proposition}\label{uniqueprop} Let $(A, \odot, u)$ be a differential rig and suppose that:
\begin{enumerate}[{\em (i)}]
\item Parametrized $(A, \odot, u)$-solutions are unique if they exists, that is, if both $f$ and $g$ are parametrized $(A, \odot, u)$-solutions of a parametrized dynamical system $(A, a_0, a_1)$, then $f= g$
\item The dynamical system $(A, u, 1)$ has an $(A, \odot, u)$-solution $e$. 
\end{enumerate}
Then $(A, \odot, u, e)$ is a differential exponential rig. 
\end{proposition} 
\begin{proof} We must show that $e$ satisfies the three identities of (\ref{des}). By definition of $e$ being an $(A, \odot, u)$-solution of $(A, u, 1)$, $0e=u$ and $\mathsf{D}[e] = (e \times 1) \odot$. So it remains to show that $\oplus e= (e \times e) \odot$. To do so, we will show that $\oplus e$ and $(e \times e)\odot$ are both parametrized $(A, \odot, u)$-solutions of the parametrized dynamical system $(A, e, 1)$. Starting with $\oplus e$: 
\begin{align*}
\langle 0, 1 \rangle \oplus e &=e \tag{\ref{ostarmonoideq}} \end{align*}
\begin{align*}
\left( (1 \times 1) \times \langle 1, 0 \rangle \right) \mathsf{D}[\oplus e] &=~ \left( (1 \times 1) \times \langle 1, 0 \rangle \right)(\oplus \times \oplus) \mathsf{D}[e] \tag{Lemma \ref{opluslem} + Lemma \ref{linlem}} \\
&=~ (\oplus \times 1) \mathsf{D}[e] \tag{\ref{ostarmonoideq}} \\
&=~ (\oplus \times 1) (e \times 1) \odot \tag{Assumption (i) + (\ref{solution4})} 
\end{align*}
Next we work with $(e \times e)\odot$: 
\begin{align*}
\langle 0,1 \rangle (e \times e) \odot &=~\langle 0e, e \rangle \odot \\
&=~ \langle u, e \rangle \odot \tag{Assumption (ii) + (\ref{solution4})} \\
&=~ e \tag{\ref{ostarmonoideq}}
\end{align*} 
\begin{align*}
&\left( (1 \times 1) \times \langle 1, 0 \rangle \right) \mathsf{D}[(e \times e) \odot] =~ \left( (1 \times 1) \times \langle 1, 0 \rangle \right)\mathsf{T}(e \times e) \mathsf{D}[\odot] \tag{Lemma \ref{tanlem}} \\
&=~ \left( (1 \times 1) \times \langle 1, 0 \rangle \right) c (\mathsf{T}(e) \times \mathsf{T}(e) ) c \mathsf{D}[\odot] \tag{Lemma \ref{tanlem}} \\
&=~\langle \pi_0 \times 1, \pi_0 \pi_1 \rangle \left( (1 \times 1) \times \langle 1, 0 \rangle \right) (\mathsf{T}(e) \times \mathsf{T}(e) ) c \mathsf{D}[\odot] \\
&=~ \langle \pi_0 \times 1, \pi_0 \pi_1 \rangle (\mathsf{T}(e) \times e) \left( (1 \times 1) \times \langle 1, 0 \rangle \right) c \mathsf{D}[\odot] \tag{Lemma \ref{tanlem}} \\
&=~ \langle \pi_0 \times 1, \pi_0 \pi_1 \rangle (\mathsf{T}(e) \times e) \langle \pi_0 \times 1, \pi_0 \pi_1 \rangle \left( (1 \times 1) \times \langle 1, 0 \rangle \right) \mathsf{D}[\odot] \\
&=~ \langle \pi_0 \times 1, \pi_0 \pi_1 \rangle \langle e \times e, \mathsf{D}[e] \rangle \left( (1 \times 1) \times \langle 1, 0 \rangle \right) \mathsf{D}[\odot] \\
&=~ \langle \pi_0 \times 1, \pi_0 \pi_1 \rangle \langle e \times e, \mathsf{D}[e] \rangle \left( (1 \times 1) \times \langle 1, 0 \rangle \right) (\pi_0 \times \pi_1) \odot \\
&~~~+ \langle \pi_0 \times 1, \pi_0 \pi_1 \rangle \langle e \times e, \mathsf{D}[e] \rangle \left( (1 \times 1) \times \langle 1, 0 \rangle \right) (\pi_1 \times \pi_0) \odot \tag{\ref{odotbilin}} \\
&=~ \langle \pi_0 \times 1, \pi_0 \pi_1 \rangle \langle e \times e, \mathsf{D}[e] \rangle (\pi_0 \times 0) \odot + \langle \pi_0 \times 1, \pi_0 \pi_1 \rangle \langle e \times e, \mathsf{D}[e] \rangle (\pi_1 \times 1) \odot \\
&=~ 0 + \langle \pi_0 \times 1, \pi_0 \pi_1 \rangle \langle \pi_1 e, \mathsf{D}[e] \rangle \odot \tag{\ref{distributeeq}} \\
&=~ \left \langle \pi_1 e, \langle \pi_0 \times 1, \pi_0 \pi_1 \rangle \mathsf{D}[e] \right \rangle \odot \\
&=~ \left \langle \pi_1 e, \langle \pi_0 \times 1, \pi_0 \pi_1 \rangle (e \times 1) \odot \right \rangle \odot \tag{Assumption (ii) + (\ref{solution4})} \\
&=~((e \times e) \times 1)(\odot \times 1)\odot \tag{\ref{ostarmonoideq}} \\ 
\end{align*}
So we have that $\oplus e$ and $(e \times e)\odot$ are both parametrized $(A, \odot, u)$-solutions of the parametrized dynamical system $(A, e, 1)$. However by assumption (ii), solutions are unique and therefore we have that $\oplus e= (e \times e)\odot$. And so we conclude that $(A, \odot, u, e)$ is a differential exponential rig. \end{proof} 

\section{Differential Exponential Maps for Differential Categories} \label{diffexpsec}

An interesting and important source of Cartesian differential categories are coKleisli categories of differential categories \cite{blute2009cartesian,blute2006differential}. In this section, we study differential exponential maps in the coKleisli category of a differential (storage) category. We also introduce $\oc$-differential exponential algebras and show that these are in bijective correspondence with differential exponential maps in the coKleisli category of a differential storage category. 

If only to introduce notation, we first briefly review the full definition of a differential category. Here we present the definition of differential category found in \cite{Blute2019}, which is mostly the same as the one found in the original paper \cite{blute2006differential} but with the addition of the interchange rule, which has now become part of the definition. Also, for simplicity and following the convention of other differential category papers, in this section we allow ourselves to work in \emph{strict} monoidal categories, that is, the associator and the unitor of the tensor product $\otimes$ are strict equalities. For a symmetric monoidal category, we let $K$ be the monoidal unit and $\sigma: A \otimes B \to B \otimes A$ be the natural symmetry isomorphism. 

\begin{definition} A \textbf{differential category} \cite[Definition 2.4]{blute2006differential} consists of: 
\begin{enumerate}[{\em (i)}]
\item An \textbf{additive symmetric monoidal category} \cite[Definition 3]{Blute2019} which is a symmetric monoidal category $\mathbb{X}$ which is enriched over commutative monoids, that is, each hom-set $\mathbb{X}(A,B)$ is a commutative monoid with addition $\mathbb{X}(A,B) \times \mathbb{X}(A,B) \xrightarrow{+} \mathbb{X}(A,B)$, $(f,g) \mapsto f +g$, and zero map $0 \in \mathbb{X}(A,B)$, such that composition and the tensor product preserves the additive structure, that is, the following equalities hold: 
   \begin{equation}\label{add2}\begin{gathered} f(g+h)k=fgk+fhk \quad \quad \quad f0g=0 
    \end{gathered}\end{equation}
       \begin{equation}\label{add3}\begin{gathered} f \otimes (g+h) \otimes k = (f \otimes g \otimes k) + (f \otimes h \otimes k) \quad \quad \quad f \otimes 0 \otimes g = 0
    \end{gathered}\end{equation}
\item Equipped with a \textbf{coalgebra modality} \cite[Definition 1]{Blute2019}, which is a quintuple $(\oc, \delta, \varepsilon, \Delta, \iota)$ consisting of an endofunctor $\oc$ and natural transformations $\oc A \xrightarrow{\delta} \oc \oc A$, $\oc A \xrightarrow{\varepsilon} A$, $\oc A \xrightarrow{\Delta} \oc A \otimes \oc A$, and $\oc A \xrightarrow{\iota} K$, such that $(\oc, \delta, \varepsilon)$ is a comonad on $\mathbb{X}$ and $(\oc A, \Delta, \iota)$ is a cocommutative comonoid, that is, the following equalities hold:
\begin{equation}\label{comonoid}\begin{gathered} 
\Delta(\Delta \otimes 1) = \Delta(1 \otimes \Delta) \quad \quad \Delta(1 \otimes \iota) = 1 = \Delta(\iota \otimes 1) \quad \quad \Delta \sigma = \Delta
\end{gathered}\end{equation}
 and $\delta$ is a comonoid morphism, that is, the following equalities hold: 
 \begin{equation}\label{deltaeq}\begin{gathered} 
 \delta \Delta = \Delta (\delta \otimes \delta) \quad \quad \quad \delta \iota = \iota 
  \end{gathered}\end{equation}
 \item Equipped with a \textbf{deriving transformation} \cite[Definition 7]{Blute2019} which is a natural transformation $\oc A \otimes A \xrightarrow{\mathsf{d}} \oc A$ such that the following axioms hold: 
 \begin{enumerate}[{\bf [d.1]}]
 \item Constant Rule: $\mathsf{d} \iota = 0$
 \item Leibniz Rule: $\mathsf{d} \Delta = (\Delta \otimes 1)(1 \otimes \mathsf{d}) + (\Delta\otimes 1)(1 \otimes \sigma)(\mathsf{d} \otimes 1)$
 \item Linear Rule: $\mathsf{d} \varepsilon = \iota \otimes 1$
 \item Chain Rule: $\mathsf{d} \delta = (\Delta \otimes 1)(\mathsf{d} \otimes \delta)\mathsf{d}$
 \item Interchange Rule: $(1 \otimes \sigma)(\mathsf{d} \otimes 1)\mathsf{d} = (\mathsf{d} \otimes 1)\mathsf{d}$. 
\end{enumerate}
\end{enumerate}
\end{definition}

For a full detailed explanation of the deriving transformation axioms and a string diagram representation, see \cite{Blute2019,blute2006differential}. Examples of differential categories can be found at the end of this section (Example \ref{RELex} and Example \ref{VECex}), while other interesting examples can be found in \cite[Section 9]{Blute2019}. 

For a differential category with finite products, its coKleisli category is a Cartesian differential category. As we will be working with coKleisli categories, we will use the notation found in \cite{blute2015cartesian} and use interpretation brackets $\llbracket - \rrbracket$ to help distinguish between composition in the base category and coKleisli composition. So for a comonad $(\oc, \delta, \varepsilon)$ on a category $\mathbb{X}$, let $\mathbb{X}_\oc$ denote its coKleisli category, which is the category whose objects are the same as $\mathbb{X}$ and where $\mathbb{X}_\oc(A,B) = \mathbb{X}(\oc A, B)$ with composition and identity defined as: 
\begin{align*}
\llbracket fg \rrbracket = \delta \oc(\llbracket f \rrbracket) \llbracket g \rrbracket && \llbracket 1 \rrbracket = \varepsilon 
\end{align*}
If $\mathbb{X}$ if has finite products then so does $\mathbb{X}_\oc$ where on objects the product is defined as in $\mathbb{X}$ and where the remaining data is defined as follows: 
\begin{align*}
\llbracket \pi_0 \rrbracket = \varepsilon \pi_0 && \llbracket \pi_1 \rrbracket = \varepsilon \pi_1 && \llbracket \langle f, g \rangle \rrbracket = \left \langle \llbracket f \rrbracket, \llbracket g \rrbracket \right \rangle && \llbracket f \times g \rrbracket = \left \langle \oc(\pi_0) \llbracket f \rrbracket , \oc(\pi_1) \llbracket g \rrbracket \right \rangle
\end{align*}
If $\mathbb{X}$ is a Cartesian left additive category then so is $\mathbb{X}_\oc$ where:
\begin{align*}
\llbracket f+g \rrbracket = \llbracket f \rrbracket + \llbracket g \rrbracket && \llbracket 0 \rrbracket = 0 && \llbracket \oplus \rrbracket = \varepsilon \oplus \end{align*}
where $\oplus$ is defined as in Lemma \ref{opluslem}. Since every category with finite biproducts is a Cartesian left additive category (where every map is additive), it follows that every differential category with finite products (which by the additive enrichment are in fact biproducts) is a Cartesian left additive category. Lastly, one then uses the deriving transformation to define the differential combinator of the coKleisli category. 

\begin{proposition}\label{coKleisliCDC} \cite[Proposition 3.2.1]{blute2009cartesian} Let $\mathbb{X}$ be a differential category with finite products. Define the natural transformation $\oc(A \times B) \xrightarrow{\chi} \oc A \otimes \oc B$ as follows: 
\[\chi := \xymatrixcolsep{5pc}\xymatrix{\oc(A \times B) \ar[r]^-{\Delta} & \oc(A \times B) \otimes \oc(A \times B) \ar[r]^-{\oc(\pi_0) \otimes \oc(\pi_1)} &\oc A \otimes \oc B
 } \]
Then the coKleisli category $\mathbb{X}_\oc$ is a Cartesian differential category with Cartesian left additive structure defined above and differential combinator $\mathsf{D}$ defined as follows on a coKleisli map $\oc A \xrightarrow{\llbracket f \rrbracket} B$: 
\[ \llbracket \mathsf{D}[f] \rrbracket := \xymatrixcolsep{2pc}\xymatrix{\oc(A \times A) \ar[r]^-{\chi} & \oc A \otimes \oc A \ar[r]^-{1 \otimes \varepsilon} & \oc A \otimes A \ar[r]^-{\mathsf{d}} & \oc A \ar[r]^-{\llbracket f \rrbracket} & B 
 } \]
\end{proposition} 

A differential exponential map in the coKleisli category of a differential category would be a map of type $\oc A \xrightarrow{e} A$ satisfying the required identities, which we can simplify slightly. 

\begin{proposition}\label{coKexp1} For a differential category $\mathbb{X}$ with finite products, a map $\oc A \xrightarrow{e} A$ is a differential exponential map in the coKleisli category $\mathbb{X}_\oc$ if and only if the following diagrams commute: 
\begin{equation}\label{cokleisliexp}\begin{gathered} \xymatrixcolsep{1.5pc}\xymatrix{ \oc A \ar[ddrr]_-{\varepsilon} \ar[r]^-{\Delta} & \oc A \otimes \oc A \ar[r]^-{\oc(0) \otimes \varepsilon} & \oc A \otimes A \ar[d]^-{\mathsf{d}} & \oc(A \times A) \ar[dd]_-{\oc(\oplus)} \ar[r]^-{\chi} & \oc A \otimes \oc A \ar[r]^-{1 \otimes e} & \oc A \otimes A \ar[d]^-{\mathsf{d}} \\
 & & \oc A \ar[d]^-{e} & & & \oc A \ar[d]^-{e} \\
 & & A & \oc A \ar[rr]_-{e}&& A 
 } \end{gathered}\end{equation}
\end{proposition} 
\begin{proof} Let $\oc A \xrightarrow{e} A$ be an arbitrary map. Then we leave it to the reader to check for themselves that we have the following three equalities (which are mostly straightforward calculations):  
\begin{align*}
\llbracket \langle 0,1 \rangle \mathsf{D}[e] \rrbracket = \Delta (\oc(0) \otimes \varepsilon) \mathsf{d} e && \llbracket \oplus e \rrbracket = \oc(\oplus) e && \llbracket (1 \times e) \mathsf{D}[e] \rrbracket = \chi (1 \otimes e) \mathsf{d} e
\end{align*}
Therefore, $e$ is a differential exponential map in the coKleisli category if and only if $\llbracket \langle 0,1 \rangle \mathsf{D}[e] \rrbracket = \llbracket 1 \rrbracket$ and $\llbracket \oplus e \rrbracket = \llbracket (1 \times e) \mathsf{D}[e] \rrbracket$, which by the above equalities is precisely that both $\Delta (\oc(0) \otimes \varepsilon) \mathsf{d} e = \varepsilon$ and $\oc(\oplus) e = \chi (1 \otimes e) \mathsf{d} e$ hold. 
\end{proof} 

We now study differential exponential maps in the presence of the Seely isomorphisms. 

\begin{definition} A \textbf{differential storage category} \cite[Definition 10]{Blute2019} is a differential category with finite products whose coalgebra modality has the \textbf{Seely isomorphisms}, that is, the natural transformation $\oc(A \times B) \xrightarrow{\chi} \oc A \otimes \oc B$ (as defined in Proposition \ref{coKleisliCDC}) is a natural isomorphism and the map $\oc \top \xrightarrow{\chi_\top} K$ defined as $\chi_\top = \iota$ is an isomorphism. 
\end{definition}

There are numerous interesting consequences of having the Seely isomorphisms, such as $\oc A$ coming equipped with a natural commutative monoid structure. Recall that in a symmetric (strict) monoidal category, a commutative monoid is a triple $(A, \blacktriangledown, v)$ consisting of an object $A$ and maps $A \otimes A \xrightarrow{\blacktriangledown} A$, and $K \xrightarrow{v} A$ such that the following equalities hold: 
\begin{equation}\label{tensormonoid}\begin{gathered} 
(\blacktriangledown \otimes 1)\nabla = (1 \otimes \blacktriangledown)\blacktriangledown \quad \quad \quad (v \otimes 1) \blacktriangledown = 1 = (1 \otimes v)\blacktriangledown \quad \quad \quad \sigma \blacktriangledown = \blacktriangledown
\end{gathered}\end{equation}
In the presence of the Seely isomorphisms, define the natural transformations $\oc A \otimes \oc A \xrightarrow{\nabla} \oc A$ and $K \xrightarrow{\nu} \oc A$ respectively as follows:
\begin{equation}\label{nabladef}\begin{gathered} \nabla := \xymatrixcolsep{1.75pc}\xymatrix{ \oc A \otimes \oc A \ar[r]^-{\chi^{-1}} & \oc(A \times A) \ar[r]^-{\oc(\oplus)} & \oc A 
 } \quad \quad \quad \nu := \xymatrixcolsep{1.75pc}\xymatrix{ K \ar[r]^-{\chi_\top^{-1}} & \oc \top \ar[r]^-{\oc(0)} & \oc A 
 } \end{gathered}\end{equation}
By \cite[Theorem 6]{Blute2019}, $(\oc A, \nabla, \nu)$ is a commutative monoid. In fact, $\oc A$ is also a bialgebra and this makes $\oc$ an \textbf{additive bialgebra modality} \cite[Definition 5]{Blute2019} and so in particular we have the following equalities: 
\begin{equation}\label{addbialg}\begin{gathered} \oc(f+g) = \Delta (\oc(f) \otimes \oc(g)) \nabla \quad \quad \quad \oc(0) = \iota \nu \end{gathered}\end{equation}
One can also define the natural transformation $A \xrightarrow{\eta} \oc A$, called the \textbf{codereliction} \cite[Definition 9]{Blute2019}, as follows: 
\begin{equation}\label{etadef}\begin{gathered} \eta := \xymatrixcolsep{5pc}\xymatrix{ A \ar[r]^-{\nu \otimes 1} & \oc A \otimes A \ar[r]^-{\mathsf{d}} & \oc A 
 } \end{gathered}\end{equation}
and the following equalities hold: 
  \begin{enumerate}[{\bf [cd.1]}]
 \item Constant Rule: $\eta \iota = 0$
 \item Leibniz Rule: $\eta \Delta = \eta \otimes \nu + \nu \otimes \eta$
 \item Linear Rule: $\eta \varepsilon = 1$
 \item (Alternative) Chain Rule: $\eta \delta = (\nu \otimes \eta)(\delta \otimes \eta) \nabla$. 
\end{enumerate}
By \cite[Theorem 4]{Blute2019}, for a differential storage category, one could have started with a codereliction $\eta$ to construct a deriving transformation as follows:
\begin{equation}\label{ddef}\begin{gathered} \mathsf{d} := \xymatrixcolsep{5pc}\xymatrix{ \oc A \otimes A \ar[r]^-{1 \otimes \eta} & \oc A \otimes \oc A \ar[r]^-{\nabla} & \oc A 
 } \end{gathered}\end{equation}
These constructions are inverses of each other and thus in the presence of the Seely isomorphisms: deriving transformations are in bijective correspondence with coderelictions. 
  
\begin{proposition}\label{coKexp2} For a differential storage category $\mathbb{X}$, a map $\oc A \xrightarrow{e} A$ is a differential exponential map in the coKleisli category $\mathbb{X}_\oc$ if and only if the following diagrams commute: 
\begin{equation}\label{cokleisliexp2}\begin{gathered} \xymatrixcolsep{3pc}\xymatrix{A \ar@{=}[dr]^-{} \ar[r]^-{\eta} & \oc A \ar[d]^-{e} & \oc A \otimes \oc A \ar[d]_-{\nabla} \ar[r]^-{1 \otimes e} & \oc A \otimes A \ar[r]^-{\mathsf{d}} & \oc A \ar[d]^-{e} \\
& A & \oc A \ar[rr]_-{e} & & A  
 } \end{gathered}\end{equation} 
\end{proposition} 
\begin{proof} Suppose that $\oc A \xrightarrow{e} A$ is a differential exponential map in the coKleisli category. Using Proposition \ref{coKexp1} we show that $e$ satisfies (\ref{cokleisliexp2}): 
\begin{align*}
1 &=~ \eta \varepsilon \tag{\bf [cd.3]} \\
&=~ \eta \Delta (\oc(0) \otimes \varepsilon) \mathsf{d} e \tag{\ref{cokleisliexp}} \\
&=~ (\eta \otimes \nu) (\oc(0) \otimes \varepsilon) \mathsf{d} e + (\nu \otimes \eta) (\oc(0) \otimes \varepsilon) \mathsf{d} e \tag{\bf [cd.2]} \\ 
&=~ 0 + \eta \varepsilon (\nu \otimes 1) \mathsf{d} e \tag{Naturality of $\eta$ and $\nu$} \\
&=~ (\nu \otimes 1) \mathsf{d} e \tag{\bf [cd.3]} \\
&=~ \eta e \tag{\ref{etadef}} \\ \\ 
\nabla e &=~ \chi^{-1} \oc(\oplus) e \tag{\ref{nabladef}} \\
&=~ \chi^{-1} \chi (1 \otimes e) \mathsf{d} e \tag{\ref{cokleisliexp}} \\
&=~ (1 \otimes e) \mathsf{d} e 
\end{align*}
Conversely, suppose $\oc A \xrightarrow{e} A$ satisfies (\ref{cokleisliexp2}). We show that $e$ satisfies (\ref{cokleisliexp}): 
\begin{align*}
\Delta (\oc(0) \otimes \varepsilon) \mathsf{d} e &=~\Delta (\oc(0) \otimes \varepsilon)(1 \otimes \eta) \nabla e \tag{\ref{ddef}} \\ 
&=~ \Delta (\iota \otimes 1) (\nu \otimes 1)(1 \otimes \varepsilon)(1 \otimes \eta) \nabla e \tag{\ref{addbialg}} \\
&=~ \varepsilon \eta e \tag{(\ref{comonoid}) + (\ref{tensormonoid})} \\
&=~ \varepsilon \tag{\ref{cokleisliexp2}} \\ \\
\chi (1 \otimes e) \mathsf{d} e &=~ \chi \nabla e \tag{\ref{cokleisliexp2}} \\ 
&=~ \chi \chi^{-1} \oc(\oplus) e \tag{\ref{nabladef}} \\
&=~ \oc(\oplus) e
\end{align*}
Therefore, by Proposition \ref{coKexp1}, $e$ is a differential exponential map. 
\end{proof}

In a differential storage category, differential exponential maps in the coKleisli category can also be characterized by commutative monoids in the base category, which the analogue of how differential exponential maps are in bijective correspondence to differential exponential rigs. 

\begin{definition}\label{!deadef} A \textbf{$\oc$-differential exponential algebra} in a differential storage category is a quadruple $(A, \blacktriangledown, v, e)$ consisting of an object $A$ and maps $A \otimes A \xrightarrow{\blacktriangledown} A$, $K \xrightarrow{v} A$, and $\oc A \xrightarrow{e} A$ such that $(A, \blacktriangledown, v)$ is a commutative monoid, 
and also that the following diagrams commute: 
\begin{equation}\label{!exp}\begin{gathered} \xymatrixcolsep{2.5pc}\xymatrix{ A \ar@{=}[dr]^-{} \ar[r]^-{\eta} & \oc A \ar[d]^-{e} & K \ar[r]^-{\nu} \ar[dr]_-{v} & \oc A \ar[d]^-{e} & \oc A \otimes \oc A \ar[d]_-{e \otimes e} \ar[r]^-{\nabla} & \oc A \ar[d]^-{e} \\
 & A & & A & A \otimes A \ar[r]_-{\blacktriangledown} & A  
 } \end{gathered}\end{equation}
\end{definition}

Note in particular that for a $\oc$-differential exponential algebra, the two rightmost diagrams of (\ref{!exp}) says that $e$ is a monoid morphism. We now show that every $\oc$-differential exponential algebra induces a differential exponential map in the coKleisli category and vice-versa. 

\begin{proposition}\label{propcok1} Let $(A, \blacktriangledown, v, e)$ be a $\oc$-differential exponential algebra. Then $\oc A \xrightarrow{e} A$ is a differential exponential map in the coKleisli category and furthermore the following diagrams commute: 
\[\xymatrixcolsep{2.5pc}\xymatrix{ \oc(A \times A) \ar[r]^-{\chi} \ar[drr]_-{\llbracket \odot_e \rrbracket} & \oc A \otimes \oc A \ar[r]^-{\varepsilon \otimes \varepsilon} & A \otimes A \ar[d]^-{\blacktriangledown} & \oc \top \ar[r]^-{\chi_\top} \ar[dr]_-{\llbracket u_e \rrbracket} & K \ar[d]^-{v} \\
& & A & & A
 } \]
 where $\odot_e$ and $u_e$ are defined as in Proposition \ref{prop2}. 
\end{proposition} 
\begin{proof} By Proposition \ref{coKexp2}, it suffices to show that $e$ satisfies both diagrams of (\ref{cokleisliexp2}). However the left diagram of (\ref{cokleisliexp2}) is precisely the left most diagram of (\ref{!exp}). So it remains to show that $\nabla e = (1 \otimes e)\mathsf{d} e$: 
\begin{align*}
(1 \otimes e)\mathsf{d} e &=~ (1 \otimes e)(1 \otimes \eta) \nabla e \tag{\ref{ddef}} \\
&=~ (1 \otimes e)(1 \otimes \eta)(e \otimes e) \blacktriangledown\tag{\ref{!exp}} \\
&=~ (e \otimes e)\blacktriangledown\tag{\ref{!exp}} \\
&=~ \nabla e \tag{\ref{!exp}}  
\end{align*}
So we conclude that $e$ is a differential exponential map in the coKleisli category. Next we show that $\llbracket u_e \rrbracket = \chi_\top v$: 
\begin{align*}
\llbracket u_e \rrbracket &=~ \llbracket 0 e \rrbracket \\
&=~ \delta \oc(0) e \\
&=~ \delta \iota \nu e \tag{\ref{addbialg}} \\
&=~ \iota \nu e \tag{\ref{deltaeq}} \\
&=~ \iota v \tag{\ref{!exp}}  \\
&=~ \chi_\top v
\end{align*}
To show the other equality, we observe that in the proof of \cite[Proposition 3.2.1]{blute2009cartesian}, it was computed out that we have the following equality for any $\oc A \xrightarrow{f} B$:
\begin{equation}\label{D2cok}\begin{gathered} \llbracket \left( (1 \times 1) \times \langle 1, 0 \rangle \right) \mathsf{D}^2[f] \rrbracket = \chi(\chi \otimes 1)(1 \otimes \varepsilon \otimes \varepsilon)(\mathsf{d} \otimes 1)\mathsf{d} f
 \end{gathered}\end{equation}
 Using the above identity, we can show that: 
 \begin{align*}
\llbracket \odot_e \rrbracket &=~ \llbracket \left( \langle 0,1 \rangle \times \langle 1, 0 \rangle \right) \mathsf{D}^2[e] \rrbracket \\
&=~ \llbracket ( \langle 0,1 \rangle \times 1) \left( (1 \times 1) \times \langle 1, 0 \rangle \right) \mathsf{D}^2[e] \rrbracket \\
&=~ \delta \oc( \llbracket \langle 0,1 \rangle \times 1 \rrbracket) \llbracket \left( (1 \times 1) \times \langle 1, 0 \rangle \right) \mathsf{D}^2[e] \rrbracket \\
&=~ \delta \oc( \langle \oc(\pi_0) \llbracket \langle 0,1 \rangle \rrbracket, \oc(\pi_1) \llbracket 1 \rrbracket \rangle ) \llbracket \left( (1 \times 1) \times \langle 1, 0 \rangle \right) \mathsf{D}^2[e] \rrbracket \\
&=~ \delta \oc( \langle \oc(\pi_0) \langle 0, \varepsilon \rangle, \oc(\pi_1) \varepsilon \rangle) \llbracket \left( (1 \times 1) \times \langle 1, 0 \rangle \right) \mathsf{D}^2[e] \rrbracket \\ 
&=~ \delta \oc( \langle \varepsilon \langle 0, \pi_0 \rangle, \varepsilon\pi_1 \rangle ) \llbracket \left( (1 \times 1) \times \langle 1, 0 \rangle \right) \mathsf{D}^2[e] \rrbracket \tag{Naturality of $\varepsilon$} \\ 
&=~ \delta \oc(\varepsilon) \oc(\langle 0, \pi_0 \rangle, \varepsilon\pi_1 \rangle ) \llbracket \left( (1 \times 1) \times \langle 1, 0 \rangle \right) \mathsf{D}^2[e] \rrbracket \\ 
&=~ \oc(\langle 0, \pi_0 \rangle, \pi_1 \rangle ) \llbracket \left( (1 \times 1) \times \langle 1, 0 \rangle \right) \mathsf{D}^2[e] \rrbracket \tag{Comonad} \\ 
&=~ \oc(\langle 0, \pi_0 \rangle, \pi_1 \rangle ) \chi(\chi \otimes 1)(1 \otimes \varepsilon \otimes \varepsilon)(\mathsf{d} \otimes 1)\mathsf{d} e \tag{\ref{D2cok}} \\ 
&=~ \Delta \left( \oc(\langle 0, \pi_0 \rangle) \otimes \oc(\pi_1) \right) (\chi \otimes 1) (1 \otimes \varepsilon \otimes \varepsilon)(\mathsf{d} \otimes 1)\mathsf{d} e \tag{Definition of $\chi$} \\
&=~ \Delta(1 \otimes \oc(\pi_1)) (\Delta \otimes 1)(\oc(0) \otimes \oc(\pi_0) \otimes 1) (1 \otimes \varepsilon \otimes \varepsilon)(\mathsf{d} \otimes 1)\mathsf{d} e \tag{Definition of $\chi$} \\
&=~ \Delta(\Delta \otimes 1)(\oc(0) \otimes \oc(\pi_0) \otimes \oc(\pi_1)) (1 \otimes \varepsilon \otimes \varepsilon)(\mathsf{d} \otimes 1)\mathsf{d} e \\
&=~ \Delta(1 \otimes \Delta)(\oc(0) \otimes \oc(\pi_0) \otimes \oc(\pi_1)) (1 \otimes \varepsilon \otimes \varepsilon)(\mathsf{d} \otimes 1)\mathsf{d} e \tag{\ref{comonoid}} \\
&=~ \Delta(1 \otimes \chi)(\oc(0) \otimes \varepsilon \otimes \varepsilon)(\mathsf{d} \otimes 1)\mathsf{d} e \\
&=~ \Delta(1 \otimes \chi)(\iota \otimes 1 \otimes 1)(\nu \otimes \varepsilon \otimes \varepsilon)(\mathsf{d} \otimes 1)\mathsf{d} e \tag{\ref{addbialg}}\\
&=~ \chi (\varepsilon \otimes \varepsilon)(\eta \otimes 1) \mathsf{d} e \tag{(\ref{comonoid}) + (\ref{etadef})} \\
&=~ \chi (\varepsilon \otimes \varepsilon)(\eta \otimes 1)(1 \otimes \eta) \nabla e\tag{\ref{ddef}} \\ 
&=~ \chi (\varepsilon \otimes \varepsilon)(\eta \otimes 1)(1 \otimes \eta) (e \otimes e) \blacktriangledown \tag{\ref{!exp}} \\ 
&=~ \chi (\varepsilon \otimes \varepsilon) \blacktriangledown \tag{\ref{!exp}} 
\end{align*}
And so we have that $\llbracket \odot_e \rrbracket = \chi (\varepsilon \otimes \varepsilon) \blacktriangledown$. 
\end{proof}

\begin{proposition}\label{propcok2} Let $\oc A \xrightarrow{e} A$ be a differential exponential map in the coKleisli category of a differential storage category. Define the maps $A \otimes A \xrightarrow{\blacktriangledown_e} A$ and $K \xrightarrow{v_e} A$ respectively as follows: 
\[\blacktriangledown_e := \xymatrixcolsep{5pc}\xymatrix{ A \otimes A \ar[r]^-{\eta \otimes \eta} & \oc A \otimes \oc A \ar[r]^-{\nabla} & \oc A \ar[r]^-{e} & A } \] 
\[ v_e := \xymatrixcolsep{5pc}\xymatrix{K \ar[r]^-{\nu} & \oc A \ar[r]^-{e} & A 
 } \]
 Then $(A, \blacktriangledown_e, v_e, e)$ is a $\oc$-differential exponential algebra.
\end{proposition} 
\begin{proof} We first show that $(A, \blacktriangledown_e, v_e)$ is a commutative monoid. Starting with showing that $\blacktriangledown_e$ is commutative: 
\begin{align*}
\sigma \blacktriangledown_e &=~ \sigma (\eta \otimes \eta) \nabla e \\
&=~ (\eta \otimes \eta) \sigma \nabla e \tag{Naturality of $\sigma$} \\
&=~ (\eta \otimes \eta) \nabla e \tag{\ref{tensormonoid}} 
\end{align*}
Since we've shown commutativity, we need only show one of the unit identities: 
\begin{align*}
(1 \otimes v_e) \blacktriangledown_e &=~ (1 \otimes \nu)(1 \otimes e) (\eta \otimes \eta) \nabla e \\ 
&=~ (1 \otimes \nu)(\eta \otimes 1)(1 \otimes e)(1 \otimes \eta) \nabla e \\ 
&=~ (1 \otimes \nu)(\eta \otimes 1)(1 \otimes e)\mathsf{d} e \tag{\ref{ddef}} \\
&=~ (1 \otimes \nu)(\eta \otimes 1)\nabla e \tag{\ref{cokleisliexp2}} \\ 
&=~ \eta e \tag{\ref{tensormonoid}} \\
&=~ 1 \tag{\ref{cokleisliexp2}}
\end{align*}
Lastly, we show that $\blacktriangledown_e$ is also associative: 
\begin{align*}
(1 \otimes \blacktriangledown_e) \blacktriangledown_e &=~ (1 \otimes \eta \otimes \eta)(1 \otimes \nabla)(1 \otimes e)(\eta \otimes \eta) \nabla e \\
&=~ (1 \otimes \eta \otimes \eta)(1 \otimes \nabla)(\eta \otimes 1)(1 \otimes e)(1 \otimes \eta) \nabla e \\
&=~ (1 \otimes \eta \otimes \eta)(1 \otimes \nabla)(\eta \otimes 1)(1 \otimes e)\mathsf{d} e \tag{\ref{ddef}} \\
&=~ (1 \otimes \eta \otimes \eta)(1 \otimes \nabla)(\eta \otimes 1)\nabla e \tag{\ref{cokleisliexp2}} \\ 
&=~ (\eta \otimes \eta \otimes \eta)(1 \otimes \nabla)\nabla e \\
&=~ (\eta \otimes \eta \otimes \eta)(\nabla \otimes 1)\nabla e \tag{\ref{tensormonoid}} \\
&=~ (\eta \otimes \eta \otimes \eta)(\nabla \otimes 1) \sigma \nabla e \tag{\ref{tensormonoid}} \\
&=~ (\eta \otimes \eta \otimes \eta)(\nabla \otimes 1) \sigma (1 \otimes e)\mathsf{d} e \tag{\ref{cokleisliexp2}} \\
&=~ (\eta \otimes \eta \otimes \eta)(\nabla \otimes 1) \sigma (1 \otimes e)(1 \otimes \eta) \nabla e\tag{\ref{ddef}} \\
&=~ (\eta \otimes \eta \otimes \eta)(\nabla \otimes 1) (e \otimes 1)(\eta \otimes 1) \sigma \nabla e\tag{Naturality of $\sigma$} \\
&=~ (\eta \otimes \eta \otimes \eta)(\nabla \otimes 1) (e \otimes 1)(\eta \otimes 1) \nabla e \tag{\ref{tensormonoid}} \\
&=~ (\eta \otimes \eta \otimes 1)(\nabla \otimes 1)(e \otimes 1)(\eta \otimes \eta) \nabla e\\
&=~ (\blacktriangledown_e \otimes 1) \blacktriangledown_e 
\end{align*}
So we conclude that $(A, \blacktriangledown_e, v_e)$ is a commutative monoid. Next we show that $e$ satisfies the three identities of (\ref{!exp}). The left most diagram of (\ref{!exp}) is precisely the left diagram of (\ref{cokleisliexp2}) and $v_e = \nu e$ by construction. So it remains only to show that $\nabla e = (e \otimes e) \blacktriangledown_e$:
\begin{align*}
 (e \otimes e) \blacktriangledown_e &=~ (e \otimes e) (\eta \otimes \eta) \nabla e \\ 
 &=~ (e \otimes 1)(\eta \otimes 1) (1 \otimes e)(1 \otimes \eta) \nabla e \\
 &=~ (e \otimes 1)(\eta \otimes 1) (1 \otimes e)\mathsf{d} e \tag{\ref{ddef}} \\
 &=~ (e \otimes 1)(\eta \otimes 1) \nabla e \tag{\ref{cokleisliexp2}} \\
 &=~ (e \otimes 1)(\eta \otimes 1) \sigma \nabla e \tag{\ref{tensormonoid}} \\
 &=~ \sigma (1 \otimes e)(1 \otimes \eta) \nabla e \tag{Naturality of $\sigma$} \\
 &=~ \sigma (1 \otimes e) \mathsf{d} e \tag{\ref{ddef}} \\
 &=~ \sigma \nabla e \tag{\ref{cokleisliexp2}} \\
 &=~ \nabla e \tag{\ref{tensormonoid}}
\end{align*}
So we conclude that $(A, \blacktriangledown_e, v_e, e)$ is a $\oc$-differential exponential algebra. 
\end{proof}

Finally we will show that the constructions of Proposition \ref{propcok1} and Proposition \ref{propcok2} are inverses of each other by showing that the category of $\oc$-differential exponential algebras is isomorphic to the category of differential exponential maps in the coKleisli category. For a differential storage category $\mathbb{X}$, define its category of $\oc$-differential exponential algebras $\mathsf{\oc DEA}[\mathbb{X}]$ as the category whose objects are $\oc$-differential exponential algebras $(A, \blacktriangledown, v, e)$ and where a map ${(A, \blacktriangledown, v, e) \xrightarrow{f} (B, \blacktriangledown^\prime, v^\prime, e^\prime)}$ is a map $A \xrightarrow{f} B$ such that the following diagrams commute: 
   \begin{equation}\label{!demmap}\begin{gathered} \xymatrixcolsep{2.5pc}\xymatrix{\oc A \ar[d]_-{e} \ar[r]^-{\oc(f)} & \oc B \ar[d]^-{e^\prime} & \top \ar[r]^-{v} \ar[dr]_-{v^\prime} & A \ar[d]^-{f} & A \otimes A \ar[d]_-{\blacktriangledown} \ar[r]^-{f \otimes f} & B \otimes B \ar[d]^-{\blacktriangledown^\prime} \\
A \ar[r]_-{f} & B & & B & A \ar[r]_-{f} & B } \end{gathered}\end{equation}
and where composition and identity maps are as in $\mathbb{X}$. Note that the two right most diagrams above imply that $f$ is a monoid morphism. 

\begin{theorem}\label{isothm2} For a differential storage category $\mathbb{X}$, its category of $\oc$-differential exponential algebras $\mathsf{\oc DEA}[\mathbb{X}]$ is isomorphic to the category of differential exponential maps of the coKleisli category $\mathsf{DEM}[\mathbb{X}_\oc]$ via the inverse functors $\oc\mathsf{DEM}[\mathbb{X}] \xrightarrow{\mathsf{F}} \mathsf{DEM}[\mathbb{X}_\oc]$ and $\mathsf{DEM}[\mathbb{X}_\oc] \xrightarrow{\mathsf{F}^{-1}} \mathsf{\oc DEA}[\mathbb{X}]$ defined respectively as
\begin{align*} \mathsf{F}(A,\blacktriangledown, v, e) = (A, e) && \llbracket \mathsf{F}(f) \rrbracket = \varepsilon f \\
\mathsf{F}^{-1}(A, e)= (A, \blacktriangledown_e, v_e, e) && \mathsf{F}^{-1}(\llbracket g \rrbracket) = \eta \llbracket g \rrbracket
\end{align*}
\end{theorem}
\begin{proof} We first need to check that $\mathsf{F}$ and $\mathsf{F}^{-1}$ are well-defined. By Proposition \ref{propcok2}, $\mathsf{F}$ is well-defined on objects and so it remains to check that it is also well-defined on maps. First note that by \cite[Proposition 4.2.5]{blute2009cartesian}, a map in the coKleisli category $\mathbb{X}_\oc$ is linear if and only if it is of the form $\llbracket g \rrbracket = \varepsilon g^\prime$, and so every map in $\mathsf{DEM}[\mathbb{X}_\oc]$ is of this form. By definition $\llbracket \mathsf{F}(f) \rrbracket$ is linear and so it remains to show that $\llbracket e \mathsf{F}(f) \rrbracket = \llbracket \mathsf{F}(f) e^\prime \rrbracket$:
\begin{align*}
\llbracket e \mathsf{F}(f) \rrbracket &=~ \delta \oc(e) \llbracket \mathsf{F}(f) \rrbracket \\
&=~ \delta \oc(e) \varepsilon f \\
&=~ \delta \varepsilon e f \tag{Naturality of $\varepsilon$} \\
&=~ ef \tag{Comonad} \\
&=~ \oc(f) e^\prime \tag{\ref{!demmap}} \\
&=~ \delta \oc(\varepsilon) \oc(f) e^\prime \tag{Comonad} \\
&=~ \delta \oc(\llbracket \mathsf{F}(f) \rrbracket) e^\prime \\
&=~ \llbracket \mathsf{F}(f) e^\prime \rrbracket
\end{align*}
So $\mathsf{F}(f)$ is a map in $\mathsf{DEM}[\mathbb{X}_\oc]$ and therefore $\mathsf{F}$ is well-defined. On the other hand, by Proposition \ref{propcok1}, $\mathsf{F}^{-1}$ is well-defined on objects and so it again remains to check that it is also well-defined on maps. Note that since every map in $\mathsf{DEM}[\mathbb{X}_\oc]$ is of the form $\llbracket g \rrbracket = \varepsilon g^\prime$, it follows from {\bf [cd.3]} that we have that $\mathsf{F}^{-1}(\llbracket g \rrbracket)= \eta \varepsilon g^\prime = g^\prime$. Since $\llbracket g \rrbracket$ is a map in $\mathsf{DEM}[\mathbb{X}_\oc]$, by Theorem \ref{isothm}, we also have that the following equalities hold: 
\begin{align*}
 \llbracket e g \rrbracket = \llbracket g e^\prime \rrbracket && \llbracket \odot_e g \rrbracket = \llbracket (g \times g) \odot_{e^\prime} \rrbracket && \llbracket u_e g \rrbracket = \llbracket u_{e^\prime} \rrbracket
\end{align*}
Now since $\llbracket g \rrbracket = \varepsilon g^\prime$, the above identities can easily be simplified out to be: 
   \begin{equation}\label{gprime1}\begin{gathered} e g^\prime = \oc(g^\prime) e^\prime \quad \quad \quad \llbracket \odot_e \rrbracket g^\prime = \oc(g^\prime \times g^\prime) \llbracket \odot_{e^\prime} \rrbracket \quad \quad \quad \llbracket u_e \rrbracket g^\prime = \llbracket u_{e^\prime} \rrbracket
   \end{gathered}\end{equation}
Using these identities, we now show that $\mathsf{F}^{-1}(\llbracket g \rrbracket)$ satisfies (\ref{!demmap}): 
\begin{align*}
e \mathsf{F}^{-1}(\llbracket g \rrbracket) &=~ e g^\prime \\
&=~ \oc(g^\prime) e^\prime \tag{\ref{gprime1}} \\
&=~ \oc( \mathsf{F}^{-1}(\llbracket g \rrbracket) ) e^\prime \\ \\
\blacktriangledown \mathsf{F}^{-1}(\llbracket g \rrbracket) &=~ (\eta \otimes \eta)(\varepsilon \otimes \varepsilon) \blacktriangledown \mathsf{F}^{-1}(\llbracket g \rrbracket) \tag{{\bf [cd.3]}} \\
&=~ (\eta \otimes \eta)\chi^{-1} \chi (\varepsilon \otimes \varepsilon) \blacktriangledown \mathsf{F}^{-1}(\llbracket g \rrbracket) \\
&=~ (\eta \otimes \eta)\chi^{-1} \llbracket \odot_e \rrbracket \mathsf{F}^{-1}(\llbracket g \rrbracket) \tag{Proposition \ref{propcok1}} \\
&=~ (\eta \otimes \eta)\chi^{-1} \llbracket \odot_e \rrbracket g^\prime \\ 
&=~ (\eta \otimes \eta) \chi^{-1} \oc( g^\prime \times g^\prime) \llbracket \odot_{e^\prime} \rrbracket \tag{\ref{gprime1}} \\
&=~ (g^\prime \otimes g^\prime) (\eta \otimes \eta) \chi^{-1} \llbracket \odot_{e^\prime} \rrbracket \tag{Naturality of $\chi$ and $\eta$} \\
&=~ (g^\prime \otimes g^\prime) (\eta \otimes \eta) \chi^{-1} \chi (\varepsilon \otimes \varepsilon) \blacktriangledown^\prime \tag{Proposition \ref{propcok1}} \\
&=~ (g^\prime \otimes g^\prime) (\eta \otimes \eta) (\varepsilon \otimes \varepsilon) \blacktriangledown^\prime \\
&=~ (g^\prime \otimes g^\prime) \blacktriangledown^\prime \tag{{\bf [cd.3]}} \\
&=~ \left(\mathsf{F}^{-1}(\llbracket g \rrbracket) \otimes \mathsf{F}^{-1}(\llbracket g \rrbracket) \right) \blacktriangledown^\prime \\ \\
v \mathsf{F}^{-1}(\llbracket g \rrbracket) &=~ v g^\prime \\
&=~ \chi^{-1}_\top \chi_\top v g^\prime \\
&=~ \chi^{-1}_\top \llbracket u_e \rrbracket g^\prime \tag{Proposition \ref{propcok1}} \\
&=~ \chi^{-1}_\top \llbracket u_{e^\prime} \rrbracket \tag{\ref{gprime1}} \\
&=~ v^\prime \tag{Proposition \ref{propcok1}} 
\end{align*}
So $\mathsf{F}^{-1}(\llbracket g \rrbracket)$ is a map in $\mathsf{\oc DEA}[\mathbb{X}]$ and therefore $\mathsf{F}^{-1}$ is well-defined. We leave it to the reader to check for themselves that $\mathsf{F}$ and $\mathsf{F}^{-1}$ preserves both identities and composition, and are thus indeed functors. Lastly, we need to show that $\mathsf{F}$ and $\mathsf{F}^{-1}$ are inverses of each other. Clearly $\mathsf{F}\mathsf{F}^{-1}(A,e) = (A,e)$ and $\mathsf{F}\mathsf{F}^{-1}(\llbracket g \rrbracket) = \llbracket g \rrbracket$. In the other direction, we clearly have that $\mathsf{F}^{-1}\mathsf{F}(f) = f$ and so it remains to show that $(A, \blacktriangledown, v, e) = \mathsf{F}\mathsf{F}^{-1}(A, \blacktriangledown, v, e) = (A, \blacktriangledown_e, v_e, e)$, that is, we need to show that $\blacktriangledown = \blacktriangledown_e$ and $v = v_e$ -- both of which follow immediately from (\ref{!exp}): 
\begin{align*}
\blacktriangledown_e &=~ (\eta \otimes \eta) \nabla e \\
&=~ (\eta \otimes \eta) (e \otimes e) \blacktriangledown \tag{\ref{!exp}} \\
&=~ \blacktriangledown \tag{\ref{!exp}} \\ \\
v_e &=~ \nu e \\
&=~ v \tag{\ref{!exp}} 
\end{align*}
So $(A, \blacktriangledown, v, e) = (A, \blacktriangledown_e, v_e, e)$. Therefore, we conclude that $\mathsf{F}$ and $\mathsf{F}^{-1}$ are inverse functors and that $\mathsf{\oc DEA}[\mathbb{X}]$ is isomorphic to $\mathsf{DEM}[\mathbb{X}_\oc]$. 
 \end{proof}

We conclude this section by briefly studying $\oc$-differential exponential algebras in two examples of differential storage categories. 

\begin{example}\label{RELex} \normalfont Let $\mathsf{REL}$ be the category of sets and relations, that is, the category whose objects are sets $X$ and where a map $X \xrightarrow{R} Y$ is a subset $R \subseteq X \times Y$. $\mathsf{REL}$ is a differential storage category where for a set $X$:
\[\oc X = \lbrace X \xrightarrow{f} \mathbb{N} \vert~ \vert{supp(f)}\vert < \infty \rbrace\]
with $supp(f) = \lbrace x \in X \vert ~ f(x) \neq 0 \rbrace$, and where the codereliction $\eta \subseteq X \times \oc X$ is defined as follows:
\[ (x, f) \in \eta \Leftrightarrow f(y) = \begin{cases} 1 & \text{ if } y=x \\ 0 & \text{ if } x \neq y \end{cases} \]
For more details on this example, see \cite[Proposition 2.7]{blute2006differential}. It turns out that in fact $\oc$ is the free exponential modality \cite{mellies2017explicit} on $\mathsf{REL}$, that is, $\oc X$ is the cofree cocommutative comonoid over $X$ in $\mathsf{REL}$ and therefore $\oc$-coalgebras are precisely cocommutative comonoids in $\mathsf{REL}$. By self-duality of $\mathsf{REL}$, $\oc$ is also a monad such that $\oc X$ is the free commutative monoid over $X$ in $\mathsf{REL}$ and $\oc$-algebras are precisely commutative monoids in $\mathsf{REL}$. In fact, the codereliction $\eta$ is the unit of the monad structure of $\oc$. It turns out that the $\oc$-differential exponential algebras are precisely the commutative monoids in $\mathsf{REL}$ (or equivalently the $\oc$-algebras). Indeed, every $\oc$-differential exponential algebra $(A, \blacktriangledown, v, e)$ is by definition a commutative monoid $(A, \blacktriangledown, v)$ and its associated $\oc$-algebra structure is precisely $e \subseteq \oc A \times A$. Conversely, given a commutative monoid $(A, \blacktriangledown, v)$, its associated $\oc$-algebra structure $e \subseteq \oc A \times A$ satisfies by definition that $\eta e = 1$ and is also a monoid morphism, and therefore $(A, \blacktriangledown, v, e)$ is a $\oc$-differential exponential algebra. In particular, since for every $X$, $(\oc X, \nabla, \nu)$ is a commutative monoid, there is a natural transformation $\mu \subseteq \oc \oc X \times \oc X$ such that $(\oc X, \nabla, \nu, \mu)$ is a $\oc$-differential exponential algebra. Explicitly, $\mu$ is defined as follows:
\[ (F, f) \in \mu \Leftrightarrow \sum \limits_{g \in supp(F)} F(g)(x) = f(x) \]
Therefore, $\mu \subseteq \oc \oc X \times \oc X$ is a differential exponential map in the coKleisli category $\mathsf{REL}_\oc$. Also, it turns out that every $X$ comes equipped with a monoid structure in $\mathsf{REL}$ given by the dual of the copying relation, and so the induced $\oc$-algebra structure $e \subseteq \oc X \times X$ is given by: 
 \[ (f, x) \in e \Leftrightarrow f(y) = 0 \text{ for all } y \neq x \]
and so for every $X$, $e \subseteq \oc X \times X$ is a differential exponential map in the coKleisli category $\mathsf{REL}_\oc$. For more examples, monoids in $\mathsf{REL}$ are studied in \cite{jenvcova2017monoids}. 
\end{example}

\begin{example}\label{VECex} \normalfont Let $k$ be a field and let $\mathsf{VEC}_k$ be the category of $k$-vector spaces and $k$-linear maps between them. For a $k$-vector space $V$, let $\mathsf{Sym}(V)$ be the symmetric algebra over $V$ \cite[Section 8, Chapter XVI]{lang2002algebra}, that is, the free commutative $k$-algebra over $V$. In particular, if $X$ is a basis set for $V$, then $\mathsf{Sym}(V) \cong k[X]$ where $k[X]$ is the polynomial ring over the set $X$. This induces a monad $\mathsf{Sym}$ on $\mathsf{VEC}_k$ such that the $\mathsf{Sym}$-algebras are precisely the commutative $k$-algebras. Now suppose that $k$ has characteristic $0$, then $\mathsf{VEC}_k$ is a differential storage category where: 
\[\oc V = \bigoplus \limits_{v \in V} \mathsf{Sym}(V)\]
and where the codereliction $V \xrightarrow{\eta} \oc V$ is defined as the injection of $V$ into the $0 \in V$ component of $\oc V$. For full details on this example, see \cite{clift2017cofree}. Similarly to the previous example, $\oc$ is the free exponential modality on \cite{mellies2017explicit} on $\mathsf{VEC}_k$, that is, $\oc V$ is the cofree cocommutative $k$-coalgebra over $V$ \cite{murfet2015sweedler} and therefore $\oc$-coalgebras are precisely cocommutative $k$-coalgebras. It turns out that, once again, $\oc$-differential exponential algebras correspond precisely to commutative monoids in $\mathsf{VEC}_k$ which are precisely the commutative $k$-algebras or equivalently the $\mathsf{Sym}$-algebras. By definition, every $\oc$-differential exponential algebra $(A, \blacktriangledown, v, e)$ is a commutative $k$-algebra and it turns out that its $\mathsf{Sym}$-algebra structure is given by pre-composing $\oc A \xrightarrow{e} A$ with the $0 \in V$ injection map $\mathsf{Sym}(A) \xrightarrow{\mathsf{i}_0} \oc A$. Conversely, given a commutative $k$-algebra $(A, \blacktriangledown, v)$, let $\mathsf{Sym}(A) \xrightarrow{\omega} A$ be its induced $\mathsf{Sym}$-algebra structure, and define $\oc A \xrightarrow{e^\omega} A$ as the unique map which makes the following diagram commute for all injection maps $\mathsf{Sym}(A) \xrightarrow{\mathsf{i}_a} \oc A$, $a \in A$: 
\[\xymatrixcolsep{5pc}\xymatrix{\mathsf{Sym}(A) \ar[r]^-{\mathsf{i}_a} \ar[dr]_-{\omega} & \oc A \ar[d]^-{e^\omega} \\ & A }\]
then it follows that $(A, \blacktriangledown, v, e^\omega)$ is a $\oc$-differential exponential algebra. In particular for every $V$, $(\oc V, \nabla, \nu)$ is a commutative $k$-algebra. As such, there is a natural transformation $\oc \oc V \xrightarrow{\mu} \oc V$ such that $(\oc X, \nabla, \nu, \mu)$ is a $\oc$-differential exponential algebra, and thus $\mu$ is also a differential exponential map in the coKleisli category. 
 \end{example}

\section{Conclusion and Future Work}\label{consec}

As the exponential function $e^x$ (and its generalizations) is so prominent and important throughout various fields and has numerous applications, this paper opens the door to numerous possibilities and applications for differential exponential maps. In particular, as the theory of differential equations in Cartesian differential categories develops, differential exponential maps should be a key component for this theory in the same way that the exponential function is a fundamental tool in solving classical differential equations. For such applications of differential exponential maps, see \cite{cockett2019differential}. 

It is also of importance and of interest to find and study more examples of differential exponential maps in a variety of Cartesian differential categories. For example, one should consider studying differential exponential maps in cofree Cartesian differential categories \cite{cockett2011faa,lemay2018tangent} and abelian functor calculus \cite{bauer2018directional}, as well as study $\oc$-differential exponential algebras in the differential storage categories of convenient vector spaces \cite{blute2010convenient} and finiteness spaces \cite{ehrhard2017introduction}. Another possible source of examples is to construct differential exponential maps in the presence of infinite sums, which many categorical models of the differential $\lambda$-calculus \cite{ehrhard2003differential,manzonetto2012categorical} have.

There are also certain interesting potential generalizations of differential exponential maps to consider. For example, the exponential function $e^x$ can also be defined as the inverse of the natural logarithm function $ln(x)$. However the natural logarithm function is only a partial function of type $\mathbb{R} \to \mathbb{R}$, since $ln(x)$ is not defined at $x=0$. As such, one must instead work in a differential restriction category \cite{journal:diff-rest}, which allows one to consider partial functions and domains of definition. One could then generalize the natural logarithm function in a differential restriction category in such a way that differential exponential maps arise as their restriction inverse. On the other hand, one could also generalize differential exponential maps to tangent categories \cite{cockett2014differential} and differential bundles \cite{cockett2016differential}, such that this notion should be a generalization of exponential maps for manifolds and Lie groups. That differential exponential maps can be axiomatized using only the basic structure of a Cartesian differential category (that is, without referencing extra requirements such as differential rig structure), will be particularly useful when generalizing exponential functions to tangent categories.  

Regarding $\oc$-differential exponential algebras, it is interesting to point out that in both examples of differential storage categories studied in this paper, there was a natural transformation $\oc \oc A \xrightarrow{\mu} \oc A$ which endows $\oc A$ with a $\oc$-differential exponential algebra structure. As such, it would be interesting to study differential storage categories with such a $\mu$ and understand what are the consequences from a differential linear logic \cite{ehrhard2017introduction} point of view. A natural question to ask is when does the codereliction $A \xrightarrow{\eta} \oc A$ and $\mu$ provide a monad structure on $\oc$ (with one of the monad identities already being a requirement for a $\oc$-differential exponential algebra), and conversely when does a monad structure on $\oc$ induce a natural $\oc$-differential exponential algebra structure. 

Lastly, another possible direction would be to generalize the trigonometric functions and the hyperbolic functions in the same way for arbitrary Cartesian differential categories. Indeed, generalizations of the (hyperbolic) sine and cosine functions would be a pair of endomorphisms whose axioms are based on the fact that $\mathsf{D}[\sin(x)](x,y) = \cos(x)y$ and $\mathsf{D}[\cos(x)](x,y) = -\sin(x)y$ (resp. $\mathsf{D}[\sinh(x)](x,y) = \cosh(x)y$ and $\mathsf{D}[\cosh(x)](x,y) = \sinh(x)y$), as well as other algebraic properties of $\sin$ and $\cos$ (resp. $\sinh$ and $\cosh$). It should also be expected that these generalized trigonometric or hyperbolic functions would also be in bijective correspondence with special sorts of differential rigs. Furthermore, since the (split) complex exponential function is constructed using $e^x$, $\cos(x)$ and $\sin(x)$ (resp. $e^x$, $\cosh(x)$, and $\sinh(x)$), combining a differential exponential map with these generalized trigonometric (or hyperbolic) functions should again result in a differential exponential map in a similar fashion. In fact, it would be desirable to find an overarching general notion that would encompass all of these varying concepts. 

 In conclusion, there are many potential interesting paths to take for future work with differential exponential maps.  

\bibliographystyle{spmpsci}      
\bibliography{expbib}   
\end{document}